\documentclass[11pt]{article}
\usepackage[a4paper,margin=1in]{geometry}
\usepackage[utf8]{inputenc}
\usepackage[T1]{fontenc}
\usepackage{amsmath,amssymb,amsthm}
\usepackage{graphicx}
\usepackage{booktabs}
\usepackage{enumitem}
\usepackage{amsthm}
\usepackage[section]{placeins}   
\usepackage{float}
\usepackage[colorlinks=true,linkcolor=blue,citecolor=blue,urlcolor=blue]{hyperref}
\hypersetup{pdftitle={The crossing number and the unit-distance crossing number of the Hamming graphs H(d,3), and their realizations over many coordinate fields},pdfauthor={Haroldo Costa Silva Filho and Edward Pegg Jr.}}

\graphicspath{{figs/}}
\newcommand{\Q}{\mathbb{Q}}
\newcommand{\R}{\mathbb{R}}
\newcommand{\bx}{\mathbin{\square}}
\newcommand{\cro}{\operatorname{cr}}\newcommand{\udcr}{\operatorname{udcr}}
\theoremstyle{plain}\newtheorem{theorem}{Theorem}\newtheorem{proposition}{Proposition}
\newtheorem{lemma}{Lemma}
\newtheorem{corollary}{Corollary}
\newtheorem{conjecture}{Conjecture}
\theoremstyle{definition}\newtheorem{remark}{Remark}

\title{The crossing number and the unit-distance crossing number of the Hamming graphs
$H(d,3)=K_3^{\bx d}$, and their realizations over many coordinate fields}
\author{
  Haroldo Costa Silva Filho\thanks{\texttt{haroldosfilho@gmail.com}}\\
  \small PGCComp (Ci\^encias Computacionais e Modelagem Matem\'atica),\\
  \small Universidade do Estado do Rio de Janeiro (UERJ), Rio de Janeiro, Brazil
  \and
  Edward Pegg Jr.\thanks{\texttt{edp@wolfram.com}}\\
  \small Wolfram Research, Champaign, Illinois, USA
}
\date{November 2026}

\begin{document}
\maketitle

\begin{abstract}
The Hamming graph $H(d,3)=K_3^{\bx d}$ ($n=3^d$ vertices) is a purely combinatorial object:
the graph of single-symbol errors of ternary codes. We ask a geometric question of it: can it
be drawn in the plane with \emph{every} edge exactly one unit long and no two non-adjacent
vertices a unit apart? It can, for every $d$; and this single constraint leads to two
problems that turn out to be linked.

The first is arithmetic: over which coordinate fields is the graph realizable? Fixing the edges to
unit length might be expected to single out one such field, as for a rigid graph; it does not.
Realized as a Minkowski sum of unit triangles (a construction that keeps the whole family
planar only for base $3$), each $H(d,3)$ admits not a single faithful realization but a
$(d-1)$-parameter continuum of them. Moving along this flex raises the coordinate field through the
classical constructibility hierarchy (ruler-and-compass, origami, and beyond), so one graph
is realizable over many number fields. The coordinate field varies along the flex and is a number
field only at countably many angles: the graph has no canonical coordinate field.

The second is combinatorial. Alongside the ordinary crossing number $\cro$ we introduce the
\emph{unit-distance crossing number} $\udcr(G)$, the least number of crossings over all faithful unit drawings, those in which every edge is a unit segment; it is the rectilinear crossing
number taken under the unit-length constraint, so $\cro\le\operatorname{rcr}\le\udcr$. A single
recursive method (the graph as three shrunken copies of itself joined by a connecting bundle) controls both. For the ordinary number we obtain the exact order $\cro(H(d,3))=\Theta(n^{2})$,
the upper bound from a one-page convex layout and the matching $\Omega(n^{2})$ from the spectral
bisection width $\mathrm{bw}\ge\tfrac34 n$ (the Laplacian gap $\lambda_2=3$) through the
bisection--crossing inequality. For the unit number we prove $\udcr(H(d,3))=O(n^{2}\log n)$ and give explicit upper bounds from
reasonably symmetric drawings, $\udcr(H(3,3))\le231$ among them (conjecturally an equality); a structure theorem shows that every faithful planar realization is
a Minkowski sum, so the realization variety is an explicit $(d-1)$-parameter flex and the $\udcr$
optimum is attained inside it, with one sharp gap, the matching $\Omega(n^{2}\log n)$ lower
bound, left as the central open problem. The two invariants differ: for $\cro$ the crossing-minimal drawings are structurally opaque, whereas the $\udcr$ optimum is
itself a faithful Minkowski realization; under the unit constraint the least-crossing drawing
and a faithful drawing coincide.

The bridge is the flex itself: the same freedom that lets the graph range over coordinate fields is
the freedom over which crossings are minimized, and for $d=3$ the least-crossing realizations we find lie near a
\emph{constructible} (origami) point, so ``which field'' and ``how few
crossings'' are two readings of one geometry. The account is self-contained, and every explicit drawing and count is
verified in exact arithmetic with reproducible code and figures; the structural and algebraic
results (faithfulness for every $d$, the three coordinate fields, $\operatorname{edim}(H(d,q))=q-1$,
the $(d-1)$-flex for \emph{every} Minkowski realization, the exact one-page count, and the explicit
unit drawings giving $\udcr\le9$ and $\udcr(H(3,3))\le231$) are moreover formally verified in Lean~4 (\url{https://github.com/haroldosfilho/HammingUDG}).
\end{abstract}



\section{The graph}
The Hamming graph $H(d,q)$ has vertex set $\{1,\dots,q\}^d$, two words being
adjacent when they differ in exactly one coordinate (Hamming distance $1$);
equivalently $H(d,q)=K_q^{\,\square d}$. For $d=q=3$,
\[
  H(3,3)=K_3\bx K_3\bx K_3,\qquad |V|=27,\quad \text{$6$-regular},\quad |E|=81 .
\]
It is the one-symbol-error graph of length-$3$ ternary codes. Its standard
straight-line drawing (Figure~\ref{fig:planar}) is correct as a graph but has
edges of many different lengths; the question we address is whether $H(3,3)$ can be drawn as a
\emph{faithful} unit-distance graph (UDG): an \emph{injective} map $p:V\to\R^2$ with
$\|p(u)-p(v)\|=1$ if and only if $uv$ is an edge (that is, all edges of length $1$, all
vertices distinct, and no non-adjacent pair at distance $1$), a class whose extremal questions
remain very active~\cite{Alon2026}. The arithmetic behind that recent activity
(complex-multiplication fields and Minkowski embeddings in the disproof of the Erd\H{o}s
unit-distance conjecture~\cite{Alon2026}) is
the same circle of ideas as the coordinate-field analysis we carry out below, only here for a
\emph{fixed} graph rather than for an extremal point set.

\begin{figure}[t]\centering
  \includegraphics[width=.9\linewidth]{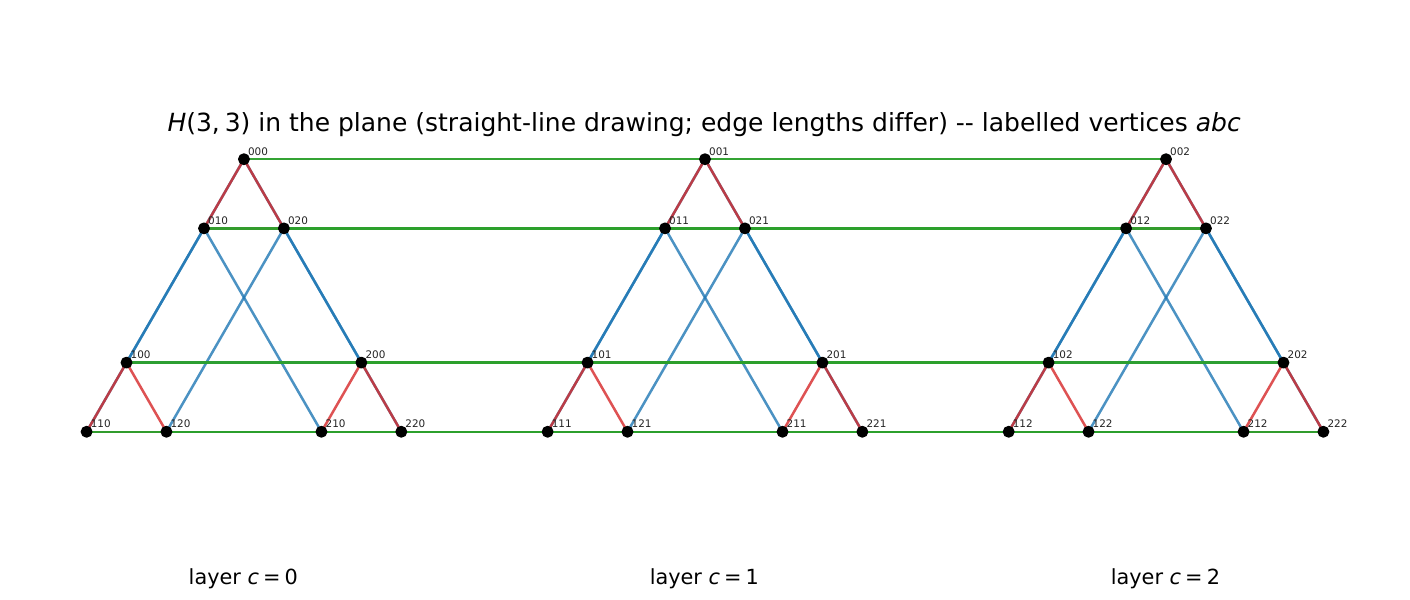}
  \caption{A straight-line drawing of $H(3,3)$ in the plane; colors indicate
  which coordinate changes. Edge lengths differ: this is not yet a unit-distance
  realization.}\label{fig:planar}
\end{figure}

\paragraph{Origin of the question.} This work grew out of a question raised by the second
author: are the new unit-distance constructions of $H(3,3)$ rigid, and does the answer depend
on the coordinate field? The answer is the thread followed here: the unit realizations are
flexible, and the field they live in varies along that flex.

\paragraph{Overview.} The paper has two halves joined by one construction. The first
(Sections~\ref{sec:mink}--\ref{sec:gallery}) is about \emph{where} $H(d,3)$ lives: realized faithfully in the plane as a Minkowski
sum of $d$ unit equilateral triangles, it carries a $(d-1)$-parameter \emph{flex}, and this flex is
the \emph{entire} planar realization variety (Theorem~\ref{thm:allmink}: in the plane every faithful
unit realization is such a Minkowski sum; the proof is finite and computer-free,
Appendix~\ref{app:uniform}). Sliding along the flex walks the coordinates up the
constructibility ladder (ruler-and-compass, origami, exotic), and the
coordinate field is a number field only at countably many angles. The
second half (Section~\ref{sec:crossing}) is about \emph{how crowded} a unit drawing must be. Two invariants are
compared, the ordinary crossing number $\cro$ and the \emph{unit-distance} crossing number $\udcr$,
and they are linked three ways: (i) by the chain $\cro\le\mathrm{rcr}\le\udcr$, since a unit drawing
is rectilinear; (ii) by \emph{one} recurrence $c(n)=3\,c(n/3)+g(n)$ (the product decomposition
$H(d,3)=H(d-1,3)\bx K_3$), which drives \emph{both}: with a one-page layout it gives
$\cro(H)=O(n^2)$, and with the Minkowski drawing it gives $\udcr(H)=O(n^2\log n)$; (iii) by the
spectral gap $\lambda_2=3$, which through the bisection--crossing inequality pins the ordinary count
exactly, $\cro(H)=\Theta(n^2)$. The unit constraint provably forces strictly more crowding, the gap
conjecturally $\Theta(\log n)$, with the matching lower bound $\udcr=\Omega(n^2\log n)$ the central
open problem. The bridge between the halves is the flex itself: the freedom that roams coordinate
fields is the freedom over which $\udcr$ is minimized, and the least-crossing faithful drawing is
found near a \emph{constructible} point (the origami drawing, giving $\udcr(H(3,3))\le231$).

\section{Hamming graphs, codes, and isomorphisms}
\emph{Definitions, tied together.} Fix an alphabet $\{0,\dots,q-1\}$ and word length $d$.
The \emph{Hamming distance} $d_H(x,y)$ between two words is the number of positions in which
they differ, and the \emph{Hamming weight} $\mathrm{wt}(x)=d_H(x,0)$ counts nonzero symbols.
The \emph{Hamming graph} $H(d,q)$ has the $q^{d}$ words as vertices and joins two words iff
$d_H=1$; thus graph distance equals Hamming distance, and a single-symbol error is exactly a
step along an edge. A \emph{code} $C\subseteq\{0,\dots,q-1\}^{d}$ with minimum distance
$\delta$ is an independent set of the $(\delta-1)$-th graph power, and it corrects
$\lfloor(\delta-1)/2\rfloor$ errors; so packings and independent sets in $H(d,q)$ \emph{are}
error-correcting codes, and the Hamming (sphere-packing) bound is a ball-packing statement in
this metric. This is the single thread linking every notion in the paper: the alphabet gives
$q$, hence $\operatorname{edim}=q-1$ (Theorem~\ref{thm:edim}); the word length gives $d$,
hence the flexibility dimension $d-1$ (Theorem~\ref{prop:flex}); and drawing the code space in
its Euclidean dimension is what makes the two crossing numbers meaningful.

The Hamming graph is a standard object in coding theory and in the theory of association
schemes (the Hamming scheme $H(d,q)$), and the family is central across computing,
communication, and cryptography. In \emph{communication}, $H(d,q)$ is the state space of
$q$-ary block codes: minimum distance, decoding radius, and the Singleton and sphere-packing
bounds are all read off its metric, and channel errors are walks on it. In \emph{computing},
the graphs $H(d,2)=Q_d$ and their $q$-ary analogues (the clique-based ``$K$-cubes'' of
LaForge~\cite{LaForge}) are classical interconnection-network and VLSI-layout topologies that
minimize connection cost and packet latency and admit optimal local routing, and where the
bisection width (used in Theorem~\ref{thm:tight}) governs communication cost. In \emph{cryptography}, Hamming weight and distance measure the
nonlinearity and correlation of Boolean functions and S-boxes: differential and linear
cryptanalysis, the covering radius, and side-channel (Hamming-weight) leakage models all live
on $H(d,q)$. Placing this code space in its true Euclidean dimension, and measuring how
crowded that placement must be, is the geometric counterpart of these questions.

The small members carry a web of alternative identities, which is part of why
labelling the graph correctly matters:
\begin{itemize}\itemsep2pt
\item $H(1,q)=K_q$ (a single symbol; the complete graph).
\item $H(d,2)=Q_d$, the $d$-cube; $H(2,2)=C_4$, $H(3,2)=Q_3$ the cube graph.
\item $H(2,q)=K_q\bx K_q$ is the $q\times q$ \emph{rook's graph} (rook moves on a
chessboard); it is a strongly regular \emph{Latin-square graph}, and for $q=3$ it is
moreover the \emph{Paley graph} of order $9$ and the $3\times3$ rook / $K_{3,3}$-line
graph. Its complement, spectrum $\{4,1^4,(-2)^4\}$, and $S_3\wr S_2$ automorphisms are
all classical.
\item $H(3,3)=K_3^{\bx3}$ is the \emph{enneagon-twist} graph (three unit star-enneagons
$\{9/1\},\{9/2\},\{9/4\}$ with a $20^\circ$ twist realize it), with spectrum
$\{6,3^6,0^{12},(-3)^8\}$ and $27$ triangles.
\end{itemize}
All of these are unit-distance graphs. The construction we use throughout is the
Minkowski sum of the factor realizations: it sends every edge to a translate of a
factor edge, hence of length exactly one. We stress the scope of this: the Minkowski
sum gives the \emph{edge lengths} for free, but not \emph{faithfulness}: injectivity
and the absence of unit distances on non-edges are a separate condition, which for the
unit triangles ($K_3$ factors) in generic position we verify directly
(Proposition~\ref{prop:mink} and Section~\ref{sec:faithful}). We do \emph{not} claim a
faithful planar realization for arbitrary unit-distance factors; the statement we need
and prove is the one for $K_3$ factors. For base $q=3$ the triangle sum stays in the
plane, and there its two crossing numbers, its flexibility, and its arithmetic become
computable.

\section{The Minkowski-sum realization}\label{sec:mink}
Let $u,v,w$ be three unit equilateral triangles (circumradius $R=1/\sqrt3$) with
vertices labelled $0,1,2$ (Figure~\ref{fig:tri}). Place vertex $abc$ at the
Minkowski sum
\begin{equation}\label{eq:mink}
  P(a,b,c)=u_a+v_b+w_c
\end{equation}
(Figure~\ref{fig:construction}).

\begin{proposition}\label{prop:mink}
For any orientations of $u,v,w$, the map \eqref{eq:mink} sends every edge of
$H(3,3)$ to a segment of length $1$.
\end{proposition}
\begin{proof}
An edge changes exactly one coordinate, say $a\to a'$. Then
$P(a',b,c)-P(a,b,c)=u_{a'}-u_a$, a side of the unit triangle $u$, of length $1$;
the cases of $b$ and $c$ are identical.
\end{proof}

\begin{remark}[The unit edges are classical; faithfulness is not]
That the map~\eqref{eq:mink} keeps every edge unit is the classical fact that the
Cartesian product of unit-distance graphs is again a unit-distance graph, realized
by the Minkowski sum of the factor drawings~\cite{EHT1965,HorvatPisanski}. What the
construction does \emph{not} give for free is \emph{faithfulness}: the Minkowski sum
can carry \emph{more} unit distances than the product (already for two $2$-vertex
paths the sum realizes five unit pairs against the four edges of their product
$C_4$~\cite{HorvatPisanski}), so the absence of ghost unit distances on non-adjacent
pairs is a separate condition. Establishing it for the $K_3$ factors, and mapping the
flex over which it holds, is where the work of this paper lies
(Section~\ref{sec:faithful}).
\end{remark}

\begin{figure}[t]\centering
  \includegraphics[width=.9\linewidth]{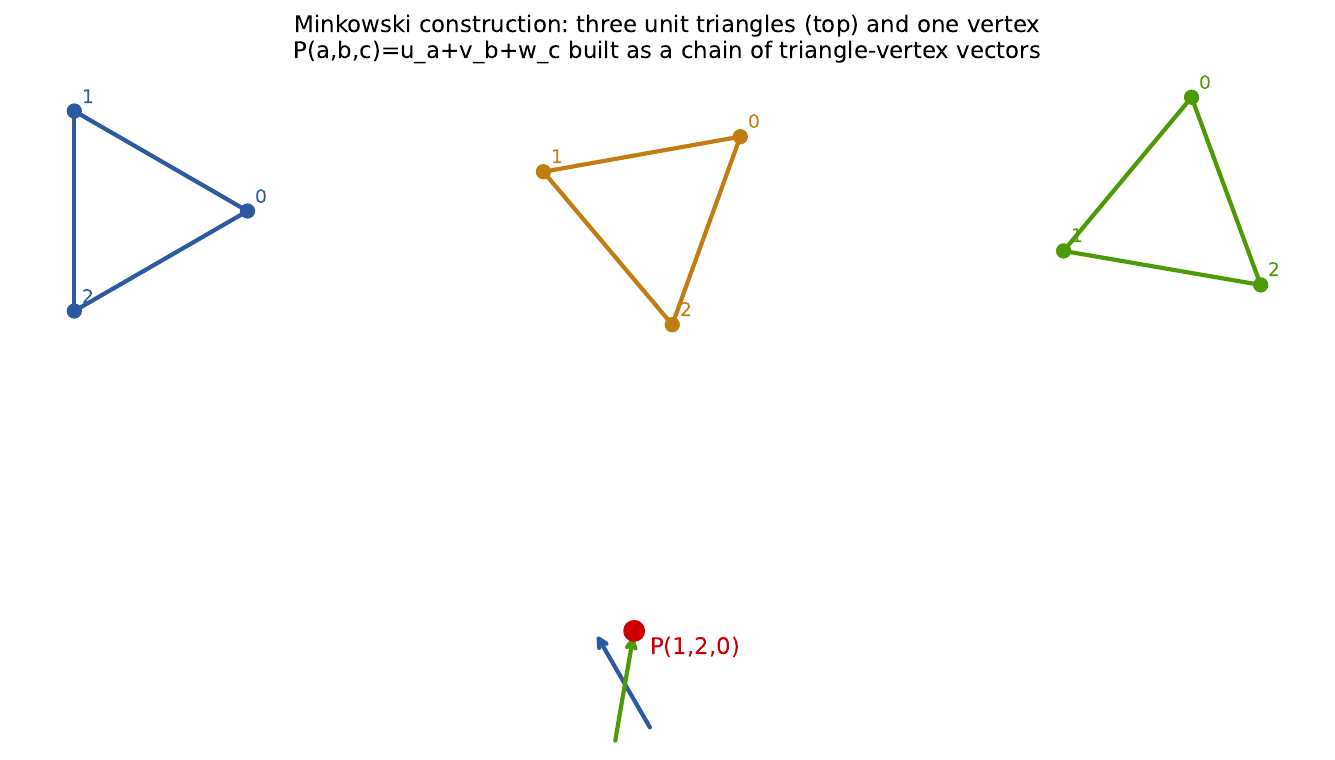}
  \caption{The Minkowski construction. Top: three unit triangles $u,v,w$ with vertices
  labelled $0,1,2$. Bottom: the vertex $P(1,2,0)=u_1+v_2+w_0$ assembled as a chain of
  triangle-vertex vectors from the origin; an edge changes one coordinate and so is a
  single triangle side, of length one.}\label{fig:construction}
\end{figure}

Thus all $81$ edges are unit \emph{at once}; what remains is faithfulness. Since
the edge lengths depend only on the \emph{shapes} of the triangles, rotating $v$
by an angle $\theta$ and $w$ by another angle (keeping $u$ fixed) preserves all
$81$ unit lengths: $H(3,3)$ admits a genuine two-parameter \emph{flex}
(Figure~\ref{fig:flex}). Its realization variety, the zero set of the
polynomials $\|P(u)-P(v)\|^2-1$, is therefore positive-dimensional. This is
the structural reason the coordinate field is not pinned: a rigid UDG has an
isolated realization and a fixed field, whereas a flexible UDG can be pushed
through different algebraic extensions.

\begin{figure}[t]\centering
  \includegraphics[width=.8\linewidth]{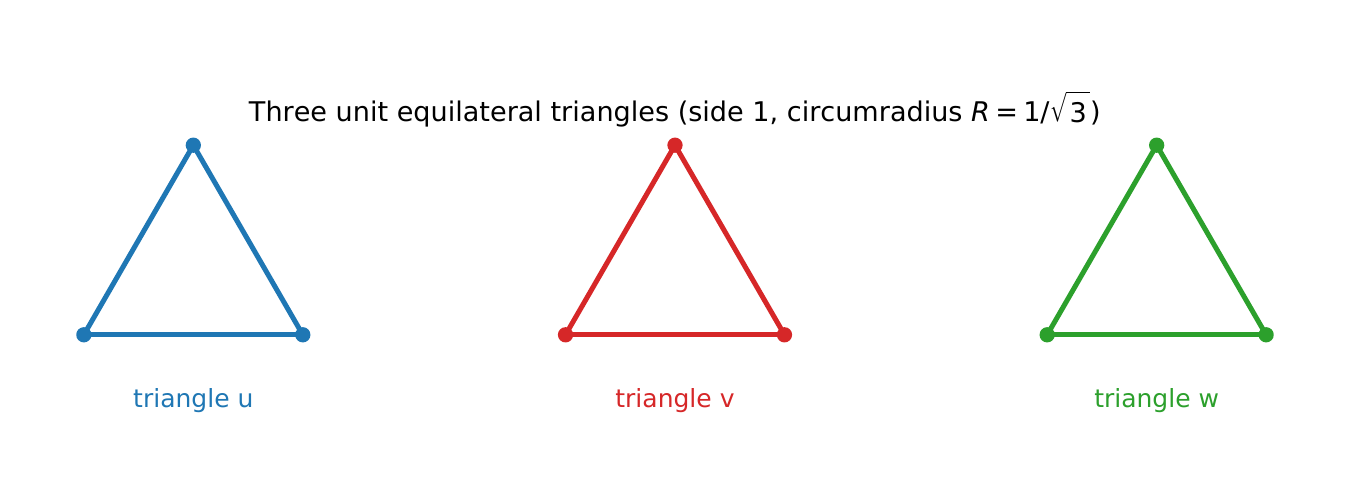}\\[4pt]
  \includegraphics[width=.42\linewidth]{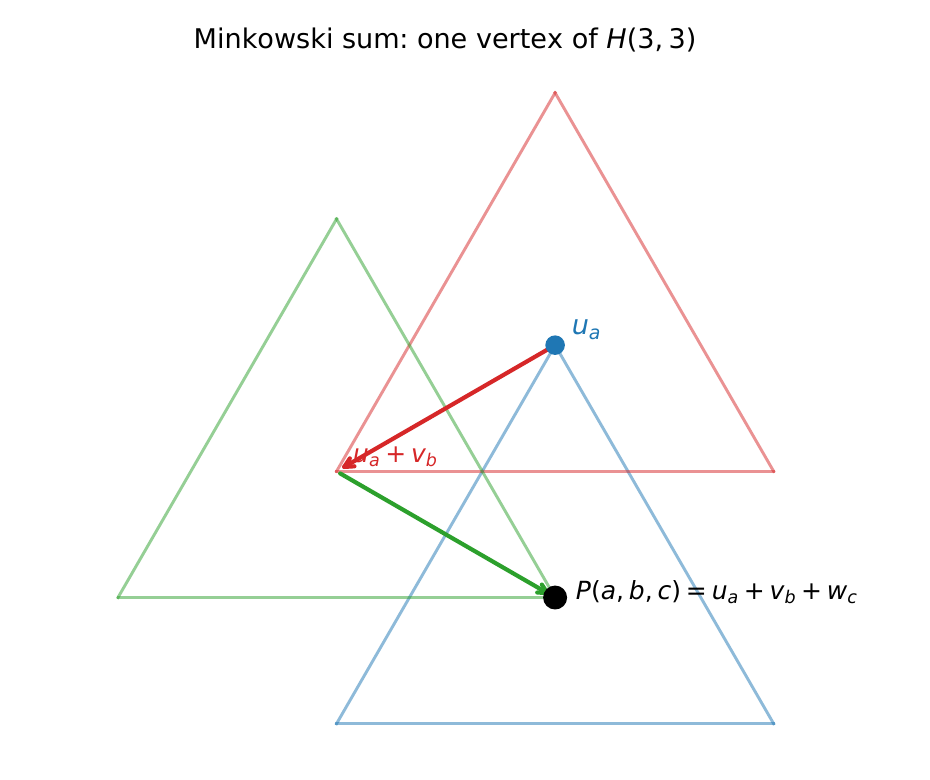}\hfill
  \includegraphics[width=.52\linewidth]{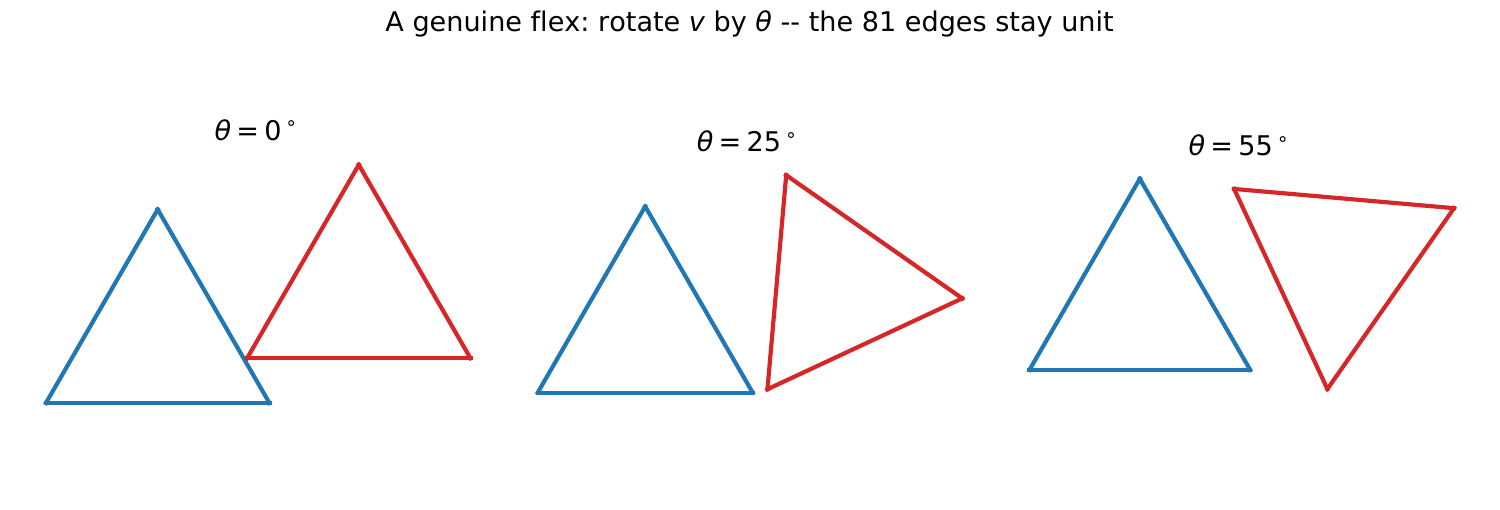}
  \caption{Top: the three unit triangles. Bottom left: one vertex
  $P(a,b,c)=u_a+v_b+w_c$. Bottom right: rotating $v$ gives a flex; the 81
  edges stay unit throughout.}\label{fig:tri}\label{fig:flex}
\end{figure}

\section{Coordinate fields and the constructibility hierarchy}\label{sec:fields}
Each edge imposes a polynomial equation
$(x_u-x_v)^2+(y_u-y_v)^2=1$ with rational coefficients, so the whole realization
system is defined over $\Q$. This does \emph{not} force the coordinates to be
algebraic: the system is under-determined (the flex makes its solution set
positive-dimensional) and at a generic angle the
coordinates are not algebraic (Remark~\ref{rem:transc}). When a realization \emph{is}
algebraic, however (this happens on a countable dense set of angles, Section~\ref{sec:faithful}),
its coordinates lie in a finite extension $K/\Q$, and the degree $[K:\Q]$ measures the
arithmetic difficulty. It is these algebraic realizations that the classical hierarchy
classifies:
\[
  \Q\ \subset\ \underbrace{\text{ruler\&compass}}_{[K:\Q]=2^a}\ \subset\
  \underbrace{\text{origami}}_{2^a3^b}\ \subset\
  \underbrace{\text{exotic}}_{\text{a prime }\ge5\text{ divides }[K:\Q]} .
\]
By Proposition~\ref{prop:mink} the whole arithmetic is concentrated in a single
\emph{base polynomial}: the minimal polynomial of $\cos\alpha$ at the base angle
$\alpha$. We now exhibit one realization in each tier (all faithfulness and
ghost counts verified in exact arithmetic), and one non-faithful degenerate case
for contrast.

\paragraph{$\Q(\sqrt3)$, base angles $0^\circ,30^\circ,90^\circ$ (not faithful).}
Base angles $0^\circ,30^\circ,90^\circ$ give coordinates in $\Q(\sqrt3)$
(degree $2$); the $81$ edges are unit, but the vertices collide ($21$ distinct
points) and $36$ non-adjacent pairs fall at distance $1$. A cheap field, but not
a faithful UDG. The same happens for every choice of base angles that are multiples of
$30^\circ$ (all $78$ choices checked): none of these immersions is faithful.

\paragraph{New faithful drawings over $\Q(\sqrt3)$.}
Faithful drawings with coordinates in $\Q(\sqrt3)$ do exist, with base angles that are not
multiples of $30^\circ$ but have rational cosine and sine (Pythagorean triples). With the unit
triangle $\{0,1,e^{i\pi/3}\}$ rotated by the angles with $(\cos,\sin)=(1,0)$,
$(\tfrac{24}{25},\tfrac{7}{25})$ and $(\tfrac45,\tfrac35)$ (about $0^\circ$, $16.26^\circ$,
$36.87^\circ$), all $81$ edges are unit and no non-adjacent pair is at distance $0$ or $1$; for
instance the vertex $(1,2,1)$ sits at
$\bigl(\tfrac{57}{25}-\tfrac{7\sqrt3}{50},\ \tfrac{37}{50}+\tfrac{12\sqrt3}{25}\bigr)$. This drawing is
faithful, is a good drawing and has exactly $231$ crossings, verified in exact arithmetic and in
Lean (paragraph \emph{Formal verification}; Figure~\ref{fig:pythagorean}(b)). The same idea in dimension $2$, with the single rotation
$(\cos,\sin)=(\tfrac45,\tfrac35)$, gives a good drawing of $H(2,3)=K_3\bx K_3$ over $\Q(\sqrt3)$ with
$9$ crossings, the exact value of $\udcr(K_3\bx K_3)$ (Figure~\ref{fig:pythagorean}(a)). Since no equilateral triangle has all its vertices in
$\Q^2$, degree $2$ is the least possible. So the compass tier of Theorem~1 ranges over degrees up
to $4$: from $\Q(\sqrt3)$ (degree $2$, Pythagorean angles) to $\Q(\sqrt2,\sqrt3)$ (degree $4$, the
next example).

\paragraph{Good drawings over $\Q(\sqrt2,\sqrt3)$.}
Base angles that are multiples of $15^\circ$ (coordinates in $\Q(\sqrt2,\sqrt3)$, degree $4$) give
$108$ faithful immersions of $H(3,3)$, with $231$ or $243$ crossings in nearby good drawings; but
each of them has exactly $6$ vertices in the interior of non-incident edges, so none is a good
drawing at the exact point. Composing them with Pythagorean rotations removes these incidences
while keeping the field: with base angles $0$, $15^\circ+\varphi$ and $105^\circ+\psi$, where
$\cos\varphi=\tfrac{24}{25}$ and $\cos\psi=\tfrac45$ (about $0^\circ$, $31.26^\circ$, $141.87^\circ$),
the drawing is faithful, is a good drawing and has exactly $231$ crossings; its coordinates lie
in $\Q(\sqrt2,\sqrt3)$ and genuinely involve $\sqrt2$ (for instance the vertex $(0,1,0)$ sits at
$\bigl(\tfrac{17\sqrt6+31\sqrt2}{100},\ \tfrac{31\sqrt6-17\sqrt2}{100}\bigr)$). So the bound
$\udcr(H(3,3))\le231$ is attained by exact good drawings at both compass degrees, $2$ and $4$.
More generally, since rational points are dense on the unit circle and the non-faithful twists
form a finite union of curves, Pythagorean twists give faithful realizations of $H(d,3)$ over
$\Q(\sqrt3)$ for every $d$.

\begin{figure}[t]\centering
  \includegraphics[width=\linewidth]{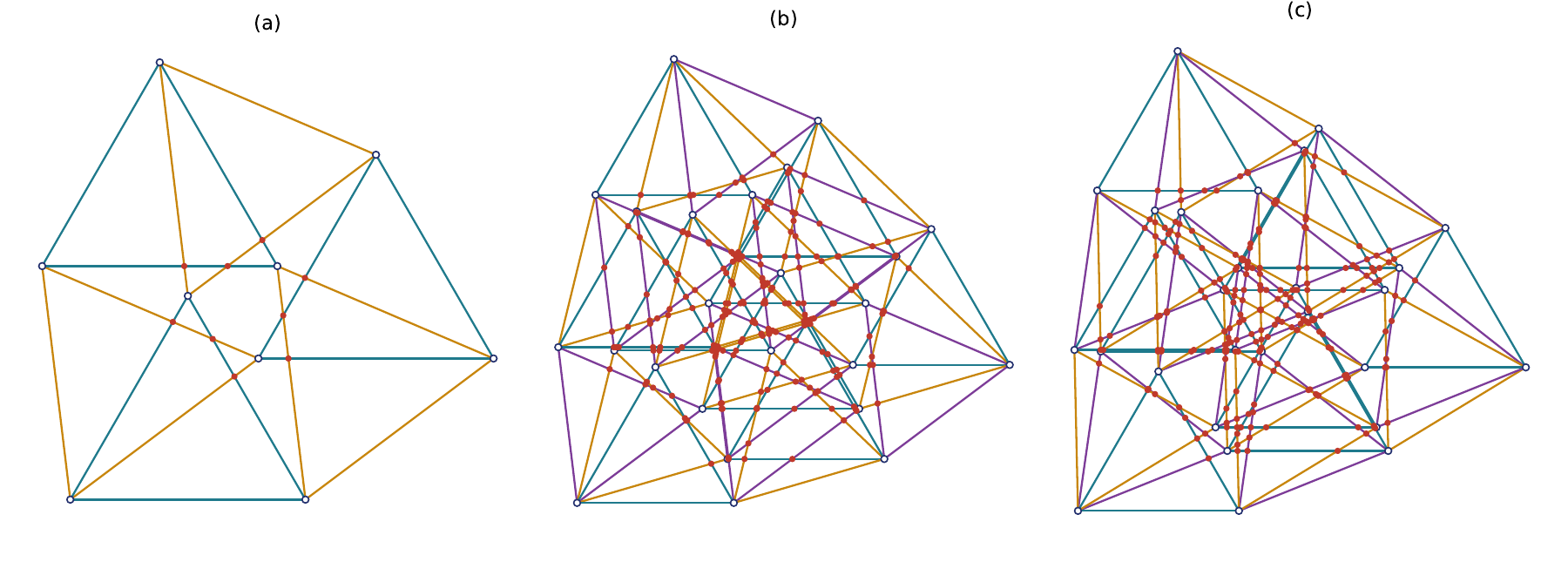}
  \caption{Exact good drawings with Pythagorean rotations; red dots are the crossings, and the edge
  colors give the coordinate that changes. (a)~$H(2,3)$ with base angles $0$ and $\alpha$,
  $(\cos\alpha,\sin\alpha)=(\tfrac45,\tfrac35)$: coordinates in $\Q(\sqrt3)$, $9$ crossings.
  (b)~$H(3,3)$ with $(\cos,\sin)=(1,0),(\tfrac{24}{25},\tfrac{7}{25}),(\tfrac45,\tfrac35)$: coordinates
  in $\Q(\sqrt3)$, $231$ crossings. (c)~$H(3,3)$ with base angles $0$, $15^\circ+\varphi$,
  $105^\circ+\psi$ ($\cos\varphi=\tfrac{24}{25}$, $\cos\psi=\tfrac45$): coordinates in
  $\Q(\sqrt2,\sqrt3)$, $231$ crossings. All three are faithful and have no vertex on a non-incident
  edge; (a) and (b) are the drawings verified in Lean.}
  \label{fig:pythagorean}
\end{figure}

\paragraph{$\Q(\sqrt2,\sqrt3)$ (ruler and compass).}
Base angles $0^\circ,45^\circ,90^\circ$; $\cos45^\circ=\tfrac{\sqrt2}{2}$. The
field is the tower of quadratics $\Q\subset\Q(\sqrt2)\subset\Q(\sqrt2,\sqrt3)$,
degree $4=2^2$; the base polynomial is $y^2-2$ with $y=2\cos45^\circ$. Exact
check: $81$ unit edges, $0$ ghosts. The numbers $\sqrt2,\sqrt3$ and the angles
$45^\circ,90^\circ$ are built by ruler and compass from a unit triangle and a
unit square (Figure~\ref{fig:tiers}, left).

\paragraph{$\Q(\sin20^\circ)$ (origami), and a degree-$6$ polynomial.}
Base angles $0^\circ,20^\circ,40^\circ$. From the triple-angle identity
$\cos3\alpha=4\cos^3\alpha-3\cos\alpha$ with $3\alpha=60^\circ$,
\begin{equation}\label{eq:cubic}
  4\cos^3 20^\circ-3\cos20^\circ=\tfrac12\ \Longrightarrow\
  \boxed{\,8x^3-6x-1=0\,}\quad(x=\cos20^\circ),
\end{equation}
irreducible over $\Q$, so $[\Q(\cos20^\circ):\Q]=3$: not ruler-and-compass, but
the trisection of $60^\circ$, realized by Beloch's fold $O6$ (the fold that maps
two points onto two lines simultaneously; equivalently a common tangent of two
parabolas, whose slope solves a cubic). Adjoining $\sqrt3$ gives the coordinate
field $\Q(\sin20^\circ)$, and $\sin20^\circ$ itself satisfies the irreducible
\emph{degree-$6$} polynomial
\begin{equation}\label{eq:deg6}
  64x^6-96x^4+36x^2-3=0\qquad(x=\sin20^\circ),
\end{equation}
equivalently $y^6-6y^4+9y^2-3=0$ for $y=2\sin20^\circ$. Thus
$[\Q(\sin20^\circ):\Q]=6=2\cdot3$: the factor $3$ from the cubic
\eqref{eq:cubic}, the factor $2$ from $\sqrt3$, prime support $\{2,3\}$, the reach of origami. Exact check: $81$ unit edges, $0$ ghosts.

\paragraph{$\Q(\cos\tfrac{2\pi}{11})$ (exotic).}
Base angles $0,\tfrac{2\pi}{11},\tfrac{4\pi}{11}$. Here $y=2\cos\tfrac{2\pi}{11}$
is a root of the irreducible quintic
\begin{equation}\label{eq:quintic}
  y^5+y^4-4y^3-3y^2+3y+1=0,
\end{equation}
equivalently $32x^5+16x^4-32x^3-12x^2+6x+1=0$ for $x=\cos\tfrac{2\pi}{11}$, so
$[\Q(\cos\tfrac{2\pi}{11}):\Q]=\varphi(11)/2=5$. Although $\cos\tfrac{2\pi}{11}$
is expressible by radicals (cyclic Galois group $C_5$), a degree divisible by
$5$ can never equal $2^a3^b$; and $11$ is not a Pierpont prime, so the regular
hendecagon is constructible neither by ruler and compass nor by paper folding.
The coordinate field is therefore \emph{exotic}. Exact check: $81$ unit edges,
$0$ ghosts.

\begin{figure}[t]\centering
  \includegraphics[width=.32\linewidth]{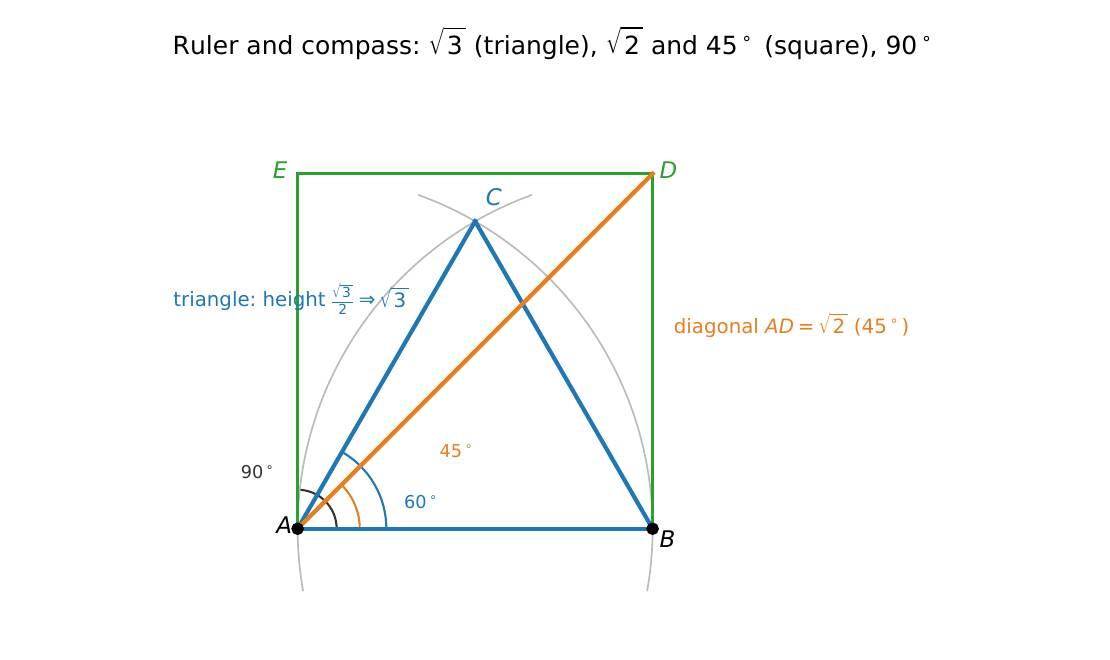}\hfill
  \includegraphics[width=.32\linewidth]{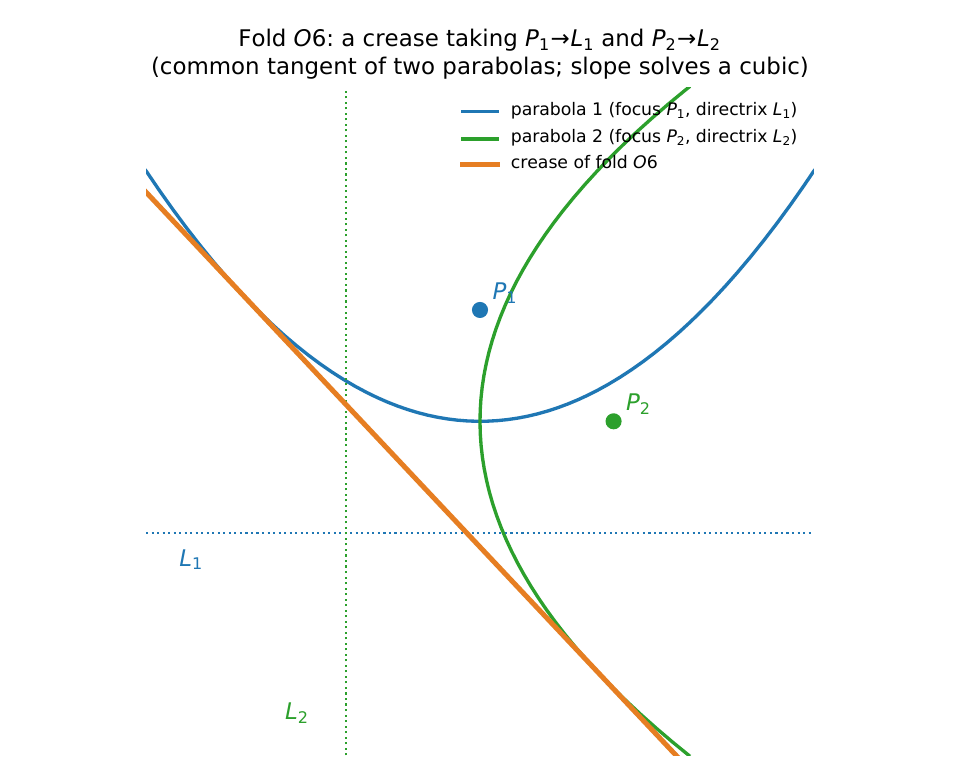}\hfill
  \includegraphics[width=.32\linewidth]{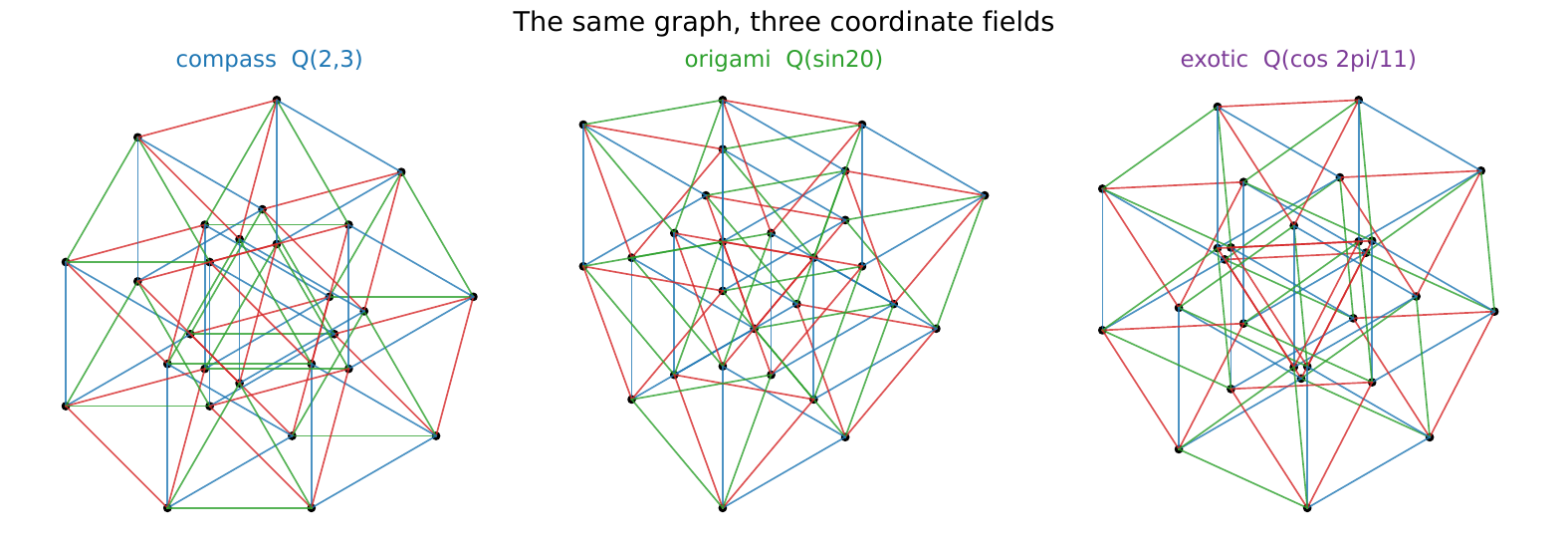}
  \caption{Left: ruler-and-compass construction of $\sqrt3,\sqrt2,45^\circ,
  90^\circ$. Middle: Beloch's fold $O6$ as a common tangent of two parabolas
  (the slope solves a cubic). Right: the same graph realized in three coordinate
  fields (compass / origami / exotic).}\label{fig:tiers}
\end{figure}

\begin{theorem}
The flexible unit-distance graph $H(3,3)$ admits faithful realizations whose
coordinate fields realize each tier of the constructibility hierarchy: compass
$\Q(\sqrt2,\sqrt3)$ (degree $4$), origami $\Q(\sin20^\circ)$ (degree $6$), and
exotic $\supseteq\Q(\cos\tfrac{2\pi}{11})$ (degree $5$). The governing base
polynomials are $y^2-2$, $y^3-3y-1$ and $y^5+y^4-4y^3-3y^2+3y+1$ respectively
(with $y=2\cos\alpha$).
\end{theorem}

\section{The space of faithful realizations}\label{sec:faithful}
Because the realization is a genuine two-parameter flex, the faithful realizations
form a \emph{two-dimensional continuum}, not a discrete set: sweeping the base
angles $(\alpha_v,\alpha_w)$ (with $\alpha_u=0$ fixing the gauge) traces the whole
family. Figure~\ref{fig:faithful} maps it. Faithfulness holds on an \emph{open, full-measure}
set (Proposition~\ref{prop:faithful-transc}): the only unfaithful parameters lie on the
ghost loci (where a non-adjacent pair falls at distance exactly $1$, or two vertices
coincide), a measure-zero union of real-analytic curves, the most visible being the
diagonal $\alpha_v=\alpha_w$ and the lines at multiples of $60^\circ$. In the diagram these
curves are thickened for visibility; the bright cells (about $60\%$ of the sampled grid at
the plotted tolerance) are those safely away from every ghost curve: the \emph{robustly}
faithful parameters, not the measure of faithfulness, which is full. The two flexes are explicit: by Proposition~\ref{prop:mink}
rotating $v$ and $w$ by independent angles preserves all $81$ edge lengths, so
$(\alpha_v,\alpha_w)$ already furnishes a two-parameter family of realizations; and
at every faithful realization the rigidity matrix has rank exactly
$49=2n-3-2$, so this two-parameter flex is the \emph{whole} infinitesimal flex,
even though the abstract graph $H(3,3)$ is \emph{generically rigid} (a random
embedding gives rank $2n-3=51$). The flexibility is thus a unit-distance phenomenon
of the Minkowski structure, present only on this special locus, not generic Laman
flexibility. It is this flex that lets $H(3,3)$ range over the
constructibility hierarchy: moving within the faithful family carries the
coordinates from compass to origami to exotic without breaking any of the $81$
unit edges.

\begin{figure}[t]\centering
  \includegraphics[width=.9\linewidth]{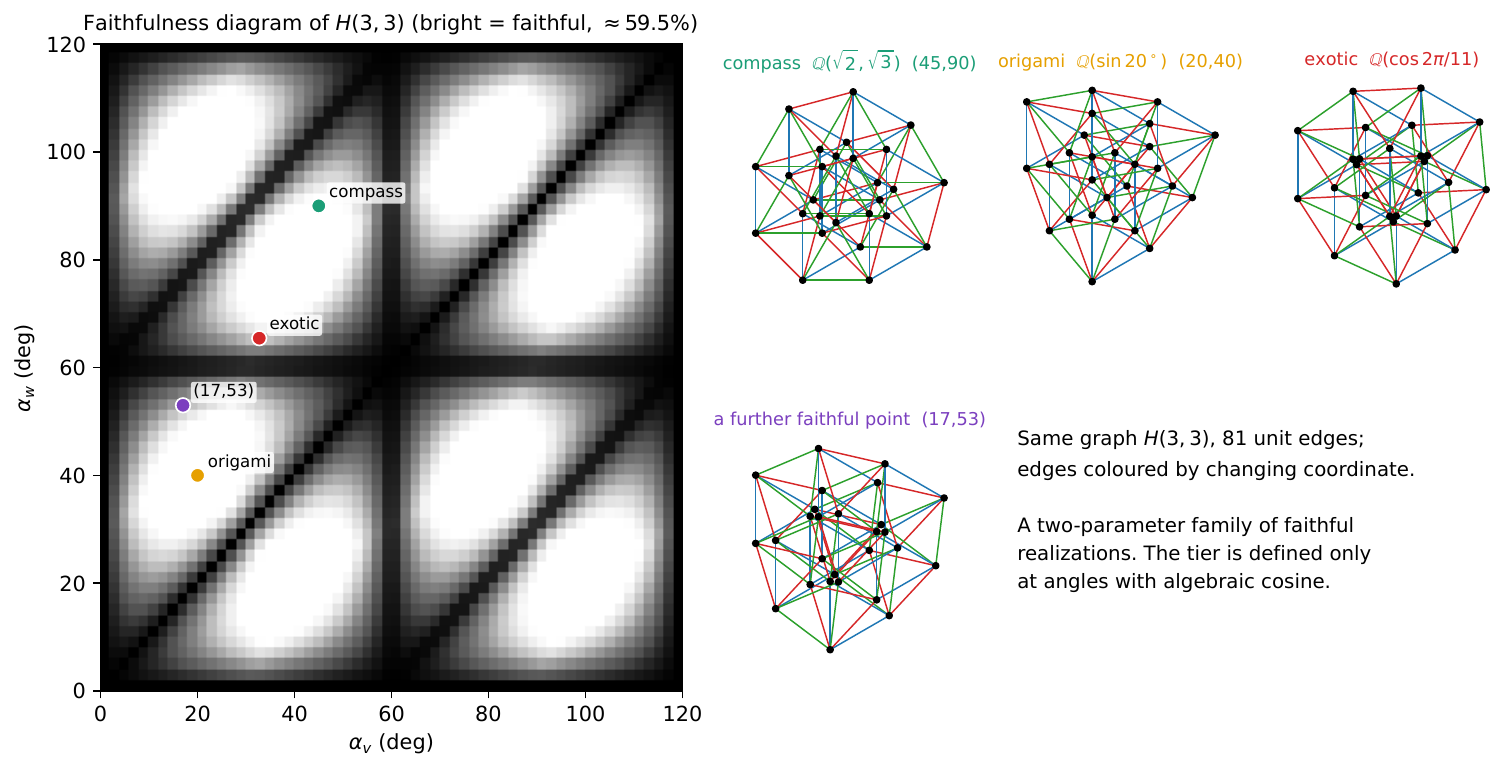}
  \caption{The faithful realizations of $H(3,3)$. \emph{Left:} the faithfulness
  diagram over the two-parameter flex $(\alpha_v,\alpha_w)$ (bright $=$ robustly faithful
  cells, safely off every ghost curve at the plotted tolerance; dark $=$ the ghost curves
  themselves, a measure-zero set thickened for visibility). \emph{Right:} four realizations of the
  \emph{same} graph: ruler-and-compass $\Q(\sqrt2,\sqrt3)$ at
  $(45^\circ,90^\circ)$, origami $\Q(\sin20^\circ)$ at $(20^\circ,40^\circ)$,
  exotic $\Q(\cos\tfrac{2\pi}{11})$ at $(\tfrac{2\pi}{11},\tfrac{4\pi}{11})$, and a
  further faithful realization at $(17^\circ,53^\circ)$; edges are
  colored by the changing coordinate.}
  \label{fig:faithful}
\end{figure}

\begin{remark}[Algebraic realizations form a countable set]\label{rem:transc}
The faithful family is a two-dimensional continuum, so $H(3,3)$ has uncountably many faithful
unit-distance realizations up to congruence. With $\alpha_u=0$ the coordinates lie in
$K=\Q(\sqrt3,\cos\alpha_v,\sin\alpha_v,\cos\alpha_w,\sin\alpha_w)$, which is a number field
exactly when $\cos\alpha_v$ and $\cos\alpha_w$ are algebraic. Such pairs of angles form a
countable set, dense in the parameter plane. The constructibility tiers are defined only on this
set, and all three occur on it: ruler-and-compass, origami, and exotic.
\end{remark}

\begin{proposition}[Faithfulness is an open condition]\label{prop:faithful-transc}
Fix $\alpha_u=0$ and parametrize the planar realizations by $(\alpha_v,\alpha_w)$. The set of
faithful parameters is open and of full Lebesgue measure; its complement is a finite union of
real-analytic ``ghost'' curves. No algebraic hypothesis on the coordinates is used.
\end{proposition}

\begin{proof}
By Proposition~\ref{prop:mink} every edge has length $1$ identically in
$(\alpha_v,\alpha_w)$, an algebraic identity requiring no field. For each
non-adjacent pair $\{x,y\}$ the squared distance $g_{xy}(\alpha_v,\alpha_w)=\|P(x)-P(y)\|^2$
is a real-analytic function; a \emph{ghost} occurs where $g_{xy}=1$. Since faithful
parameters exist (e.g.\ the compass point $(45^\circ,90^\circ)$), no $g_{xy}-1$ is
identically zero, so each $\{g_{xy}=1\}$ is a real-analytic curve (closed and of measure zero), and the ghost locus is their finite union. Its complement, the
faithful set, is therefore open and of full measure, and membership is decided by
the real inequalities $g_{xy}\neq1$ (checked by interval arithmetic), never by the
coordinate field.
\end{proof}

\begin{remark}[Two certificates: faithfulness is field-free, the tier is not]
\emph{Faithfulness} is an open real
condition, injectivity together with $\|P(x)-P(y)\|\neq1$ on non-edges, so it
is certified \emph{without any algebraic extension}: the $81$ unit edges hold by the
Minkowski identity~\eqref{eq:mink} (a polynomial identity in the base angles, valid for every $\alpha$), and the absence of ghosts is certified
either by interval arithmetic bounding every non-adjacent distance away from $1$, or
by placing the angles strictly inside a faithful cell of Figure~\ref{fig:faithful}.
The coordinate field never enters. The \emph{tier} (ruler-and-compass, origami,
or exotic) is the opposite: it \emph{is} the degree and Galois group of
$\Q(\text{coordinates})$, an arithmetic certificate that lives inside the extension
and exists only when the coordinates are algebraic. Hence a faithful
realization with non-algebraic coordinates carries a complete faithfulness certificate but \emph{no tier at all}:
faithfulness is geometry over $\R$, the tier is arithmetic over $\Q$.
\end{remark}

\section{The enneagon twist: $H(3,3)$ as three star-enneagons}
The origami tier has a symmetric realization that connects this paper to
the concentric-polygon constructions of the heptagon graphs~\cite{HCSF-hept}. Take three concentric
\emph{unit-edge star-enneagons} $\{9/1\},\{9/2\},\{9/4\}$ of radii
$\rho_r=1/(2\sin r\pi/9)$, and give the inner layer a $20^\circ$ twist. Joining the
layers by the unit braces they admit (whose subtended angles are $40^\circ$,
$20^\circ$ and $100^\circ$, all governed by the trisection cubic $8x^3-6x+1$, the
minimal polynomial of $\cos\tfrac{2\pi}{9}=\cos40^\circ$) gives $27$ points and
$81$ unit segments (verified to $5\times10^{-16}$); Figure~\ref{fig:enneagon}.

\begin{proposition}\label{prop:enneagon}
The star-enneagon twist is isomorphic to $H(3,3)$: $27$ vertices, $81$ edges,
$6$-regular, $27$ triangles, adjacency spectrum $\{6,3^{6},0^{12},(-3)^{8}\}$ (an explicit
isomorphism to $\{0,1,2\}^3$ is recorded). Here the spectrum is used as an \emph{identity
certificate}: two unit drawings as different-looking as the concentric star-enneagons and the
Minkowski sum are confirmed to be the same abstract graph because they share this adjacency
spectrum and the triangle count, both invariants of the graph and not of the drawing. It is an
\emph{origami} realization:
its coordinates lie in $\Q(\cos\tfrac{\pi}{9},\sin\tfrac{\pi}{9})$, which contains
the cyclic cubic $\Q(\cos\tfrac{2\pi}{9})$ of the trisection (Galois group $C_3$),
so the field is origami but not ruler-and-compass, since $9=3^2$ is not a
constructible order.
\end{proposition}

\begin{figure}[t]\centering
  \includegraphics[width=.9\linewidth]{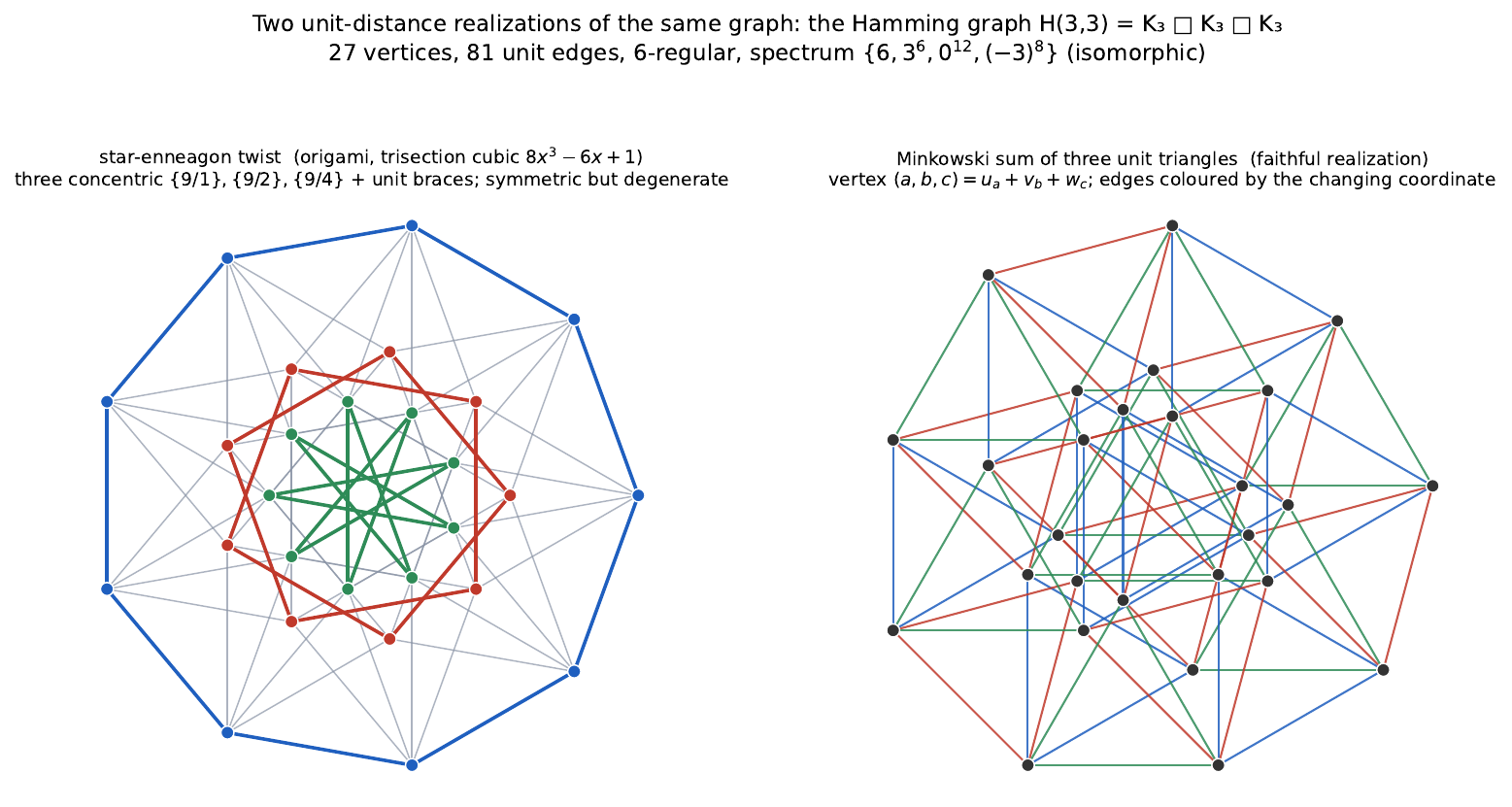}
  \caption{Two unit-distance realizations of the \emph{same} graph. \emph{Left:} the
  star-enneagon twist: three concentric $\{9/1\},\{9/2\},\{9/4\}$ with a
  $20^\circ$ turn of the inner layer, an origami realization in the trisection field.
  \emph{Right:} a Minkowski-sum realization. Both are $H(3,3)$ (same spectrum, $27$
  triangles, $81$ unit edges).}
  \label{fig:enneagon}
\end{figure}

\begin{remark}
The maximal $9$-fold symmetry has a price: this realization is \emph{degenerate}: $9$ vertices (the inner layer) each lie in the interior of two non-incident edges, a total
of $18$ vertex--edge incidences (the $27$ points remain distinct, so this is a concurrence, not a
collision), and correspondingly
infinitesimally flexible (rigidity rank $49$, the two symmetry-forced flexes above).
It is thus one \emph{degenerate} point of the flexible family of
Figure~\ref{fig:faithful}, not a strict realization; the faithful, non-degenerate
representatives of this same origami tier are the Minkowski ones. Its value is as a
symmetric picture, and as a bridge: the concentric three-star construction that
yields the heptagon graphs yields, for the $9$-gon, the Hamming graph itself.
\end{remark}

\section{A construction algorithm and exact coordinates}
The planar realization is produced by the following procedure; its arithmetic
difficulty is concentrated entirely in Step~2 (the base angles), while Steps~3--4
are rational translations and a distance audit.

\begin{center}
\fbox{\begin{minipage}{0.94\linewidth}
\textbf{Algorithm 1.} Planar unit-distance realization of $H(3,3)$.\\[2pt]
\textbf{Input:} base angles $\alpha_u,\alpha_v,\alpha_w$.
\begin{enumerate}\setlength{\itemsep}{1pt}
  \item $R\gets 1/\sqrt3$ \quad(circumradius of a unit triangle).
  \item \textbf{for} $t\in\{u,v,w\}$ and $k\in\{0,1,2\}$:\quad
        $t_k\gets R\bigl(\cos(\alpha_t+120^\circ k),\ \sin(\alpha_t+120^\circ k)\bigr)$
        \quad(build the three triangles).
  \item \textbf{for} $(a,b,c)\in\{0,1,2\}^3$:\quad $P(a,b,c)\gets u_a+v_b+w_c$
        \quad(Minkowski sum: the 27 vertices).
  \item \emph{assert} $\|P(x)-P(y)\|=1$ for every edge $xy$
        \quad(holds by Proposition~\ref{prop:mink}).
  \item $\textit{faithful}\gets\bigl[\ \|P(x)-P(y)\|\neq1$ for all non-adjacent
        $x,y\ \bigr]$.
  \item $K\gets\Q(\cos\alpha_u,\sin\alpha_u,\cos\alpha_v,\dots)$;\quad
        $\textit{tier}\gets$ compass / origami / exotic by $[K:\Q]$.
\end{enumerate}
\textbf{Output:} coordinates $\{P(a,b,c)\}$ and $(\textit{faithful},\,K,\,\textit{tier})$.
\end{minipage}}
\end{center}

The nine generators are, in each tier, exact cosines and sines of the base
angles ($R=1/\sqrt3$, $P(a,b,c)=u_a+v_b+w_c$):
\[
\begin{array}{@{}l@{\;}l@{}}
\text{compass} & {\footnotesize\setlength{\tabcolsep}{4pt}
\begin{tabular}{lll}
$u_0=R(\cos(0),\sin(0))$ & $v_0=R(\cos(\tfrac{\pi}{4}),\sin(\tfrac{\pi}{4}))$ & $w_0=R(\cos(\tfrac{\pi}{2}),\sin(\tfrac{\pi}{2}))$ \\
$u_1=R(\cos(\tfrac{2\pi}{3}),\sin(\tfrac{2\pi}{3}))$ & $v_1=R(\cos(\tfrac{\pi}{4}+\tfrac{2\pi}{3}),\sin(\tfrac{\pi}{4}+\tfrac{2\pi}{3}))$ & $w_1=R(\cos(\tfrac{\pi}{2}+\tfrac{2\pi}{3}),\sin(\tfrac{\pi}{2}+\tfrac{2\pi}{3}))$ \\
$u_2=R(\cos(\tfrac{4\pi}{3}),\sin(\tfrac{4\pi}{3}))$ & $v_2=R(\cos(\tfrac{\pi}{4}+\tfrac{4\pi}{3}),\sin(\tfrac{\pi}{4}+\tfrac{4\pi}{3}))$ & $w_2=R(\cos(\tfrac{\pi}{2}+\tfrac{4\pi}{3}),\sin(\tfrac{\pi}{2}+\tfrac{4\pi}{3}))$ \\
\end{tabular}}
\\[10pt]
\text{origami} & {\footnotesize\setlength{\tabcolsep}{4pt}
\begin{tabular}{lll}
$u_0=R(\cos(0),\sin(0))$ & $v_0=R(\cos(\tfrac{\pi}{9}),\sin(\tfrac{\pi}{9}))$ & $w_0=R(\cos(\tfrac{2\pi}{9}),\sin(\tfrac{2\pi}{9}))$ \\
$u_1=R(\cos(\tfrac{2\pi}{3}),\sin(\tfrac{2\pi}{3}))$ & $v_1=R(\cos(\tfrac{\pi}{9}+\tfrac{2\pi}{3}),\sin(\tfrac{\pi}{9}+\tfrac{2\pi}{3}))$ & $w_1=R(\cos(\tfrac{2\pi}{9}+\tfrac{2\pi}{3}),\sin(\tfrac{2\pi}{9}+\tfrac{2\pi}{3}))$ \\
$u_2=R(\cos(\tfrac{4\pi}{3}),\sin(\tfrac{4\pi}{3}))$ & $v_2=R(\cos(\tfrac{\pi}{9}+\tfrac{4\pi}{3}),\sin(\tfrac{\pi}{9}+\tfrac{4\pi}{3}))$ & $w_2=R(\cos(\tfrac{2\pi}{9}+\tfrac{4\pi}{3}),\sin(\tfrac{2\pi}{9}+\tfrac{4\pi}{3}))$ \\
\end{tabular}}
\\[10pt]
\text{exotic}  & {\footnotesize\setlength{\tabcolsep}{4pt}
\begin{tabular}{lll}
$u_0=R(\cos(0),\sin(0))$ & $v_0=R(\cos(\tfrac{2\pi}{11}),\sin(\tfrac{2\pi}{11}))$ & $w_0=R(\cos(\tfrac{4\pi}{11}),\sin(\tfrac{4\pi}{11}))$ \\
$u_1=R(\cos(\tfrac{2\pi}{3}),\sin(\tfrac{2\pi}{3}))$ & $v_1=R(\cos(\tfrac{2\pi}{11}+\tfrac{2\pi}{3}),\sin(\tfrac{2\pi}{11}+\tfrac{2\pi}{3}))$ & $w_1=R(\cos(\tfrac{4\pi}{11}+\tfrac{2\pi}{3}),\sin(\tfrac{4\pi}{11}+\tfrac{2\pi}{3}))$ \\
$u_2=R(\cos(\tfrac{4\pi}{3}),\sin(\tfrac{4\pi}{3}))$ & $v_2=R(\cos(\tfrac{2\pi}{11}+\tfrac{4\pi}{3}),\sin(\tfrac{2\pi}{11}+\tfrac{4\pi}{3}))$ & $w_2=R(\cos(\tfrac{4\pi}{11}+\tfrac{4\pi}{3}),\sin(\tfrac{4\pi}{11}+\tfrac{4\pi}{3}))$ \\
\end{tabular}}

\end{array}
\]

For the compass tier the $27$ vertices are the following exact numbers in
$\Q(\sqrt2,\sqrt3)$:
\begin{center}\resizebox{\linewidth}{!}{{\scriptsize\setlength{\tabcolsep}{4pt}
\begin{tabular}{lll}
\texttt{000}\,$(\frac{\sqrt{6} + 2 \sqrt{3}}{6},\frac{\sqrt{6} + 2 \sqrt{3}}{6})$ & \texttt{100}\,$(\frac{- \sqrt{3} + \sqrt{6}}{6},\frac{1}{2} + \frac{\sqrt{6} + 2 \sqrt{3}}{6})$ & \texttt{200}\,$(\frac{- \sqrt{3} + \sqrt{6}}{6},- \frac{1}{2} + \frac{\sqrt{6} + 2 \sqrt{3}}{6})$ \\
\texttt{001}\,$(- \frac{1}{2} + \frac{\sqrt{6} + 2 \sqrt{3}}{6},\frac{- \sqrt{3} + \sqrt{6}}{6})$ & \texttt{101}\,$(- \frac{1}{2} + \frac{- \sqrt{3} + \sqrt{6}}{6},\frac{- \sqrt{3} + \sqrt{6}}{6} + \frac{1}{2})$ & \texttt{201}\,$(- \frac{1}{2} + \frac{- \sqrt{3} + \sqrt{6}}{6},- \frac{1}{2} + \frac{- \sqrt{3} + \sqrt{6}}{6})$ \\
\texttt{002}\,$(\frac{1}{2} + \frac{\sqrt{6} + 2 \sqrt{3}}{6},\frac{- \sqrt{3} + \sqrt{6}}{6})$ & \texttt{102}\,$(\frac{- \sqrt{3} + \sqrt{6}}{6} + \frac{1}{2},\frac{- \sqrt{3} + \sqrt{6}}{6} + \frac{1}{2})$ & \texttt{202}\,$(\frac{- \sqrt{3} + \sqrt{6}}{6} + \frac{1}{2},- \frac{1}{2} + \frac{- \sqrt{3} + \sqrt{6}}{6})$ \\
\texttt{010}\,$(\frac{\sqrt{3} \left(- \sqrt{6} - \sqrt{2} + 4\right)}{12},\frac{\sqrt{3} \left(- \sqrt{2} + \sqrt{6} + 4\right)}{12})$ & \texttt{110}\,$(\frac{\sqrt{3} \left(- \sqrt{6} - 2 - \sqrt{2}\right)}{12},\frac{1}{2} + \frac{- \sqrt{6} + 3 \sqrt{2} + 4 \sqrt{3}}{12})$ & \texttt{210}\,$(\frac{\sqrt{3} \left(- \sqrt{6} - 2 - \sqrt{2}\right)}{12},- \frac{1}{2} + \frac{- \sqrt{6} + 3 \sqrt{2} + 4 \sqrt{3}}{12})$ \\
\texttt{011}\,$(- \frac{1}{2} + \frac{- 3 \sqrt{2} - \sqrt{6} + 4 \sqrt{3}}{12},\frac{\sqrt{3} \left(-2 - \sqrt{2} + \sqrt{6}\right)}{12})$ & \texttt{111}\,$(\frac{- 3 \sqrt{2} - 2 \sqrt{3} - \sqrt{6}}{12} - \frac{1}{2},\frac{- 2 \sqrt{3} - \sqrt{6} + 3 \sqrt{2}}{12} + \frac{1}{2})$ & \texttt{211}\,$(\frac{- 3 \sqrt{2} - 2 \sqrt{3} - \sqrt{6}}{12} - \frac{1}{2},- \frac{1}{2} + \frac{- 2 \sqrt{3} - \sqrt{6} + 3 \sqrt{2}}{12})$ \\
\texttt{012}\,$(\frac{- 3 \sqrt{2} - \sqrt{6} + 4 \sqrt{3}}{12} + \frac{1}{2},\frac{\sqrt{3} \left(-2 - \sqrt{2} + \sqrt{6}\right)}{12})$ & \texttt{112}\,$(\frac{- 3 \sqrt{2} - 2 \sqrt{3} - \sqrt{6}}{12} + \frac{1}{2},\frac{- 2 \sqrt{3} - \sqrt{6} + 3 \sqrt{2}}{12} + \frac{1}{2})$ & \texttt{212}\,$(\frac{- 3 \sqrt{2} - 2 \sqrt{3} - \sqrt{6}}{12} + \frac{1}{2},- \frac{1}{2} + \frac{- 2 \sqrt{3} - \sqrt{6} + 3 \sqrt{2}}{12})$ \\
\texttt{020}\,$(\frac{\sqrt{3} \left(- \sqrt{2} + \sqrt{6} + 4\right)}{12},\frac{\sqrt{3} \left(- \sqrt{6} - \sqrt{2} + 4\right)}{12})$ & \texttt{120}\,$(\frac{\sqrt{3} \left(-2 - \sqrt{2} + \sqrt{6}\right)}{12},\frac{- 3 \sqrt{2} - \sqrt{6} + 4 \sqrt{3}}{12} + \frac{1}{2})$ & \texttt{220}\,$(\frac{\sqrt{3} \left(-2 - \sqrt{2} + \sqrt{6}\right)}{12},- \frac{1}{2} + \frac{- 3 \sqrt{2} - \sqrt{6} + 4 \sqrt{3}}{12})$ \\
\texttt{021}\,$(- \frac{1}{2} + \frac{- \sqrt{6} + 3 \sqrt{2} + 4 \sqrt{3}}{12},\frac{\sqrt{3} \left(- \sqrt{6} - 2 - \sqrt{2}\right)}{12})$ & \texttt{121}\,$(- \frac{1}{2} + \frac{- 2 \sqrt{3} - \sqrt{6} + 3 \sqrt{2}}{12},\frac{- 3 \sqrt{2} - 2 \sqrt{3} - \sqrt{6}}{12} + \frac{1}{2})$ & \texttt{221}\,$(- \frac{1}{2} + \frac{- 2 \sqrt{3} - \sqrt{6} + 3 \sqrt{2}}{12},\frac{- 3 \sqrt{2} - 2 \sqrt{3} - \sqrt{6}}{12} - \frac{1}{2})$ \\
\texttt{022}\,$(\frac{1}{2} + \frac{- \sqrt{6} + 3 \sqrt{2} + 4 \sqrt{3}}{12},\frac{\sqrt{3} \left(- \sqrt{6} - 2 - \sqrt{2}\right)}{12})$ & \texttt{122}\,$(\frac{- 2 \sqrt{3} - \sqrt{6} + 3 \sqrt{2}}{12} + \frac{1}{2},\frac{- 3 \sqrt{2} - 2 \sqrt{3} - \sqrt{6}}{12} + \frac{1}{2})$ & \texttt{222}\,$(\frac{- 2 \sqrt{3} - \sqrt{6} + 3 \sqrt{2}}{12} + \frac{1}{2},\frac{- 3 \sqrt{2} - 2 \sqrt{3} - \sqrt{6}}{12} - \frac{1}{2})$ \\
\end{tabular}}
}\end{center}

In the origami tier the coordinates are exact elements of $\Q(\sin20^\circ)$
(with $\cos20^\circ$ a root of $8x^3-6x-1$); in the exotic tier, exact elements
of $\Q(\cos\tfrac{2\pi}{11})$ (root of $32x^5+16x^4-32x^3-12x^2+6x+1$), obtained
by summing the corresponding generators above.

\section{In space}\label{sec:space}
Placing the three unit triangles in the three coordinate planes ($u$ in $xy$,
$v$ in $yz$, $w$ in $zx$) and forming $P(a,b,c)=u_a+v_b+w_c\in\R^3$ gives a
realization that is \emph{automatically} unit (Proposition~\ref{prop:mink}) and,
by exact check, faithful: $27$ distinct vertices, $81$ unit edges, $0$ ghosts,
with all coordinates in $\Q(\sqrt3)$ (Figure~\ref{fig:space}, left). By
contrast, the naive placement of $abc$ at the integer point
$(a,b,c)\in\{0,1,2\}^3$ (the $3\times3\times3$ cube) is \emph{not} unit: each
$K_3$ also joins the endpoints $0$ and $2$ of a line, at distance $2$: it is a
perfectly valid straight-line drawing of $H(3,3)$ (all $81$ edges are present)
but \emph{not} a unit-distance one, since $54$ edges have length $1$ and the
$27$ endpoint edges have length $2$ (Figure~\ref{fig:cube}).

\begin{figure}[htbp]\centering
  \includegraphics[width=.67\linewidth]{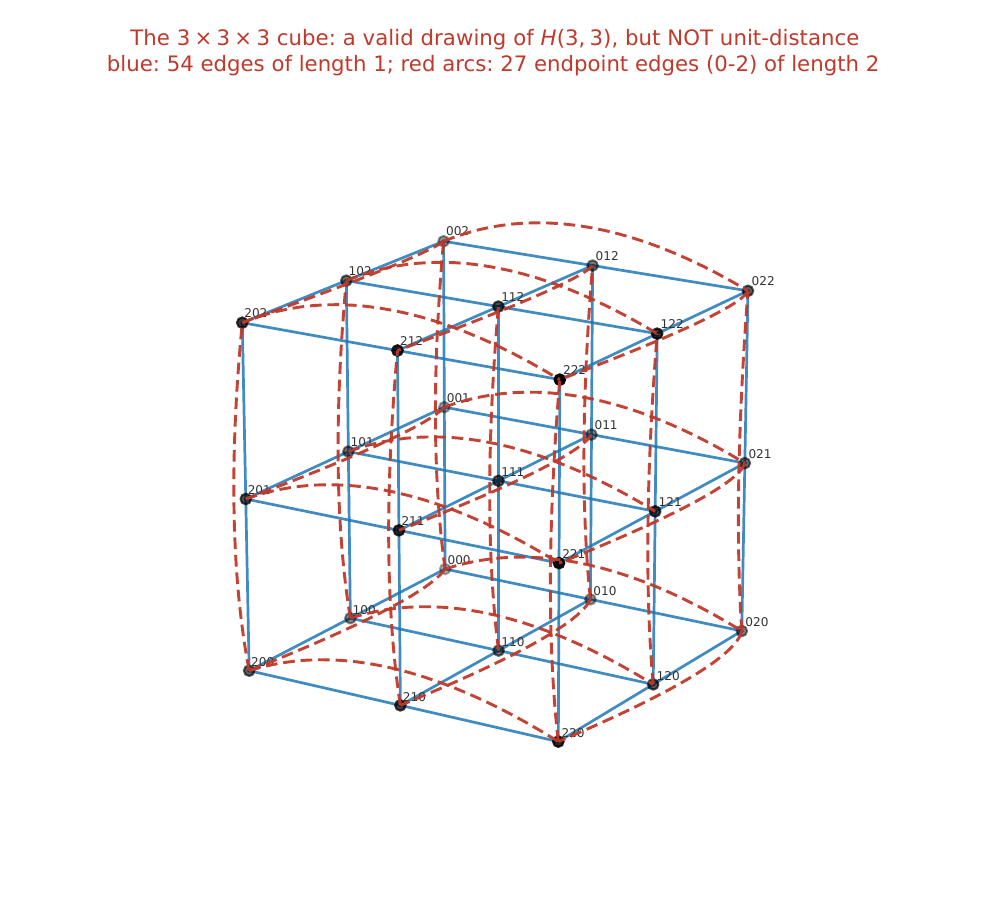}
  \caption{The $3\times3\times3$ integer cube is a valid drawing of $H(3,3)$
  (all $81$ edges present) but is \emph{not} unit-distance: the $54$ blue edges
  have length $1$, while the $27$ red endpoint edges $0$--$2$ (drawn as arcs, as
  they are collinear with the unit edges) have length $2$.}\label{fig:cube}
\end{figure}

\begin{figure}[htbp]\centering
  \includegraphics[width=.58\linewidth]{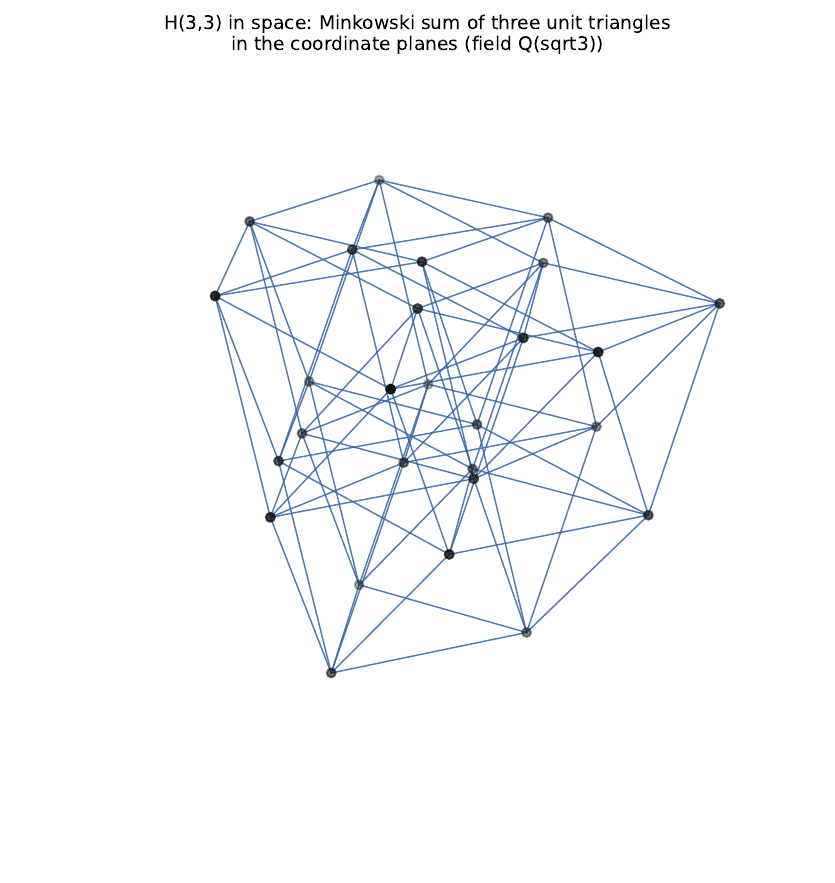}
  \caption{The coordinate-plane spatial realization of $H(3,3)$ in $\R^3$:
  $P(a,b,c)=u_a+v_b+w_c$ with $u,v,w$ three unit triangles in the $xy$, $yz$ and
  $zx$ planes. All $81$ edges are unit and the $27$ vertices are distinct
  (faithful), with coordinates in $\Q(\sqrt3)$.}\label{fig:space3d}
\end{figure}

A second spatial realization of $H(3,3)$ uses a \emph{tetrahedron}: place the
three unit triangles on three faces of a regular tetrahedron (normals along
$(1,1,1),(1,-1,-1),(-1,1,-1)$). Two sub-cases are worth keeping. In the
\emph{perfectly symmetric} position all coordinates are rational multiples of
$\sqrt6$, so the field is $\Q(\sqrt6)$ (degree $2$); the $81$ edges are unit, but
the symmetry produces $9$ \emph{ghosts}: it is not faithful. Rotating one
triangle by the rational angle $(\cos,\sin)=(\tfrac45,\tfrac35)$ (a $3$--$4$--$5$
turn) removes the ghosts, giving a \emph{faithful} realization with $27$ distinct
vertices; the turn exposes $\sqrt2$ and $\sqrt3$ separately, so the field is
$\Q(\sqrt2,\sqrt3)$ (degree $4$, still ruler and compass), Figure~\ref{fig:tetra}.

\begin{figure}[htbp]\centering
  \includegraphics[width=.66\linewidth]{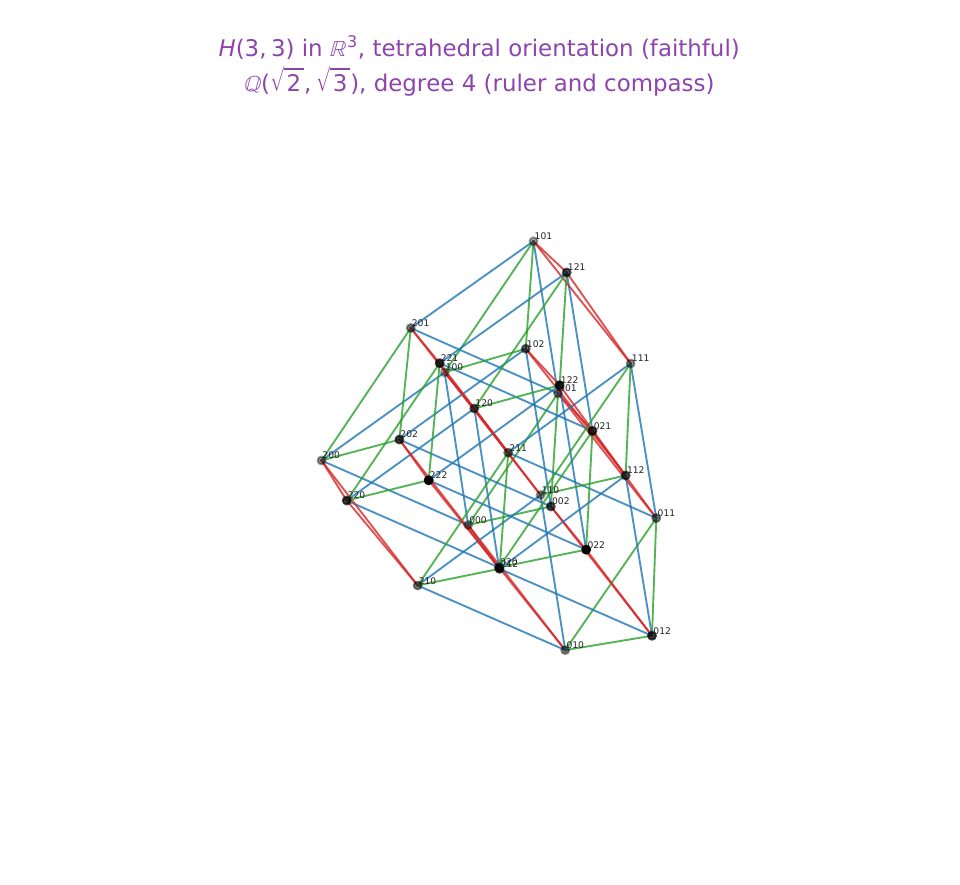}
  \caption{The faithful tetrahedral $27$-vertex realization of $H(3,3)$ in $\R^3$:
  three unit triangles on three faces of a regular tetrahedron, faithful after a
  rational $3$--$4$--$5$ turn; coordinate field $\Q(\sqrt2,\sqrt3)$, degree $4$
  (the symmetric position, without the turn, is $\Q(\sqrt6)$ but has $9$
  ghosts).}\label{fig:tetra}
\end{figure}

The tetrahedral construction is still flexible: keeping the same three
face-normals but rotating one triangle by a \emph{non-rational} angle lifts
$H(3,3)$ up the hierarchy \emph{in space}, just as in the plane. A $20^\circ$
turn gives a faithful realization whose coordinate field contains $\cos20^\circ$
(the origami cubic $8x^3-6x-1$), an \emph{origami} realization in $\R^3$
(prime support $\{2,3\}$, degree $12=2^2\cdot3$); a $\tfrac{2\pi}{11}$ turn gives
a faithful realization whose field contains $\cos\tfrac{2\pi}{11}$ (degree $5$), \emph{exotic} in $\R^3$. All ghost counts are verified to $30$ digits. Thus
$H(3,3)$ reaches the full constructibility hierarchy in the plane \emph{and} in
space (Figure~\ref{fig:space}).

\begin{remark}[In space the flex grows, and leaves the Minkowski family]
In space the family of faithful realizations is larger than in the plane, and the
structure theorem fails. In $\R^3$ the graph is again generically rigid (a random
embedding has rigidity rank $3n-6=75$), but a Minkowski realization sits on an internal
flex of dimension $75-66=9$ (against $2$ in the plane), verified across random
orientations. Of these nine, only \emph{six} come from the Minkowski family: the
independent reorientations of the three unit triangles contribute $3\times3-3=6$
nontrivial flexes once the global rotation is removed; we verified that the nine
triangle-rotation velocity fields lie in the rigidity kernel but span only a
$6$-dimensional complement to the trivial motions. The remaining \emph{three} flex
directions move the realization \emph{off} the Minkowski family while keeping every
edge unit. Hence, unlike in the plane, in $\R^3$ the Minkowski sums form a proper
$6$-parameter subvariety of the $9$-parameter faithful realization variety:
Theorem~\ref{thm:allmink} is a planar phenomenon, and there are faithful spatial
realizations of $H(d,3)$ that are \emph{not} Minkowski sums. (The reason is structural:
the orientation dichotomy $\varepsilon\in\{1,2\}$ that drives the base
case~\ref{prop:base9} is intrinsically planar; in $\R^3$ a unit triangle has a
continuum of orientations and nothing forces the additive form.) Faithfulness remains
an open condition over $\R$ (Figure~\ref{fig:faithful3d}), and the tiers still live only on the algebraic sublocus;
passing from the plane to space enlarges the faithful family and
detaches it from the Minkowski model.
\end{remark}

\begin{figure}[t]\centering
  \includegraphics[width=.72\linewidth]{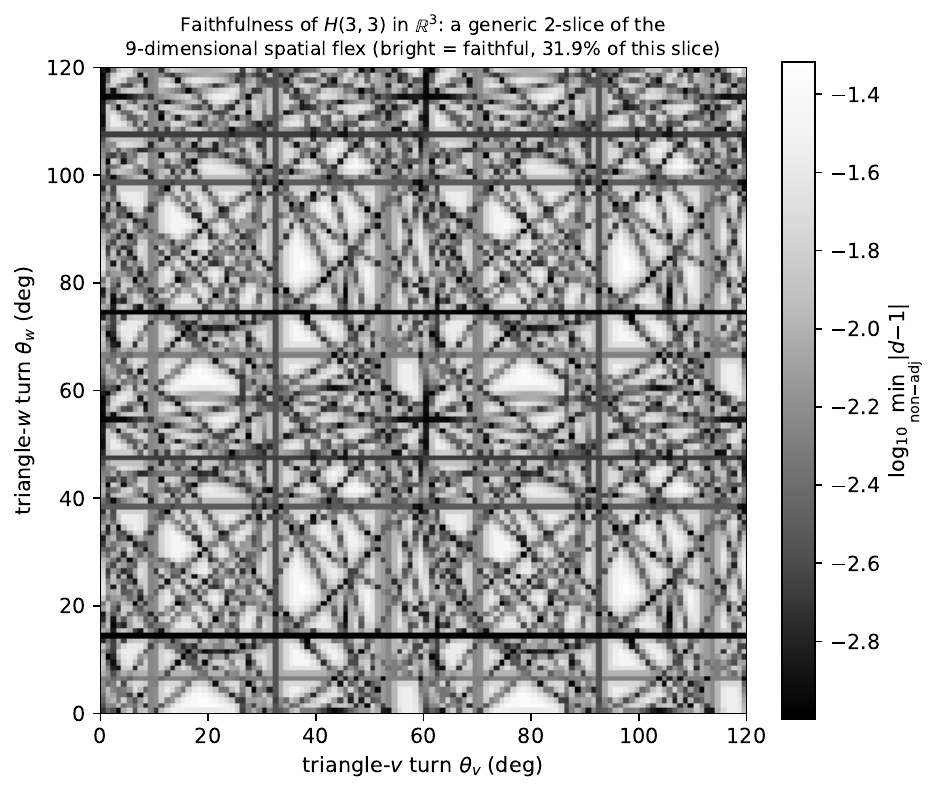}
  \caption{Faithfulness of $H(3,3)$ in $\R^3$: a generic two-parameter slice of the
  nine-dimensional spatial flex (turning two of the three unit triangles about fixed
  generic axes). Bright is robustly faithful (about $32\%$ of the sampled cells at the
  plotted tolerance); as in the plane, faithfulness fails only on a measure-zero set of
  ghost surfaces. The full spatial
  family of faithful realizations is a nine-parameter continuum (against the planar two-parameter family). The
  faithfulness diagram of the plane (Figure~\ref{fig:faithful}) is thus one
  two-dimensional shadow of a much higher-dimensional faithful set.}
  \label{fig:faithful3d}
\end{figure}

\begin{figure}[t]\centering
  \includegraphics[width=.46\linewidth]{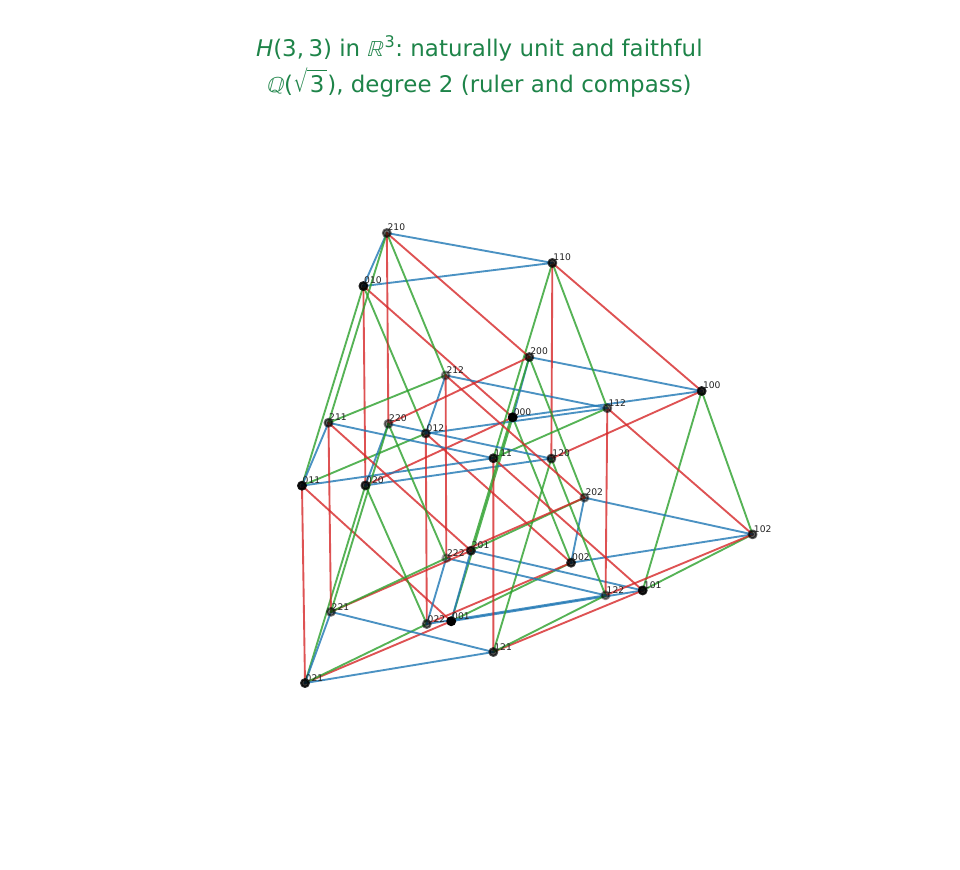}\hfill
  \includegraphics[width=.46\linewidth]{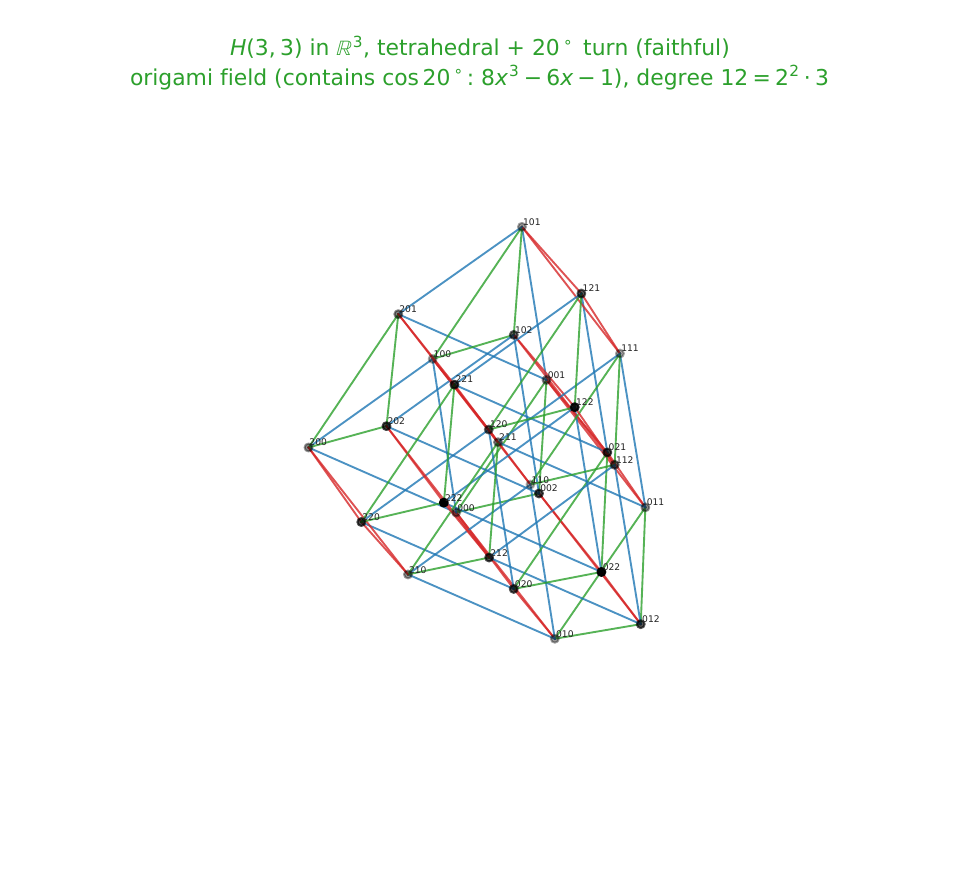}
  \caption{Left: $H(3,3)$ in $\R^3$ from the coordinate planes, naturally unit and
  faithful, coordinates in $\Q(\sqrt3)$ (compass). Right: the tetrahedral
  construction with a $20^\circ$ turn: a faithful \emph{origami} realization of
  $H(3,3)$ in $\R^3$, field containing $\cos20^\circ$, degree $12=2^2\cdot3$.}
  \label{fig:space}
\end{figure}

We record the exact $\R^3$ coordinates of the three spatial cases. First, the
coordinate-plane realization, all in $\Q(\sqrt3)$ (faithful):
\begin{center}\resizebox{\linewidth}{!}{{\tiny
\setlength{\tabcolsep}{2pt}
\begin{tabular}{ll ll ll}
\texttt{000}&$(\frac{\sqrt{3}}{3},\frac{\sqrt{3}}{3},\frac{\sqrt{3}}{3})$ & \texttt{100}&$(- \frac{\sqrt{3}}{6},\frac{1}{2} + \frac{\sqrt{3}}{3},\frac{\sqrt{3}}{3})$ & \texttt{200}&$(- \frac{\sqrt{3}}{6},- \frac{1}{2} + \frac{\sqrt{3}}{3},\frac{\sqrt{3}}{3})$ \\
\texttt{001}&$(\frac{1}{2} + \frac{\sqrt{3}}{3},\frac{\sqrt{3}}{3},- \frac{\sqrt{3}}{6})$ & \texttt{101}&$(\frac{1}{2} - \frac{\sqrt{3}}{6},\frac{1}{2} + \frac{\sqrt{3}}{3},- \frac{\sqrt{3}}{6})$ & \texttt{201}&$(\frac{1}{2} - \frac{\sqrt{3}}{6},- \frac{1}{2} + \frac{\sqrt{3}}{3},- \frac{\sqrt{3}}{6})$ \\
\texttt{002}&$(- \frac{1}{2} + \frac{\sqrt{3}}{3},\frac{\sqrt{3}}{3},- \frac{\sqrt{3}}{6})$ & \texttt{102}&$(- \frac{1}{2} - \frac{\sqrt{3}}{6},\frac{1}{2} + \frac{\sqrt{3}}{3},- \frac{\sqrt{3}}{6})$ & \texttt{202}&$(- \frac{1}{2} - \frac{\sqrt{3}}{6},- \frac{1}{2} + \frac{\sqrt{3}}{3},- \frac{\sqrt{3}}{6})$ \\
\texttt{010}&$(\frac{\sqrt{3}}{3},- \frac{\sqrt{3}}{6},\frac{1}{2} + \frac{\sqrt{3}}{3})$ & \texttt{110}&$(- \frac{\sqrt{3}}{6},\frac{1}{2} - \frac{\sqrt{3}}{6},\frac{1}{2} + \frac{\sqrt{3}}{3})$ & \texttt{210}&$(- \frac{\sqrt{3}}{6},- \frac{1}{2} - \frac{\sqrt{3}}{6},\frac{1}{2} + \frac{\sqrt{3}}{3})$ \\
\texttt{011}&$(\frac{1}{2} + \frac{\sqrt{3}}{3},- \frac{\sqrt{3}}{6},\frac{1}{2} - \frac{\sqrt{3}}{6})$ & \texttt{111}&$(\frac{1}{2} - \frac{\sqrt{3}}{6},\frac{1}{2} - \frac{\sqrt{3}}{6},\frac{1}{2} - \frac{\sqrt{3}}{6})$ & \texttt{211}&$(\frac{1}{2} - \frac{\sqrt{3}}{6},- \frac{1}{2} - \frac{\sqrt{3}}{6},\frac{1}{2} - \frac{\sqrt{3}}{6})$ \\
\texttt{012}&$(- \frac{1}{2} + \frac{\sqrt{3}}{3},- \frac{\sqrt{3}}{6},\frac{1}{2} - \frac{\sqrt{3}}{6})$ & \texttt{112}&$(- \frac{1}{2} - \frac{\sqrt{3}}{6},\frac{1}{2} - \frac{\sqrt{3}}{6},\frac{1}{2} - \frac{\sqrt{3}}{6})$ & \texttt{212}&$(- \frac{1}{2} - \frac{\sqrt{3}}{6},- \frac{1}{2} - \frac{\sqrt{3}}{6},\frac{1}{2} - \frac{\sqrt{3}}{6})$ \\
\texttt{020}&$(\frac{\sqrt{3}}{3},- \frac{\sqrt{3}}{6},- \frac{1}{2} + \frac{\sqrt{3}}{3})$ & \texttt{120}&$(- \frac{\sqrt{3}}{6},\frac{1}{2} - \frac{\sqrt{3}}{6},- \frac{1}{2} + \frac{\sqrt{3}}{3})$ & \texttt{220}&$(- \frac{\sqrt{3}}{6},- \frac{1}{2} - \frac{\sqrt{3}}{6},- \frac{1}{2} + \frac{\sqrt{3}}{3})$ \\
\texttt{021}&$(\frac{1}{2} + \frac{\sqrt{3}}{3},- \frac{\sqrt{3}}{6},- \frac{1}{2} - \frac{\sqrt{3}}{6})$ & \texttt{121}&$(\frac{1}{2} - \frac{\sqrt{3}}{6},\frac{1}{2} - \frac{\sqrt{3}}{6},- \frac{1}{2} - \frac{\sqrt{3}}{6})$ & \texttt{221}&$(\frac{1}{2} - \frac{\sqrt{3}}{6},- \frac{1}{2} - \frac{\sqrt{3}}{6},- \frac{1}{2} - \frac{\sqrt{3}}{6})$ \\
\texttt{022}&$(- \frac{1}{2} + \frac{\sqrt{3}}{3},- \frac{\sqrt{3}}{6},- \frac{1}{2} - \frac{\sqrt{3}}{6})$ & \texttt{122}&$(- \frac{1}{2} - \frac{\sqrt{3}}{6},\frac{1}{2} - \frac{\sqrt{3}}{6},- \frac{1}{2} - \frac{\sqrt{3}}{6})$ & \texttt{222}&$(- \frac{1}{2} - \frac{\sqrt{3}}{6},- \frac{1}{2} - \frac{\sqrt{3}}{6},- \frac{1}{2} - \frac{\sqrt{3}}{6})$ \\
\end{tabular}}
}\end{center}
Next, the symmetric tetrahedral position, all in $\Q(\sqrt6)$ (degree $2$, but
with $9$ ghosts):
\begin{center}\resizebox{\linewidth}{!}{{\scriptsize\setlength{\tabcolsep}{3pt}
\begin{tabular}{ll ll ll}
\texttt{000}&$(0,- \frac{\sqrt{6}}{6},- \frac{\sqrt{6}}{6})$ & \texttt{100}&$(- \frac{\sqrt{6}}{6},- \frac{\sqrt{6}}{3},\frac{\sqrt{6}}{6})$ & \texttt{200}&$(\frac{\sqrt{6}}{6},- \frac{\sqrt{6}}{2},0)$ \\
\texttt{001}&$(- \frac{\sqrt{6}}{6},0,\frac{\sqrt{6}}{6})$ & \texttt{101}&$(- \frac{\sqrt{6}}{3},- \frac{\sqrt{6}}{6},\frac{\sqrt{6}}{2})$ & \texttt{201}&$(0,- \frac{\sqrt{6}}{3},\frac{\sqrt{6}}{3})$ \\
\texttt{002}&$(\frac{\sqrt{6}}{6},\frac{\sqrt{6}}{6},0)$ & \texttt{102}&$(0,0,\frac{\sqrt{6}}{3})$ & \texttt{202}&$(\frac{\sqrt{6}}{3},- \frac{\sqrt{6}}{6},\frac{\sqrt{6}}{6})$ \\
\texttt{010}&$(- \frac{\sqrt{6}}{6},0,- \frac{\sqrt{6}}{2})$ & \texttt{110}&$(- \frac{\sqrt{6}}{3},- \frac{\sqrt{6}}{6},- \frac{\sqrt{6}}{6})$ & \texttt{210}&$(0,- \frac{\sqrt{6}}{3},- \frac{\sqrt{6}}{3})$ \\
\texttt{011}&$(- \frac{\sqrt{6}}{3},\frac{\sqrt{6}}{6},- \frac{\sqrt{6}}{6})$ & \texttt{111}&$(- \frac{\sqrt{6}}{2},0,\frac{\sqrt{6}}{6})$ & \texttt{211}&$(- \frac{\sqrt{6}}{6},- \frac{\sqrt{6}}{6},0)$ \\
\texttt{012}&$(0,\frac{\sqrt{6}}{3},- \frac{\sqrt{6}}{3})$ & \texttt{112}&$(- \frac{\sqrt{6}}{6},\frac{\sqrt{6}}{6},0)$ & \texttt{212}&$(\frac{\sqrt{6}}{6},0,- \frac{\sqrt{6}}{6})$ \\
\texttt{020}&$(\frac{\sqrt{6}}{6},\frac{\sqrt{6}}{6},- \frac{\sqrt{6}}{3})$ & \texttt{120}&$(0,0,0)$ & \texttt{220}&$(\frac{\sqrt{6}}{3},- \frac{\sqrt{6}}{6},- \frac{\sqrt{6}}{6})$ \\
\texttt{021}&$(0,\frac{\sqrt{6}}{3},0)$ & \texttt{121}&$(- \frac{\sqrt{6}}{6},\frac{\sqrt{6}}{6},\frac{\sqrt{6}}{3})$ & \texttt{221}&$(\frac{\sqrt{6}}{6},0,\frac{\sqrt{6}}{6})$ \\
\texttt{022}&$(\frac{\sqrt{6}}{3},\frac{\sqrt{6}}{2},- \frac{\sqrt{6}}{6})$ & \texttt{122}&$(\frac{\sqrt{6}}{6},\frac{\sqrt{6}}{3},\frac{\sqrt{6}}{6})$ & \texttt{222}&$(\frac{\sqrt{6}}{2},\frac{\sqrt{6}}{6},0)$ \\
\end{tabular}}
}\end{center}
Finally, the faithful tetrahedral realization after the $3$--$4$--$5$ turn, all
in $\Q(\sqrt2,\sqrt3)$ (degree $4$, faithful):
\begin{center}\resizebox{\linewidth}{!}{{\tiny\setlength{\tabcolsep}{2pt}
\begin{tabular}{ll ll ll}
\texttt{000}&$(- \frac{\sqrt{2}}{5},\frac{- 4 \sqrt{6} - 3 \sqrt{2}}{30},\frac{- 2 \sqrt{6} - \sqrt{2}}{10})$ & \texttt{100}&$(\frac{- 5 \sqrt{6} - 6 \sqrt{2}}{30},\frac{- 3 \sqrt{6} - \sqrt{2}}{10},\frac{- 3 \sqrt{2} + 4 \sqrt{6}}{30})$ & \texttt{200}&$(\frac{- 6 \sqrt{2} + 5 \sqrt{6}}{30},\frac{- 14 \sqrt{6} - 3 \sqrt{2}}{30},\frac{- 3 \sqrt{2} - \sqrt{6}}{30})$ \\
\texttt{001}&$(\frac{- 5 \sqrt{6} - 6 \sqrt{2}}{30},\frac{- 3 \sqrt{2} + \sqrt{6}}{30},\frac{- 3 \sqrt{2} + 4 \sqrt{6}}{30})$ & \texttt{101}&$(\frac{- 5 \sqrt{6} - 3 \sqrt{2}}{15},\frac{- 4 \sqrt{6} - 3 \sqrt{2}}{30},\frac{- 3 \sqrt{2} + 14 \sqrt{6}}{30})$ & \texttt{201}&$(- \frac{\sqrt{2}}{5},\frac{- 3 \sqrt{6} - \sqrt{2}}{10},\frac{- \sqrt{2} + 3 \sqrt{6}}{10})$ \\
\texttt{002}&$(\frac{- 6 \sqrt{2} + 5 \sqrt{6}}{30},\frac{- \sqrt{2} + 2 \sqrt{6}}{10},\frac{- 3 \sqrt{2} - \sqrt{6}}{30})$ & \texttt{102}&$(- \frac{\sqrt{2}}{5},\frac{- 3 \sqrt{2} + \sqrt{6}}{30},\frac{- \sqrt{2} + 3 \sqrt{6}}{10})$ & \texttt{202}&$(\frac{- 3 \sqrt{2} + 5 \sqrt{6}}{15},\frac{- 4 \sqrt{6} - 3 \sqrt{2}}{30},\frac{- 3 \sqrt{2} + 4 \sqrt{6}}{30})$ \\
\texttt{010}&$(\frac{\sqrt{2} \left(3 - 4 \sqrt{3}\right)}{30},\frac{\sqrt{2}}{5},\frac{- 14 \sqrt{6} - 3 \sqrt{2}}{30})$ & \texttt{110}&$(\frac{- 3 \sqrt{6} + \sqrt{2}}{10},\frac{- 5 \sqrt{6} + 6 \sqrt{2}}{30},\frac{- 4 \sqrt{6} - 3 \sqrt{2}}{30})$ & \texttt{210}&$(\frac{\sqrt{6} + 3 \sqrt{2}}{30},\frac{- 5 \sqrt{6} + 3 \sqrt{2}}{15},\frac{- 3 \sqrt{6} - \sqrt{2}}{10})$ \\
\texttt{011}&$(\frac{- 3 \sqrt{6} + \sqrt{2}}{10},\frac{6 \sqrt{2} + 5 \sqrt{6}}{30},\frac{- 4 \sqrt{6} - 3 \sqrt{2}}{30})$ & \texttt{111}&$(\frac{- 14 \sqrt{6} + 3 \sqrt{2}}{30},\frac{\sqrt{2}}{5},\frac{- \sqrt{2} + 2 \sqrt{6}}{10})$ & \texttt{211}&$(\frac{\sqrt{2} \left(3 - 4 \sqrt{3}\right)}{30},\frac{- 5 \sqrt{6} + 6 \sqrt{2}}{30},\frac{- 3 \sqrt{2} + \sqrt{6}}{30})$ \\
\texttt{012}&$(\frac{\sqrt{6} + 3 \sqrt{2}}{30},\frac{3 \sqrt{2} + 5 \sqrt{6}}{15},\frac{- 3 \sqrt{6} - \sqrt{2}}{10})$ & \texttt{112}&$(\frac{\sqrt{2} \left(3 - 4 \sqrt{3}\right)}{30},\frac{6 \sqrt{2} + 5 \sqrt{6}}{30},\frac{- 3 \sqrt{2} + \sqrt{6}}{30})$ & \texttt{212}&$(\frac{\sqrt{2} + 2 \sqrt{6}}{10},\frac{\sqrt{2}}{5},\frac{- 4 \sqrt{6} - 3 \sqrt{2}}{30})$ \\
\texttt{020}&$(\frac{\sqrt{2} \left(3 + 4 \sqrt{3}\right)}{30},\frac{- 3 \sqrt{2} + 4 \sqrt{6}}{30},\frac{- 5 \sqrt{6} + 3 \sqrt{2}}{15})$ & \texttt{120}&$(\frac{- \sqrt{6} + 3 \sqrt{2}}{30},\frac{- 3 \sqrt{2} - \sqrt{6}}{30},\frac{\sqrt{2}}{5})$ & \texttt{220}&$(\frac{\sqrt{2} + 3 \sqrt{6}}{10},\frac{- 2 \sqrt{6} - \sqrt{2}}{10},\frac{- 5 \sqrt{6} + 6 \sqrt{2}}{30})$ \\
\texttt{021}&$(\frac{- \sqrt{6} + 3 \sqrt{2}}{30},\frac{- \sqrt{2} + 3 \sqrt{6}}{10},\frac{\sqrt{2}}{5})$ & \texttt{121}&$(\frac{- 2 \sqrt{6} + \sqrt{2}}{10},\frac{- 3 \sqrt{2} + 4 \sqrt{6}}{30},\frac{3 \sqrt{2} + 5 \sqrt{6}}{15})$ & \texttt{221}&$(\frac{\sqrt{2} \left(3 + 4 \sqrt{3}\right)}{30},\frac{- 3 \sqrt{2} - \sqrt{6}}{30},\frac{6 \sqrt{2} + 5 \sqrt{6}}{30})$ \\
\texttt{022}&$(\frac{\sqrt{2} + 3 \sqrt{6}}{10},\frac{- 3 \sqrt{2} + 14 \sqrt{6}}{30},\frac{- 5 \sqrt{6} + 6 \sqrt{2}}{30})$ & \texttt{122}&$(\frac{\sqrt{2} \left(3 + 4 \sqrt{3}\right)}{30},\frac{- \sqrt{2} + 3 \sqrt{6}}{10},\frac{6 \sqrt{2} + 5 \sqrt{6}}{30})$ & \texttt{222}&$(\frac{3 \sqrt{2} + 14 \sqrt{6}}{30},\frac{- 3 \sqrt{2} + 4 \sqrt{6}}{30},\frac{\sqrt{2}}{5})$ \\
\end{tabular}}
}\end{center}

As an alternative view, Figure~\ref{fig:lattice} places each spatial realization
of $H(3,3)$ inside a reference lattice, which makes the three-dimensional
structure easier to read.

\begin{figure}[htbp]\centering
  \includegraphics[width=.33\linewidth]{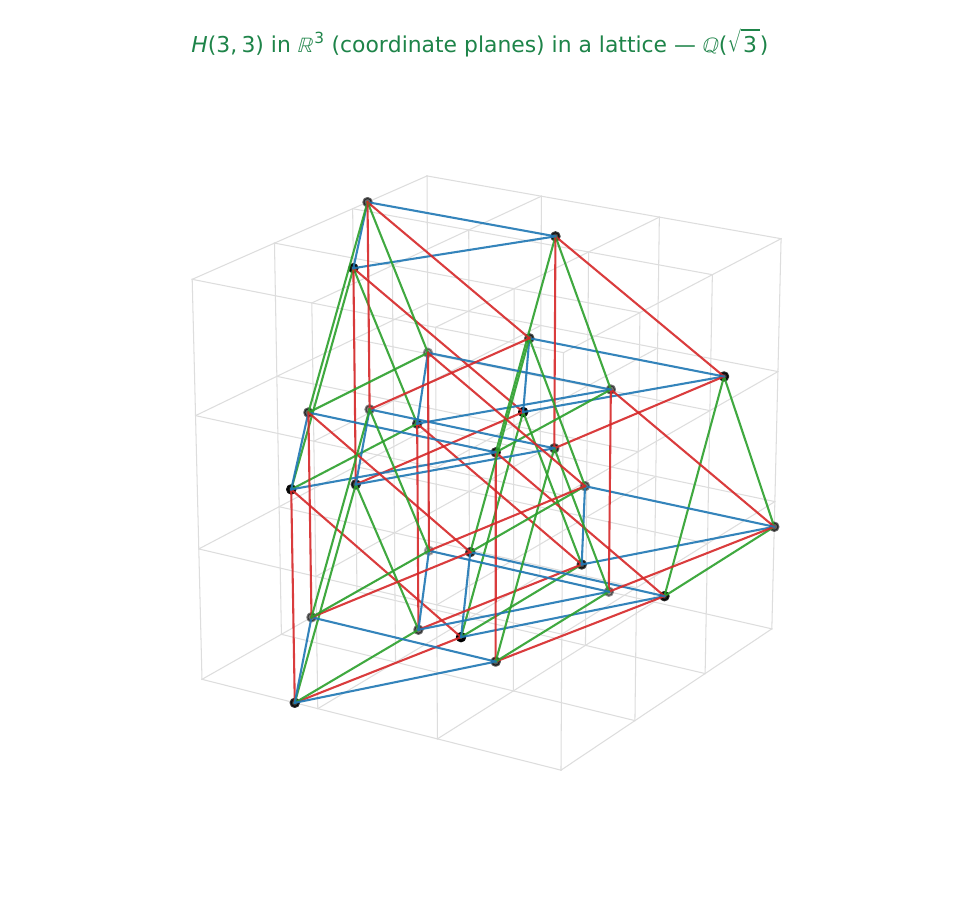}\hfill
  \includegraphics[width=.33\linewidth]{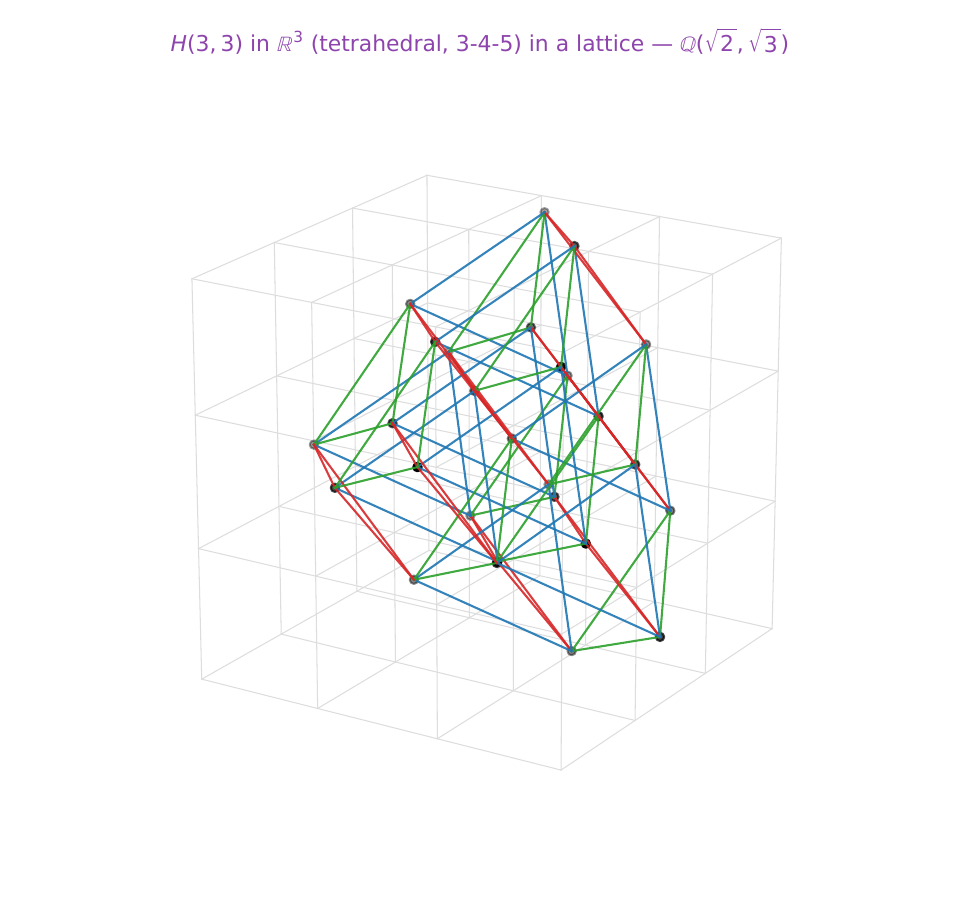}\hfill
  \includegraphics[width=.33\linewidth]{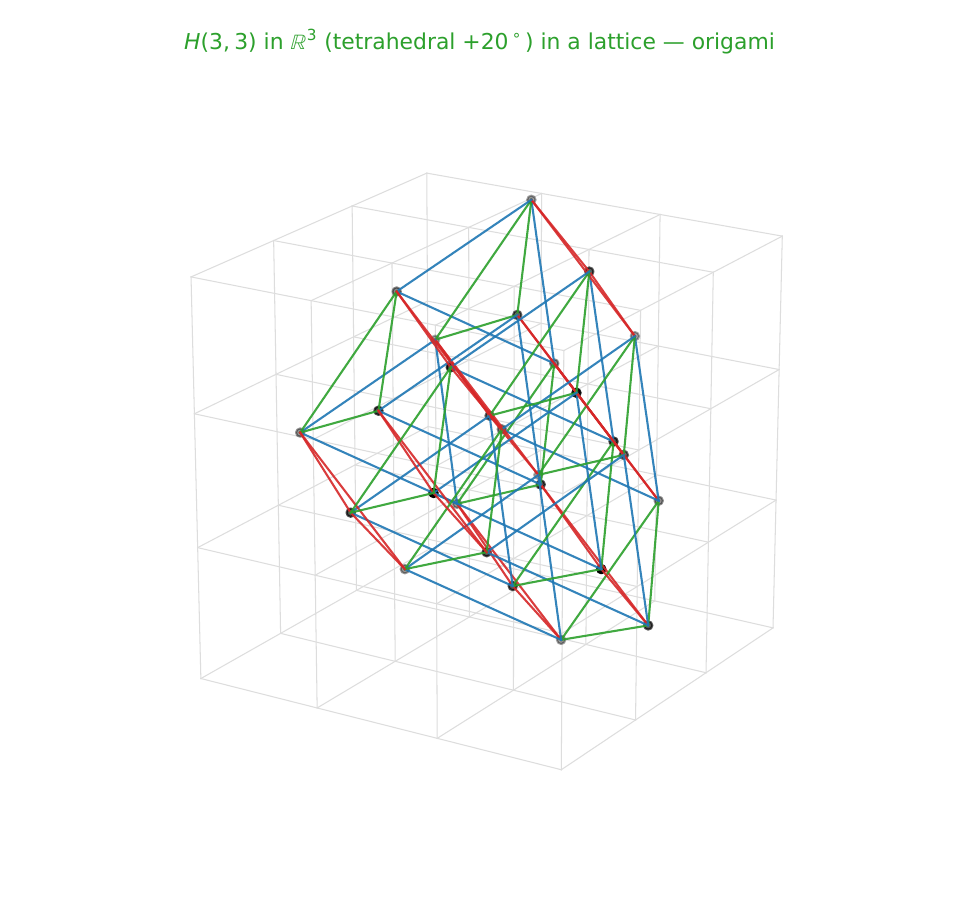}
  \caption{The three spatial realizations of $H(3,3)$ shown inside a reference
  lattice: coordinate planes ($\Q(\sqrt3)$, compass), tetrahedral $3$--$4$--$5$
  ($\Q(\sqrt2,\sqrt3)$, compass), and tetrahedral $+20^\circ$ (origami).}
  \label{fig:lattice}
\end{figure}

\section{The family $H(d,3)$: a $(d-1)$-parameter flex}
The Minkowski construction of Section~\ref{sec:mink} is not special to $d=3$. For every $d$,
\[
  H(d,3)\;=\;\underbrace{K_3\bx K_3\bx\cdots\bx K_3}_{d\ \text{factors}}\;=\;K_3^{\,\mathbin{\square} d},
  \qquad V=\{0,1,2\}^d,
\]
two words adjacent iff they differ in exactly one coordinate, is a \emph{planar}
unit-distance graph: place the vertex $(a_1,\dots,a_d)$ at
\begin{equation}\label{eq:immersion}
  P(a_1,\dots,a_d)\;=\;\sum_{i=1}^{d} T_i[a_i],
\end{equation}
where $T_1,\dots,T_d$ are unit equilateral triangles (circumradius $1/\sqrt3$) in
independently chosen orientations. Every edge changes one coordinate, so it is a
translate of a triangle side and has length one; for generic orientations the $3^d$
points are distinct with no ghost unit distance, so the drawing is faithful.

\begin{theorem}\label{thm:edim}
For $q\ge3$ the Euclidean dimension of the Hamming graph (in the sense of
Erd\H{o}s--Harary--Tutte and Soifer~\cite{EHT1965,Soifer}) is
$\operatorname{edim}(H(d,q))=q-1$: its least faithful unit-distance realization lives in
$\R^{q-1}$, as the Minkowski sum of $d$ regular unit $(q-1)$-simplices in generic position. In particular $q=3$ is exactly the case that
stays in the plane. (The binary case $q=2$ is the exception: $H(d,2)=Q_d$ is the hypercube, with
$\operatorname{edim}(H(d,2))=2$ for every $d\ge2$, e.g.\ $H(2,2)=C_4$ needs the plane, since
$d$ unit $1$-simplices cannot be placed in general position on the line.)
\end{theorem}

\begin{proof}
Take $q\ge3$. Each factor $K_q$ is faithful in $\R^{q-1}$ as a regular unit simplex and no lower,
so $\operatorname{edim}(H(d,q))\ge q-1$. Conversely, choosing $d$ regular unit simplices
$T_1,\dots,T_d\subset\R^{q-1}$ in generic orientations (which exist since $q-1\ge2$) and placing
vertex $a$ at $\sum_i T_i[a_i]$ makes every edge a translate of a simplex edge, hence unit;
genericity keeps the $q^{d}$ points distinct with no extra unit distance, so the realization is
faithful in $\R^{q-1}$. Thus $\operatorname{edim}(H(d,q))=q-1$.
\end{proof}

\paragraph{The immersion of $H(4,3)$ explicitly.} Here
$H(4,3)=K_3\bx K_3\bx K_3\bx K_3$ has $81$ vertices, is $8$-regular (degree $2d=8$) with $324$ edges,
and Figure~\ref{fig:h43} realizes \eqref{eq:immersion} with the four unit triangles
(twists $0,\,0.5,\,1.1,\,1.9$ rad)
\[
\begin{aligned}
T_1&=\{(0.577,0),(-0.289,0.5),(-0.289,-0.5)\},\\
T_2&=\{(0.507,0.277),(-0.493,0.300),(-0.014,-0.577)\},\\
T_3&=\{(0.262,0.514),(-0.577,-0.031),(0.315,-0.484)\},\\
T_4&=\{(-0.187,0.546),(-0.380,-0.435),(0.566,-0.112)\},
\end{aligned}
\]
so, e.g., $P(0,0,0,0)=(1.159,1.338)$ and $|P(0,0,0,0)-P(1,0,0,0)|=1$; all $324$ edges
are unit and the realization is faithful (verified in exact arithmetic).

\begin{figure}[t]\centering
  \includegraphics[width=.55\linewidth]{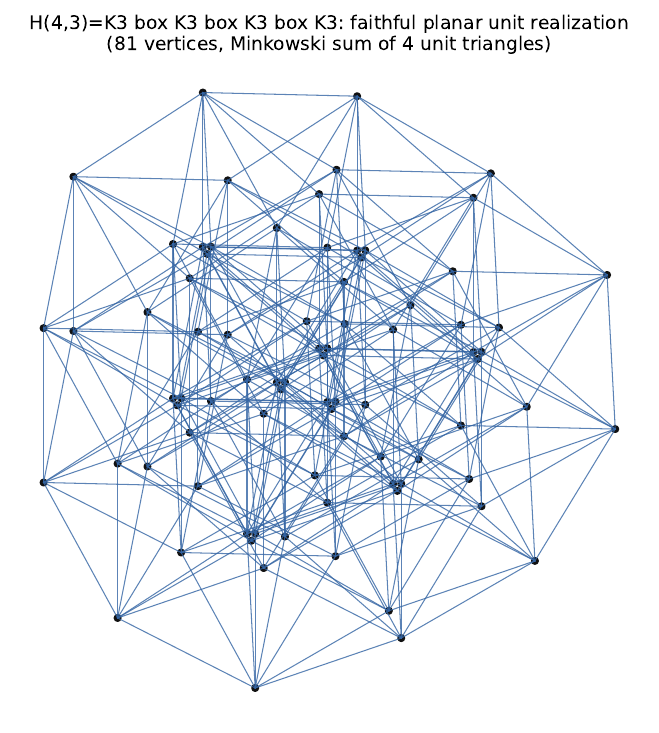}
  \caption{A faithful planar unit-distance realization of $H(4,3)=K_3^{\,\mathbin{\square}4}$
  ($81$ vertices, $324$ edges) as a Minkowski sum of four unit triangles.}
  \label{fig:h43}
\end{figure}

The flexibility of $H(3,3)$ extends to the whole family, not only $d=3$.

\begin{theorem}\label{prop:flex}
Let $P(a)=\sum_{i=1}^{d}T_i[a_i]$, $a\in\{0,1,2\}^d$, be a Minkowski unit realization of
$H(d,3)$, $d\ge1$, with \emph{arbitrary} orientations of the unit triangles $T_1,\dots,T_d$
(faithful or not). Then the infinitesimal flexes of $P$ are exactly the maps
\[
  v(a)=\tau+\sum_{i=1}^{d}c_i\,J\,T_i[a_i]\qquad(\tau\in\R^2,\ c\in\R^d),
\]
where $J$ is the rotation by $90^\circ$. This space has dimension $d+2$; the internal
infinitesimal flexibility (modulo the $3$ trivial motions) has dimension exactly $d-1$, and
the rigidity matrix has rank exactly $2n-3-(d-1)$, $n=3^d$. No genericity is needed.
\end{theorem}

\begin{proof}
Let $v\colon\{0,1,2\}^d\to\R^2$ be an infinitesimal flex: $(v(a)-v(b))\cdot(P(a)-P(b))=0$
for every edge $ab$. An edge that changes coordinate $j$ from $x$ to $y$ has
$P(a)-P(b)=T_j[x]-T_j[y]$. For $a\in\{0,1,2\}^d$ write $a^{j\to x}$ for $a$ with its $j$-th
coordinate replaced by $x$, and $e^{j}_{x}=T_j[x]-T_j[0]$.

\emph{Step 1 (lines).} The three vertices $a^{j\to0},a^{j\to1},a^{j\to2}$ form a unit
triangle, and the flex conditions on its three edges say that their velocities are an
infinitesimal isometry of a nondegenerate triangle, i.e.\ a translation plus an
infinitesimal rotation. Hence
\[
  v(a)=v(a^{j\to0})+\beta_j(a)\,J\,e^{j}_{a_j},
\]
with a real number $\beta_j(a)$ that does not depend on $a_j$.

\emph{Step 2 (mixed differences).} Fix $j\neq k$, a vertex $a$, and $x,z\in\{1,2\}$. Expanding
$v(a^{j\to x,k\to z})-v(a^{j\to0,k\to z})-v(a^{j\to x,k\to0})+v(a^{j\to0,k\to0})$ first along
the $j$-lines and then along the $k$-lines gives
\[
  \bigl(\beta_j(a^{k\to z})-\beta_j(a^{k\to0})\bigr)\,e^{j}_{x}
  =\bigl(\beta_k(a^{j\to x})-\beta_k(a^{j\to0})\bigr)\,e^{k}_{z}.
\]
The two sides $e^{j}_{1},e^{j}_{2}$ of the triangle $T_j$ are not parallel, so at most one of
them is parallel to $e^{k}_{z}$; choosing $x$ with $e^{j}_{x}\nparallel e^{k}_{z}$, both
real coefficients vanish, and $\beta_j(a^{k\to z})=\beta_j(a^{k\to0})$. This holds for every
orientation of the triangles: parallel sides of $T_j$ and $T_k$, or even $T_j=T_k$, do not
matter.

\emph{Step 3 (constancy).} By Step~2, $\beta_j$ does not depend on any coordinate $a_k$,
$k\neq j$, and by Step~1 not on $a_j$; so $\beta_j\equiv c_j$ is a constant. Then
$w(a)=v(a)-\sum_i c_i\,J\,T_i[a_i]$ satisfies $w(a)=w(a^{j\to0})$ for every $j$, so $w$ is
constant on the connected graph: $w\equiv\tau$, and $v$ has the stated form.

\emph{Step 4 (dimensions).} Conversely every such $v$ is a flex, since along an edge in
coordinate $j$ it changes by $c_j\,J(T_j[x]-T_j[y])$, which is orthogonal to
$T_j[x]-T_j[y]$. The parametrization $(\tau,c)\mapsto v$ is injective: changing only the
$j$-th coordinate gives $c_j\,J\,e^{j}_{1}=0$, so $c_j=0$, and then $\tau=0$. Hence the flex
space has dimension $d+2$. The trivial motions are the translations ($c=0$) and the global
rotation $v=J\,P=\sum_i J\,T_i[a_i]$ ($c_1=\dots=c_d$), a $3$-dimensional subspace for
$d\ge1$; the internal flexibility is $d+2-3=d-1$, and by rank--nullity the rigidity matrix,
with $2n$ columns, has rank $2n-(d+2)=2n-3-(d-1)$.
\end{proof}

\begin{remark}[Degenerate realizations]
Version~2 of this paper stated Theorem~\ref{prop:flex} at a \emph{generic} Minkowski
realization, where the edge directions of different triangles are non-parallel. The proof
above removes that hypothesis: the flex dimension is $d-1$ at \emph{every} Minkowski
realization, including degenerate ones in which vertices coincide or sides of different
triangles are parallel (for instance all $T_i$ equal), where $P$ is not faithful. The point
is that a triangle always offers two non-parallel sides, so the mixed-difference argument
never needs genericity.
\end{remark}

We also verified the rank numerically, including degenerate orientations:
$\dim\mathrm{flex}=1,2,3,4$ for $d=2,3,4,5$, matching $d-1$. So $H(3,3)$'s two-parameter
flex of Section~\ref{sec:faithful} is the case $d=3$, and each $H(d,3)$ has the \emph{same} kind of
flexibility one dimension richer: the graph is generically rigid in the plane, yet its unit
realizations sweep a $(d-1)$-parameter continuum, exactly the Minkowski twist family. Thus the
tier phenomena of Sections~\ref{sec:fields}--\ref{sec:faithful} recur along the family, with $d-1$ free angles.

\begin{remark}[General $q$]\label{rem:flexq}
The same argument applies to $H(d,q)$, whose generic Minkowski realization is a sum of $d$
unit $(q-1)$-simplices in $\R^{q-1}$ (Theorem~\ref{thm:edim}). Each simplex is rigid and may
rotate by $SO(q-1)$, giving at least $(d-1)\binom{q-1}{2}$ internal flexes; for $q=3$ this is
$d-1$ and is exact. For $q\ge4$ the rotation group is non-abelian and the flex is strictly
larger (numerically $4$ and $9$ for $H(2,4),H(3,4)$, versus the rotational count $3,6$): the
extra motions couple rotations across simplices, and pinning the exact dimension for $q\ge4$
is open.
\end{remark}

\begin{figure}[t]\centering
  \includegraphics[width=.58\linewidth]{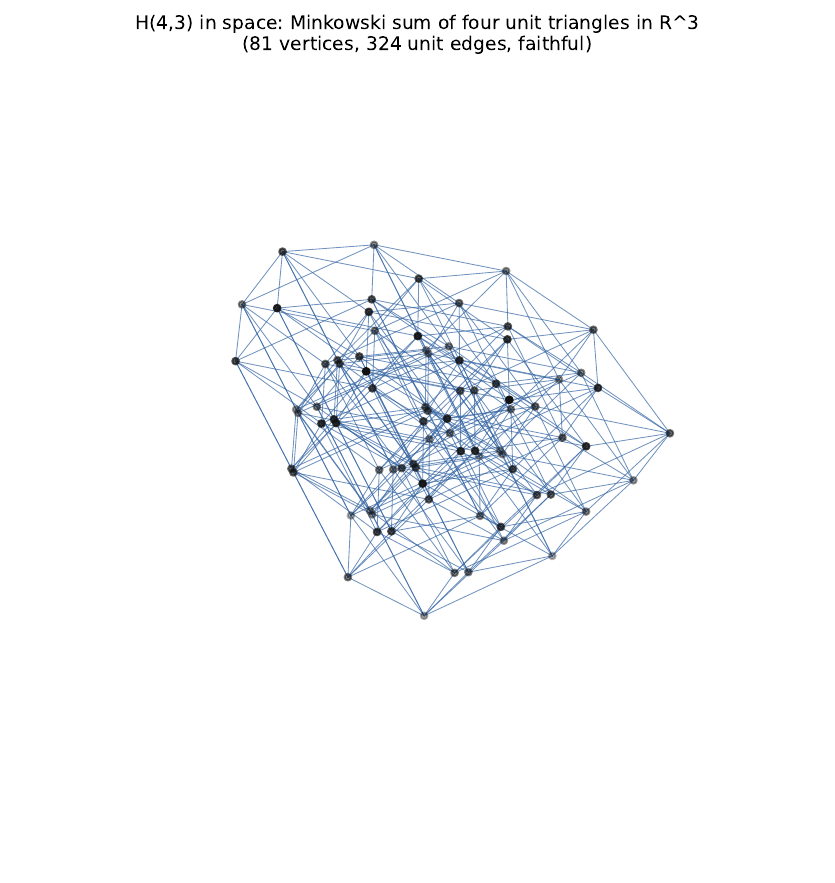}
  \caption{$H(4,3)$ lifted into space: the Minkowski sum of four unit triangles in
  four generic planes of $\R^3$ ($81$ vertices, $324$ unit edges, faithful). In the
  plane its unit realization is forced to cross; in $\R^3$ it is crossing-free.}
  \label{fig:h43space}
\end{figure}

\section{A gallery of the family $H(d,3)$}\label{sec:gallery}
The Minkowski construction draws the whole family at once. We collect large plots of the
faithful planar unit-distance realizations for $d=2,\dots,6$
(Figures~\ref{fig:g23}--\ref{fig:g63}) and four positions along the flex of $H(3,3)$
(Figure~\ref{fig:flexgallery}); the self-similar,
triangular structure (each level three shrunken-and-translated copies of the previous
one) is visible throughout, and it is exactly what the recursive method of the next
section exploits.

\begin{figure}[tp]\centering
  \includegraphics[width=.6\linewidth]{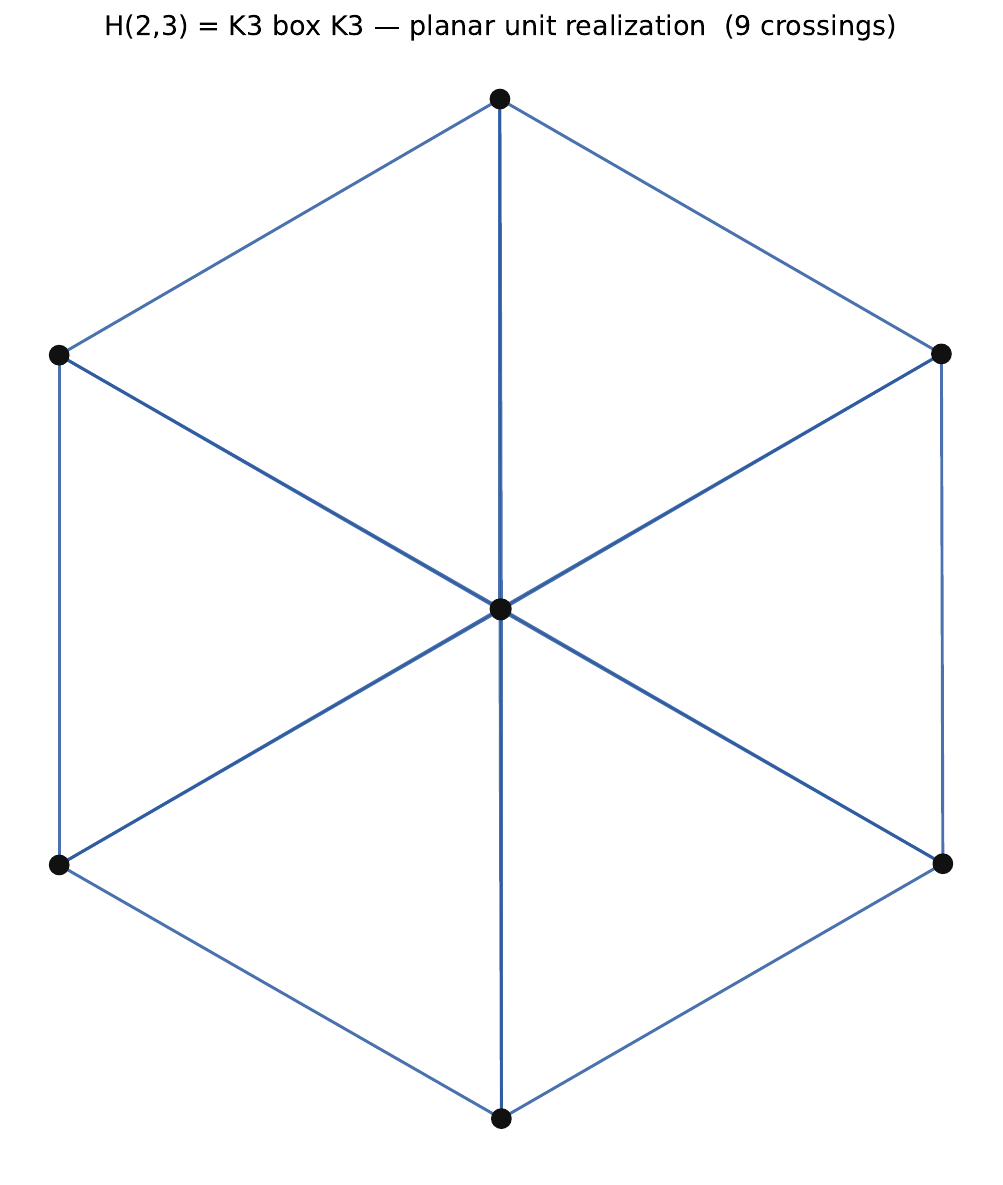}
  \caption{$H(2,3)=K_3\bx K_3$ (the $3\times3$ rook's / Paley$(9)$ / Latin-square graph),
  as a faithful planar unit-distance graph. Every one of the $18$ edges is a unit
  segment; the drawing has $9$ crossings and this is its unit-distance crossing number,
  $\udcr(K_3\bx K_3)=9$, against the ordinary $\cro=3$.}\label{fig:g23}
\end{figure}

\begin{figure}[tp]\centering
  \includegraphics[width=.6\linewidth]{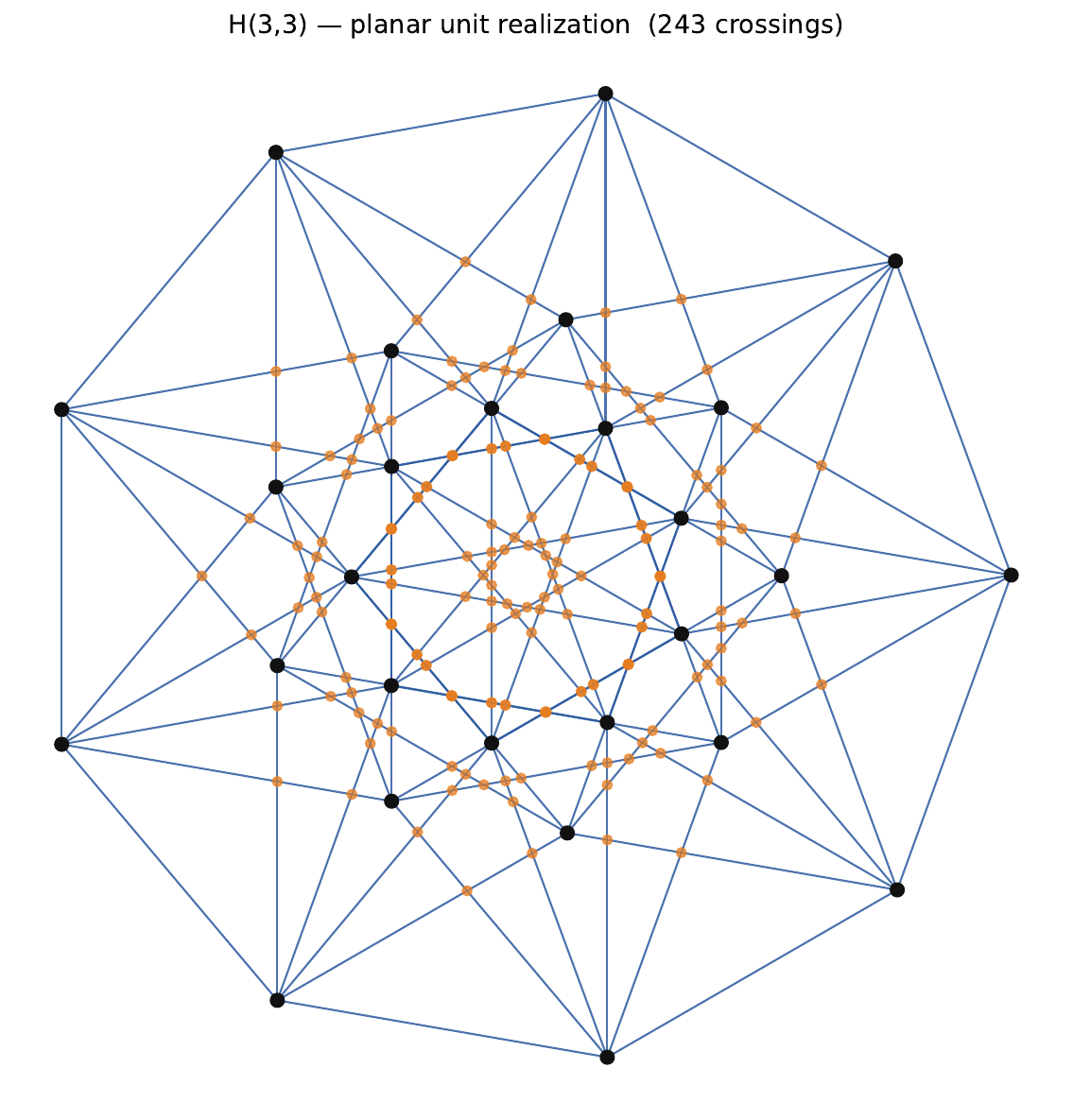}
  \caption{$H(3,3)=K_3^{\bx3}$ ($27$ vertices, $81$ unit edges): three translated copies
  of the $K_3\bx K_3$ drawing of Figure~\ref{fig:g23}, joined by unit-triangle bundles.
  This is a \emph{generic} (non-origami) realization, and the drawing shown has $243$
  crossings; good drawings near the \emph{origami} realization of the \emph{same} graph
  have $231$ (Figure~\ref{fig:h33origami}), the upper bound entered in
  Table~\ref{tab:values} and Conjecture~\ref{conj:mintier}.}
  \label{fig:g33}
\end{figure}

\begin{figure}[tp]\centering
  \includegraphics[width=.6\linewidth]{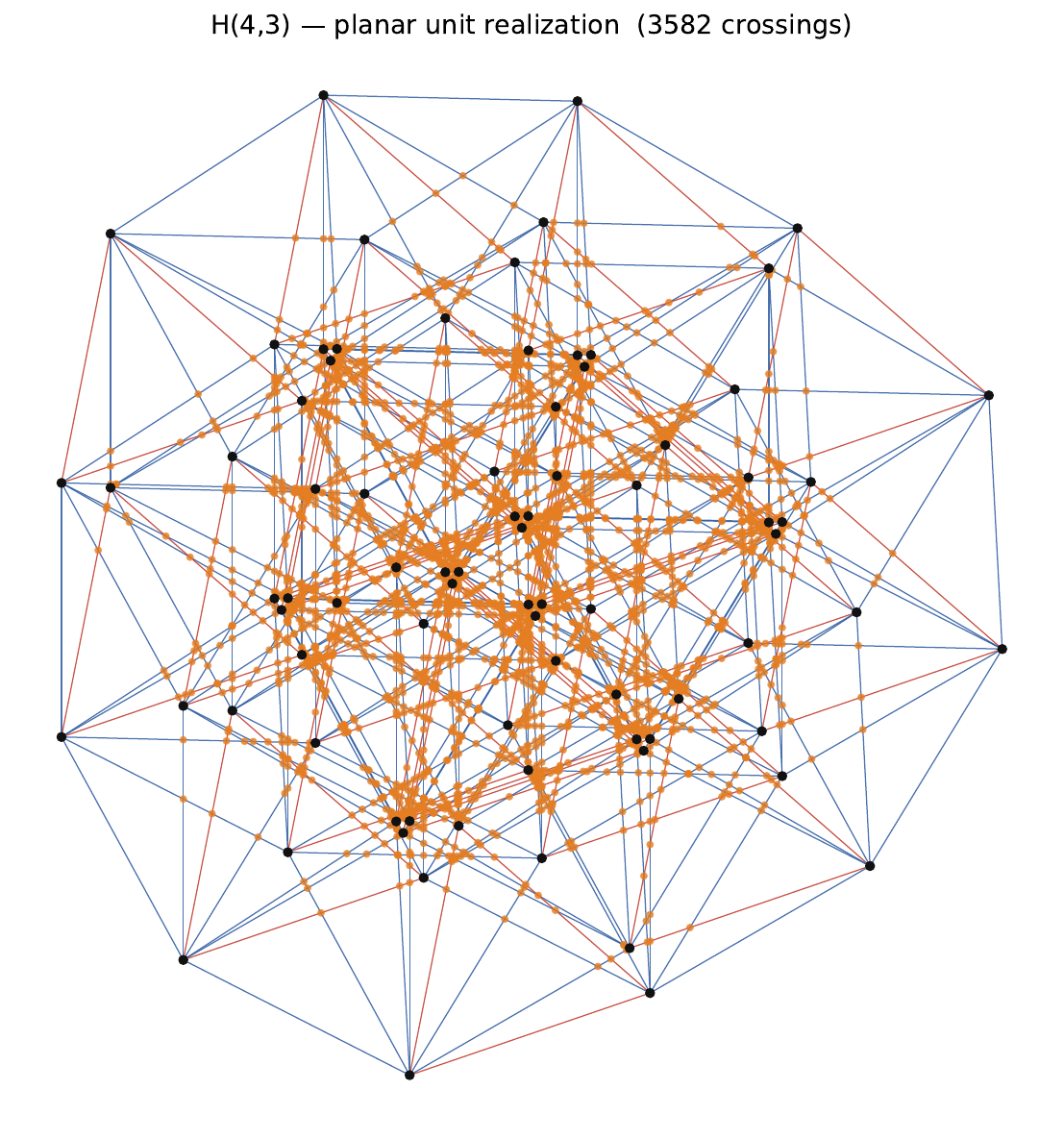}
  \caption{$H(4,3)=K_3^{\bx4}$ ($81$ vertices, $324$ unit edges). The bundle edges of the
  outermost triangle factor are drawn in red; removing them leaves three copies of
  Figure~\ref{fig:g33}.}\label{fig:g43}
\end{figure}

\begin{figure}[tp]\centering
  \includegraphics[width=.6\linewidth]{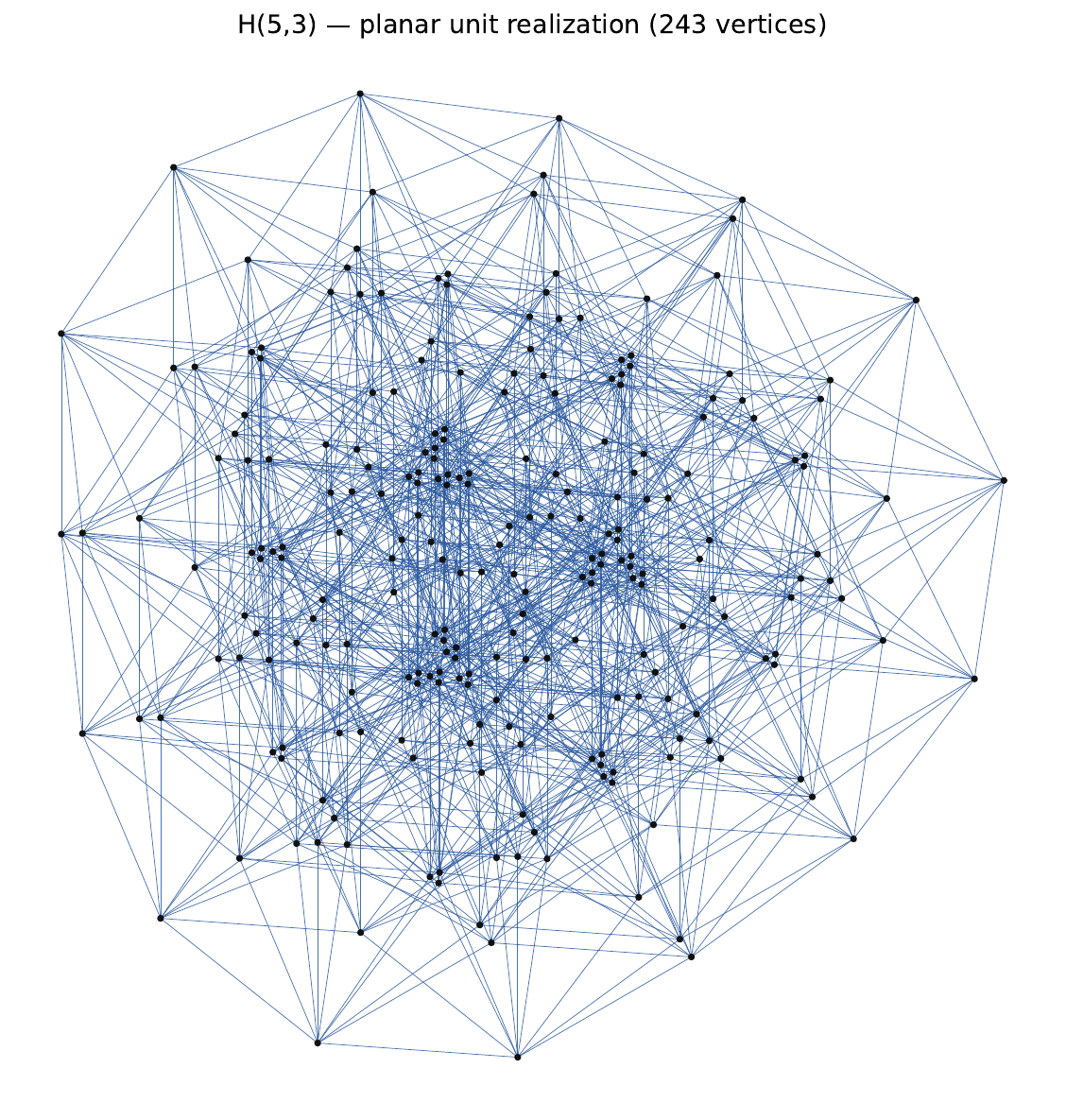}
  \caption{$H(5,3)=K_3^{\bx5}$ ($243$ vertices, $1215$ unit edges): the fifth member of
  the family, drawn as a faithful planar unit-distance graph by the Minkowski sum of five
  unit triangles.}\label{fig:g53}
\end{figure}

\begin{figure}[tp]\centering
  \includegraphics[width=.6\linewidth]{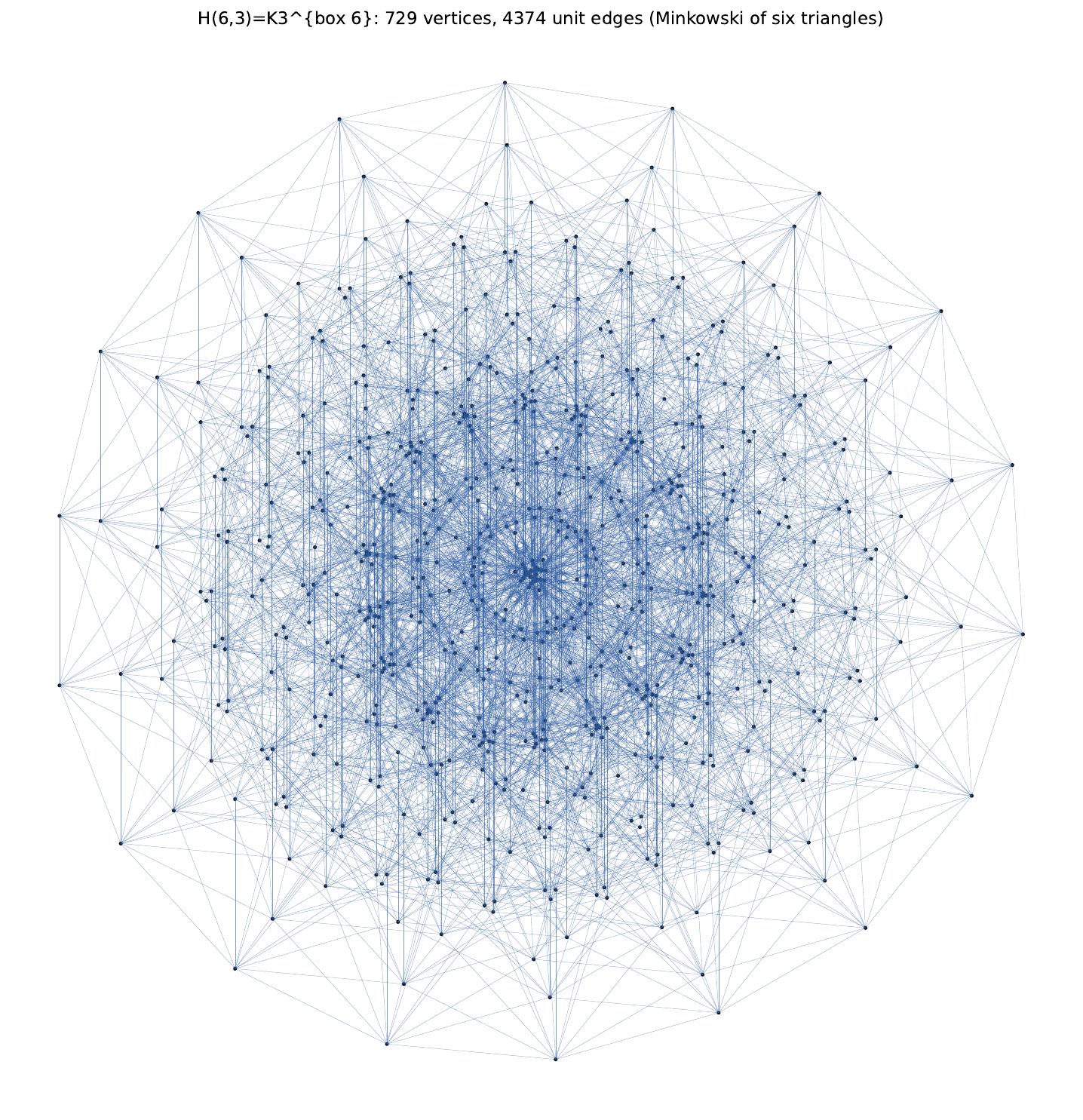}
  \caption{$H(6,3)=K_3^{\bx6}$: $729$ vertices and $4374$ unit edges. The Minkowski sum
  of six unit triangles fills the plane with the self-similar triangular pattern; even
  at this size every edge is exactly one unit long.}\label{fig:g63}
\end{figure}

\begin{figure}[tp]\centering
  \includegraphics[width=.48\linewidth]{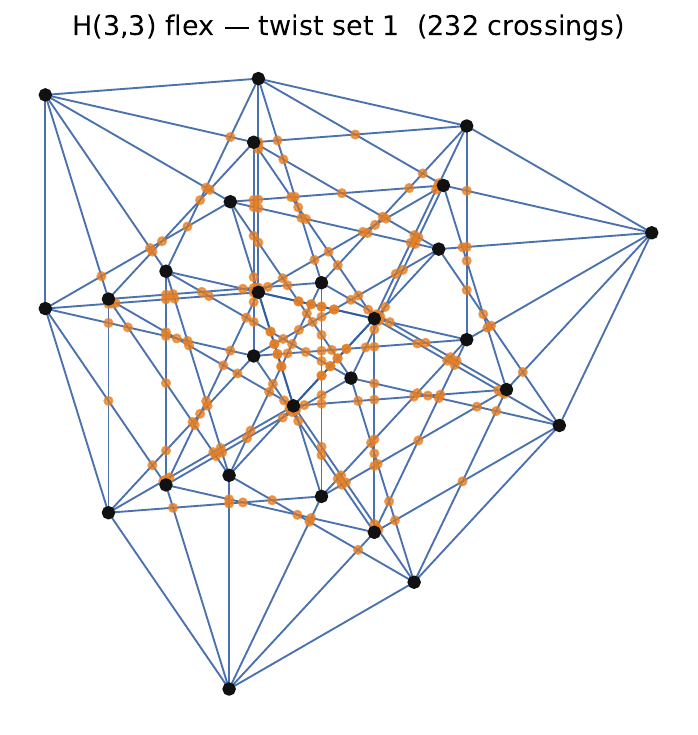}\hfill
  \includegraphics[width=.48\linewidth]{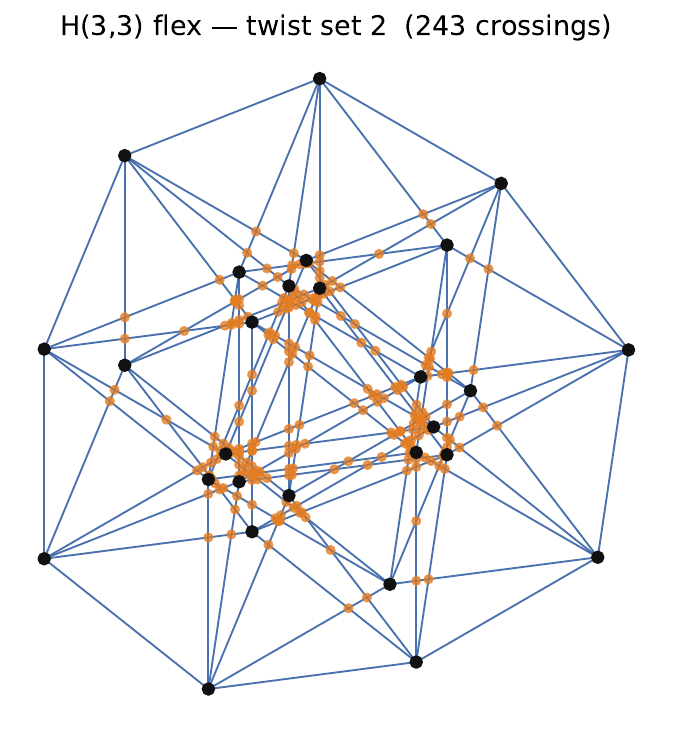}\\[6pt]
  \includegraphics[width=.48\linewidth]{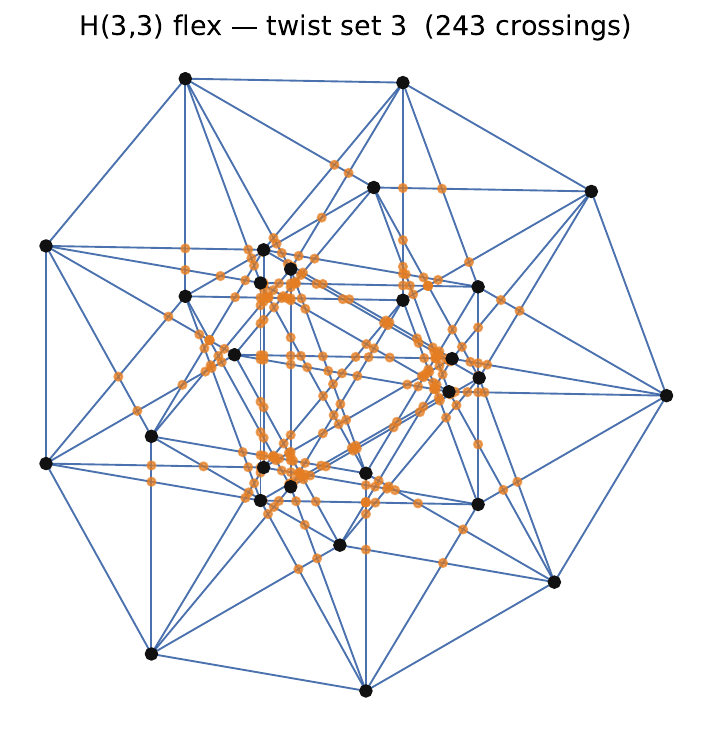}\hfill
  \includegraphics[width=.48\linewidth]{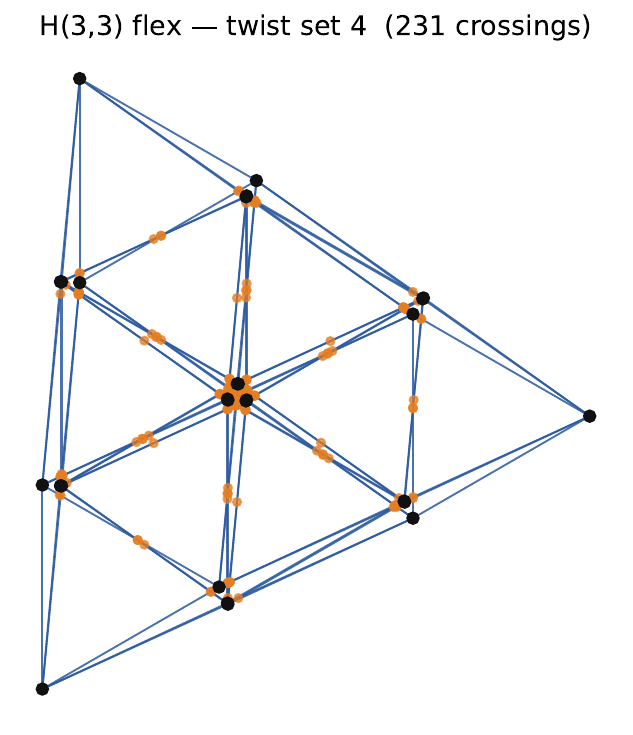}
  \caption{Four faithful unit-distance realizations of $H(3,3)$ along its two-parameter
  flex (the twist angles of two of the three triangles). All keep the $81$ edges unit;
  the coordinate field and the crossing pattern vary continuously along the flex.}
  \label{fig:flexgallery}
\end{figure}

\clearpage
\section{The recursive method for the crossing number}\label{sec:crossing}
Being a planar unit-distance graph, each $H(d,3)$ carries two crossing invariants:
the ordinary crossing number $\cro$~\cite{Guy}, and the \emph{unit-distance crossing number}
$\udcr$, the least number of crossings over faithful unit-distance realizations. We make the
latter precise. A \emph{strict faithful unit drawing} is an injective map $p:V\to\R^2$ with
$\|p(u)-p(v)\|=1$ iff $uv\in E$, edges drawn as the closed segments $p(u)p(v)$, in \emph{general
position}: no vertex lies in the interior of a non-incident edge, no two independent edges
overlap in positive length, and no point lies on three edges at once. A \emph{crossing} is a
transversal intersection of the interiors of two independent edges, counted once per pair; write
$\mathrm{cr}_u(D)$ for the number of such crossing pairs in a drawing $D$, and
$\udcr(G)=\min_D \mathrm{cr}_u(D)$ over all strict faithful unit drawings. For the graphs
$H(d,3)$ this minimum is over a concrete, finite-dimensional set rather than over all abstract
drawings: by the structure Theorem~\ref{thm:allmink} (proved below), \emph{in the plane there is
no faithful unit-distance representation of $H(d,3)$ that is not a Minkowski sum}. The strict
faithful unit drawings are therefore exactly the general-position points of the $(d-1)$-parameter
Minkowski twist family, and $\udcr(H(d,3))$ is a minimum of $\mathrm{cr}_u$ over that family
(Corollary~\ref{cor:udcrmink}). A unit-distance drawing is in particular a straight-line
(rectilinear) drawing with all edges of equal length, so it refines the \emph{rectilinear}
crossing number $\mathrm{rcr}$~\cite{DujmovicLaRose,BienstockDean}:
\[
  \cro(G)\ \le\ \mathrm{rcr}(G)\ \le\ \udcr(G).
\]
In the taxonomy of crossing-number variants surveyed by Schaefer~\cite{SchaeferSurvey} (rectilinear, geometric, constrained, local, degenerate, monotone, convex, and many more), the unit-distance crossing number does not appear; it is a new
\emph{geometric-constraint} variant, namely the rectilinear crossing number under the
unit-length constraint. The rest of this section is the technical
core of the paper: a \emph{recursive, combinatorial method} that reduces the crossing count of an
instance on $n$ vertices to that of three instances on $n/3$ vertices plus three geometric
sums. The method is not ad hoc: it is the Cartesian-product decomposition
$H(d,3)=H(d-1,3)\bx K_3$, which splits the vertices into three copies of the smaller graph
(the factor $K_3$) and the edges into copy edges and bundle edges; counting crossings by these
classes yields a divide-and-conquer recurrence $c(n)=3\,c(n/3)+g(n)$, solved by the master
theorem. The ``three copies plus bundles'' step is exactly the recursive construction of the
clique-based cube $K_q^{d}=K_q^{d-1}\bx K_q$ (LaForge's $K$-cube~\cite{LaForge}: $q$ copies of
$K_q^{d-1}$ joined across identical labels), so our method is the crossing-number reading of
that structure; the linear (one-page) layout used below is a book drawing in the sense of
Bernhart and Kainen~\cite{BernhartKainen}.
Combined with a one-page layout and a bisection bound it pins the ordinary crossing number
exactly, $\cro(H)=\Theta(n^{2})$, and gives $\udcr(H)=O(n^{2}\log n)$ for the unit-distance one.
Throughout we state magnitudes in the number of vertices $n=3^{d}$; exact counts are not
needed. Figure~\ref{fig:crvsudcr} shows the two crossing numbers of $H(3,3)$ side by side.

\begin{figure}[tp]\centering
  \includegraphics[width=\linewidth]{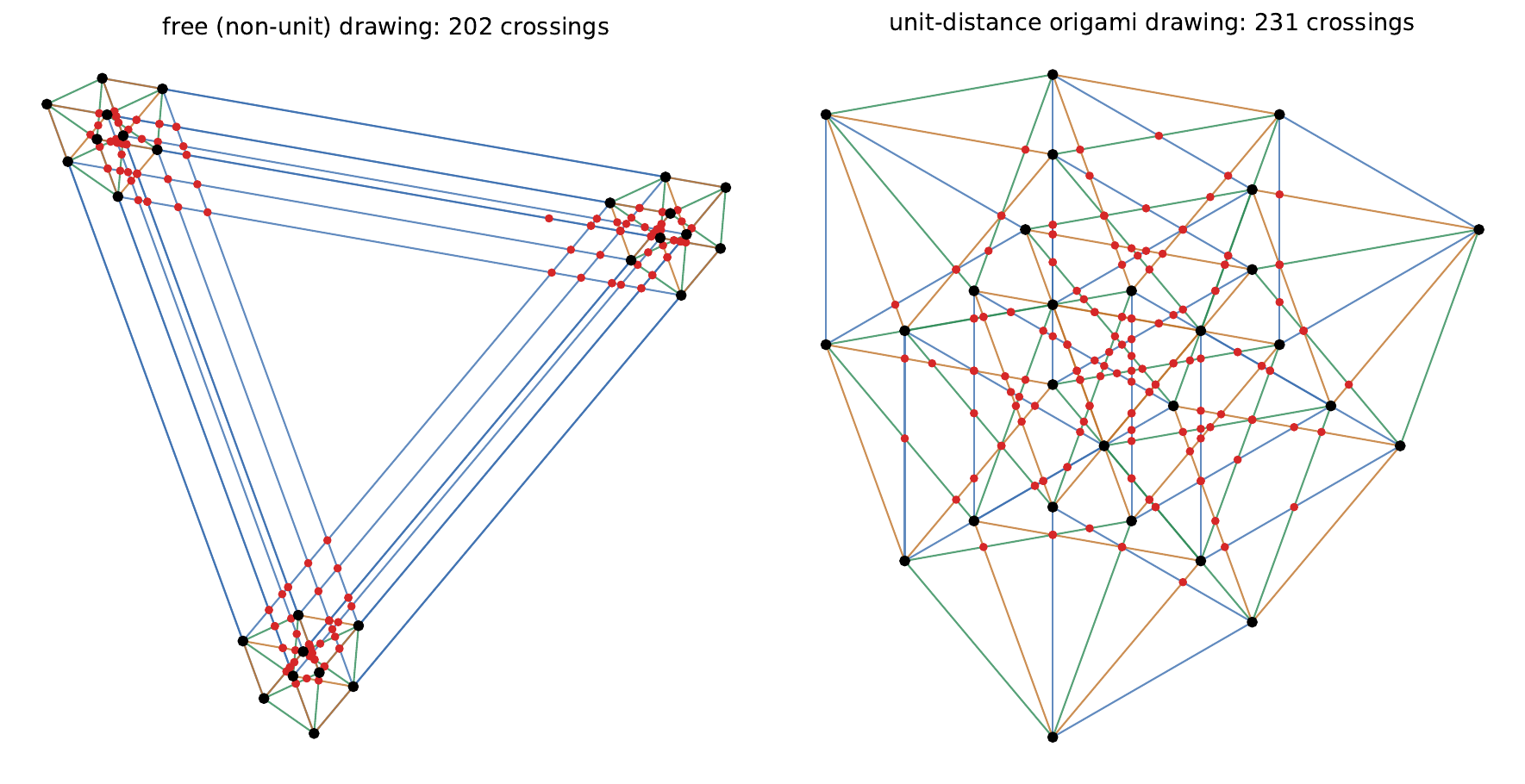}
  \caption{The two crossing numbers of $H(3,3)$ side by side. Left: a free (non-unit)
  separated drawing; right: the unit-distance drawing at the origami reference base
  (twists $0^\circ,20^\circ,40^\circ$, $\Q(\sin20^\circ)$; the same drawing as
  Figure~\ref{fig:h33origami}). Every edge on the right must be a unit segment, so it is more
  crowded: the unit drawing carries more crossings than the free one on the same vertices, the gap that grows to a factor $\Theta(\log n)$.}
  \label{fig:crvsudcr}
\end{figure}

\subsection{The recursive drawing}\label{ss:recdraw}
Fix the drawing $D_{d-1}$ of $H(d-1,3)$ from \eqref{eq:immersion}, and add one more
unit triangle $T_d$ with vertices $t_0,t_1,t_2\in\R^2$. Define the drawing $D_d$ of
$H(d,3)$ by
\[
  D_d(a_1,\dots,a_{d-1},k)\ =\ D_{d-1}(a_1,\dots,a_{d-1})\ +\ t_k .
\]
Two structural facts follow at once. \emph{(i)} For each fixed $k$ the map
$x\mapsto x+t_k$ is a \emph{translation}, so $D_d$ contains three congruent copies
$D_{d-1}+t_0,\ D_{d-1}+t_1,\ D_{d-1}+t_2$ of the previous drawing. \emph{(ii)} The edges
of $H(d,3)$ split into \emph{copy edges} (one of the first $d-1$ coordinates changes),
each lying inside a single copy, and \emph{bundle edges} (the last coordinate changes),
which for each vertex $x$ of the smaller instance form a triangle $\beta_x$ on the three
points $D_{d-1}(x)+t_0,\,+t_1,\,+t_2$, a translate of $T_d$. There are $n/3$ bundle
triangles (Figures~\ref{fig:recstep} and~\ref{fig:bundles}).

\begin{figure}[t]\centering
  \includegraphics[width=.58\linewidth]{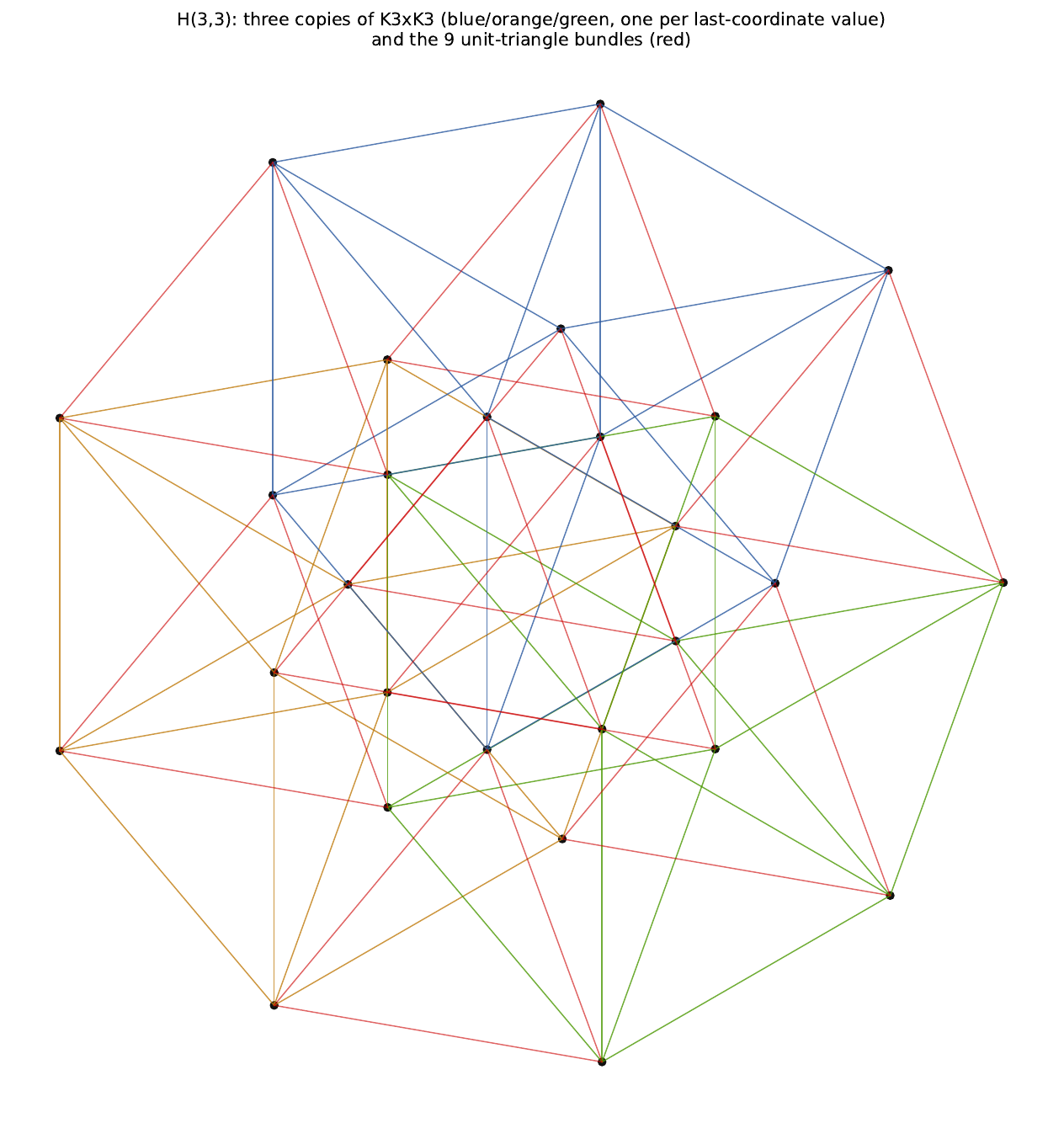}
  \caption{The recursive structure of $H(3,3)$ made explicit: the three copies of
  $K_3\bx K_3$ (blue, orange, green, one per value of the last coordinate) and the
  $3^{2}=9$ unit-triangle bundles (red) joining corresponding vertices across the
  copies.}\label{fig:bundles}
\end{figure}

\subsection{The four combinatorial types of crossing}\label{ss:decomp}
Every crossing of $D_d$ has exactly one of four types, by the classes of the two edges
and whether they lie in the same copy:
\begin{equation}\label{eq:decomp}
  c(n)=\underbrace{3\,c(n/3)}_{\text{intra-copy}}
      +\underbrace{X}_{\text{copy--copy}}
      +\underbrace{Y}_{\text{bundle--bundle}}
      +\underbrace{Z}_{\text{copy--bundle}},
  \qquad g(n)=X+Y+Z.
\end{equation}
\emph{Intra-copy.} Two copy edges of one copy: as that copy is a translate of the smaller
drawing, these are precisely its crossings, i.e.\ $c(n/3)$ per copy and $3\,c(n/3)$ in all.
\emph{Copy--copy $X$.} Two copy edges in different copies: the copies are translates by a
unit vector $t_k-t_{k'}$, so this counts crossings between the smaller drawing and a
unit-shifted congruent copy of itself, over the three pairs. \emph{Bundle--bundle $Y$.} Two
bundle triangles $\beta_x,\beta_y$, each a translate of the fixed triangle $T_d$; a pair of
translates of a triangle cross in a number fixed by their offset, so $Y$ is a sum over the
$\binom{n/3}{2}$ base pairs. \emph{Copy--bundle $Z$.} A copy edge against a bundle edge.
The four-way split (Figure~\ref{fig:types}) expresses the crossings of the $n$-vertex drawing
through those of the $n/3$-vertex one plus three geometric sums, which makes the recursion
computable level by level.

\begin{figure}[t]\centering
  \includegraphics[width=.59\linewidth]{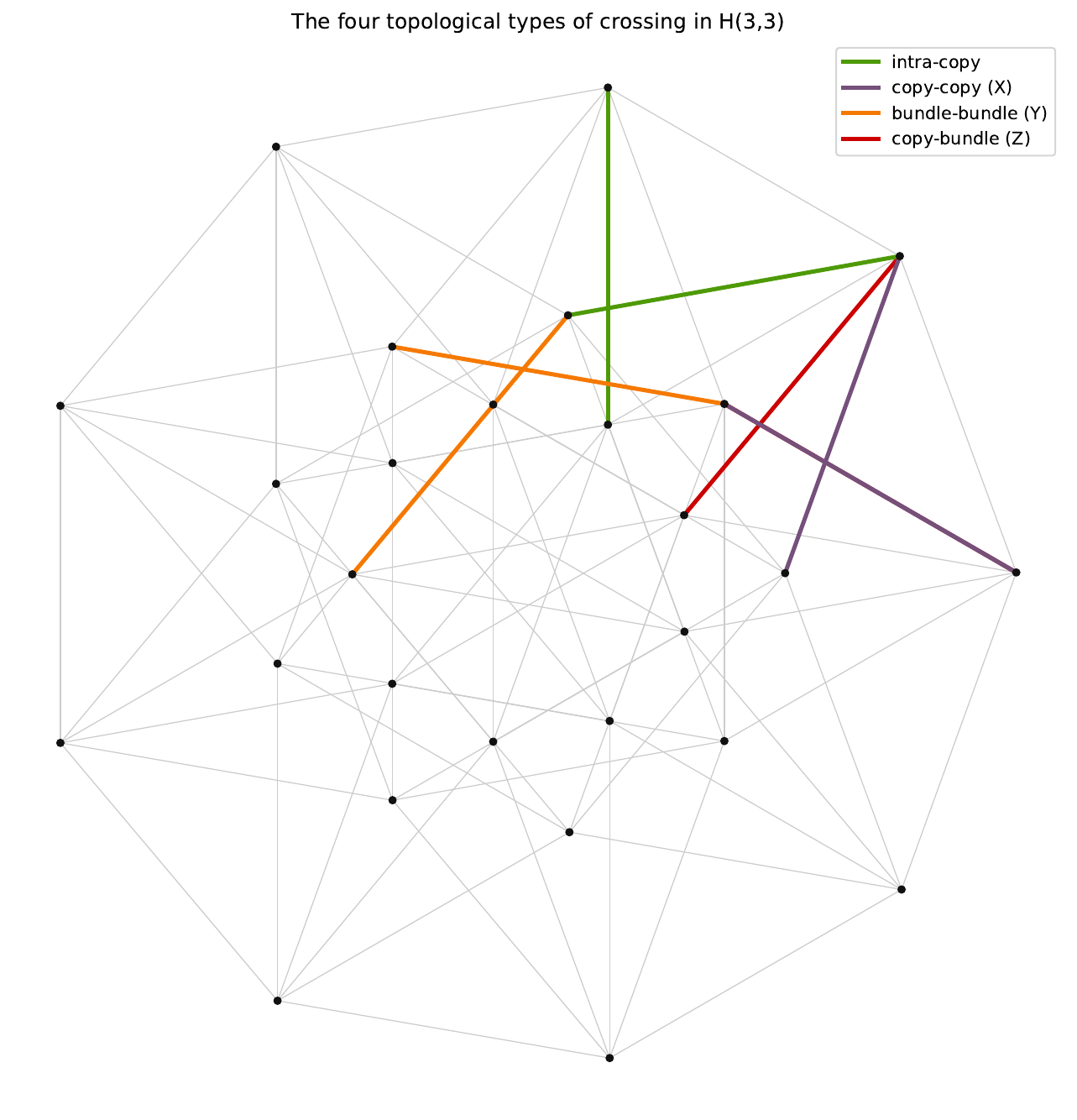}
  \caption{The four combinatorial types of crossing, one example of each highlighted on the
  $H(3,3)$ drawing (all other edges grey): intra-copy (green), copy--copy $X$ (purple),
  bundle--bundle $Y$ (orange), copy--bundle $Z$ (red). The recursion counts these four
  classes separately.}\label{fig:types}
\end{figure}

\begin{figure}[t]\centering
  \includegraphics[width=.57\linewidth]{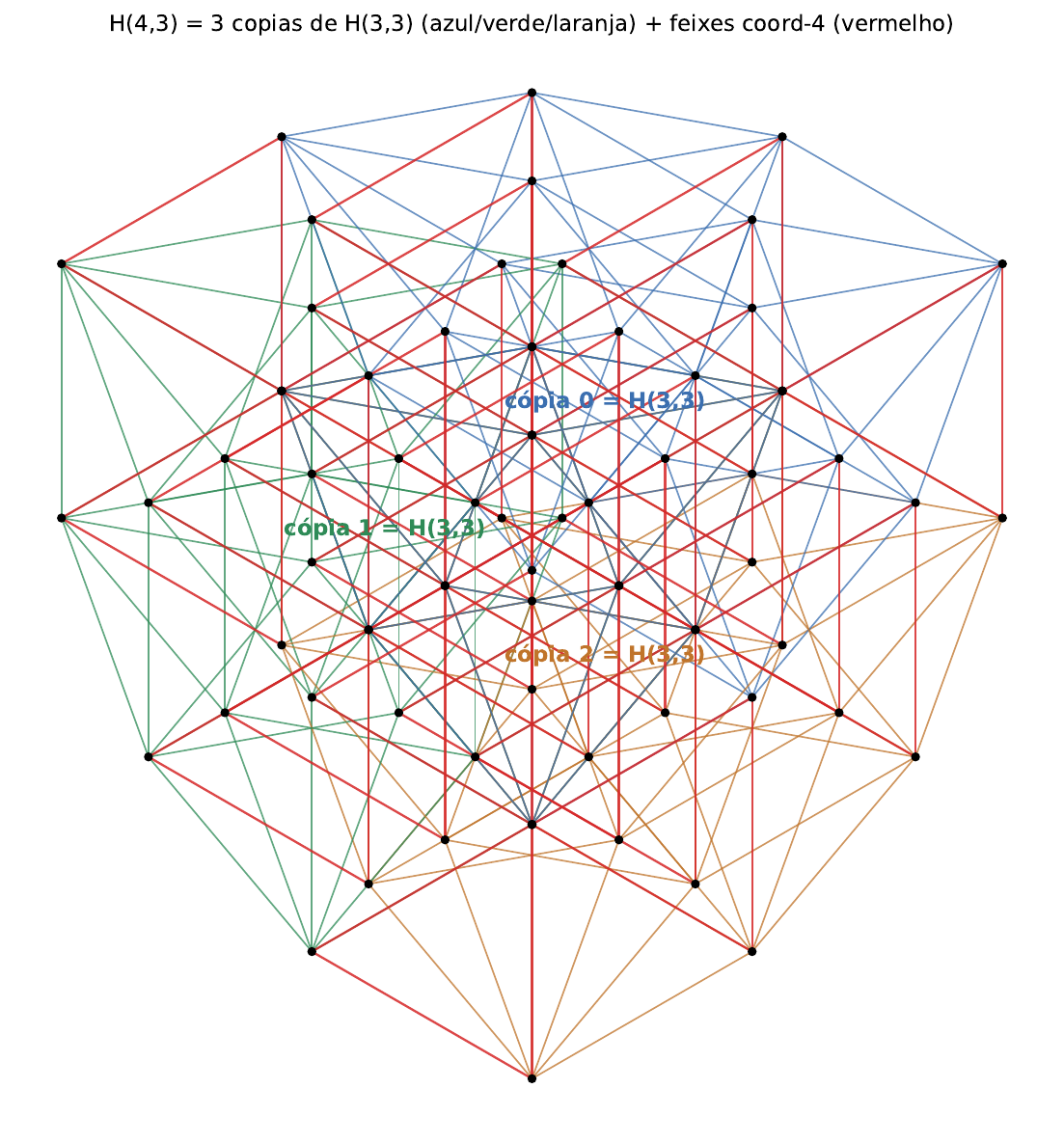}
  \caption{The unit recursion $H(4,3)=H(3,3)\bx K_3$: three congruent copies of the origami
  $H(3,3)$ drawing (blue, green, orange, one per value of the last coordinate) joined by the
  coordinate-$4$ bundles (red). The intra-copy crossings are $3\,c(H(3,3))$ exactly; the bundles
  add $g=X+Y+Z$.}\label{fig:unitrec}
\end{figure}

\paragraph{Worked example: $d=3\to d=4$.} The step from $H(3,3)$ to
$H(4,3)=H(3,3)\bx K_3$ (Figure~\ref{fig:unitrec}) makes the recurrence concrete. Adding a fourth triangle splits
$H(4,3)$ into three \emph{congruent} copies of the $H(3,3)$ drawing (each a translate of the
same drawing) plus the coordinate-$4$ bundles, so the intra-copy term is structurally exact,
\[
  \text{intra}\ =\ 3\,c(H(3,3)),
\]
independent of any count. Writing $g=X+Y+Z$ for the bundle crossings (copy--copy, bundle--
bundle, copy--bundle), a good drawing at the origami base gives
$c(H(4,3))=3\,c(H(3,3))+g\le3546$, with the bundle--bundle term $Y$ of lowest order and $X,Z$
dominant. The same peeling takes $H(4,3)$ to three copies inside $H(5,3)$, and so on down the
tower, giving the upper bounds
\[
  \udcr(H(d,3))\ \le\ 9,\ 231,\ 3546,\ 44028,\ 490935\qquad(d=2,\dots,6).
\]

One caveat about the exact symmetric point. The \emph{exact} origami point is a faithful
realization (no coincident vertices, no extra unit distance), but it is \emph{not} a good drawing:
several vertices lie in the interior of non-incident edges ($18$ such incidences for $H(3,3)$), as
constructible points tend to do, sitting on degeneracy loci. The
individual counts are therefore taken, as is standard for the crossing number, in a \emph{good
drawing} infinitesimally close to it (no vertex on a non-incident edge, no three edges
concurrent), which is why we state them as upper bounds; at the exact point the count is not
well defined. That the crossing-optimal realization is exactly where edges concur is the same
phenomenon from the other side: constructibility and crossing-optimality meet on the
degeneracy locus.

\paragraph{Anatomy of the crossings.} The four-type partition \eqref{eq:decomp} controls
the growth. The bundle--bundle term $Y$ is of lowest order (two translated triangles meet in $O(1)$ points, so $Y=O(n^{2})$), while the copy--copy
term $X$ and the copy--bundle term $Z$ dominate. Bounding these two crudely by the number of
edge pairs gives $g=X+Y+Z=O(n^{2}\log^{2}n)$; the observed order is one logarithm smaller,
$g(n)=O(n^{2}\log n)$, the single $\log$ coming from the depth of the recursion, and it is
this factor that separates the unit count from the ordinary one.

\subsection{The recursion and the unit upper bound}
Let $c(n)$ denote the crossing count of the Minkowski drawing $D$ of the instance on $n$
vertices (built from a twist vector; the count depends on it, but the bounds below do not).
Collecting \eqref{eq:decomp}, the drawing on $n$ vertices is three translated copies of the
drawing on $n/3$ vertices plus the bundle sums:
\begin{equation}\label{eq:rec}
  c(n)=3\,c(n/3)+g(n),\qquad \udcr(H)\le c(n).
\end{equation}

\begin{proposition}\label{prop:rec}
The drawing $D$ is a faithful unit-distance drawing whose crossing number satisfies
\eqref{eq:rec}, with $g(n)=X+Y+Z$ the geometric bundle sums. Hence $\udcr(H)\le c(n)$, and
solving the recursion gives $\udcr(H)=O(n^{2}\log^{2}n)$ from the crude bound on $g$; this is
sharpened to $O(n^{2}\log n)$ in Theorem~\ref{thm:udcr}.
\end{proposition}

\begin{proof}
By construction $D$ places the vertices at Minkowski sums of unit triangles and draws every
edge as a unit segment (Proposition~\ref{prop:mink}), so it is a faithful unit drawing and
$\udcr(H)\le c(n)$. The four-type partition \eqref{eq:decomp} classifies each crossing once;
the intra-copy class is $3\,c(n/3)$ because each copy is a translate of the smaller drawing,
which gives \eqref{eq:rec}. Bounding $g(n)=O(n^{2}\log^{2}n)$ by edge pairs, the recursion
$c(n)=3\,c(n/3)+O(n^{2}\log^{2}n)$ is dominated by its top level (the branching factor $3$
against the size drop to $n/3$ costs only $O(n)$), so $c(n)=O(n^{2}\log^{2}n)$; the sharper
$O(n^{2}\log n)$ is Theorem~\ref{thm:udcr}.
\end{proof}

Only the base value $c=9$ at $n=9$ is constant across the Minkowski family; it is the exact value
$\udcr(K_3\bx K_3)=9$ (Proposition~\ref{prop:base9}). Because the count is otherwise drawing-dependent, we fix an \emph{exact} reference: the
\emph{origami} realization, with all $d$ triangles twisted by multiples of $20^\circ=60^\circ/3$
(coordinates in the trisection field $\Q(\sin20^\circ)$). This is not an arbitrary choice: it
is the crossing-minimizing constructible configuration (\S\ref{ss:mintier}), so it gives both a
canonical algebraic reference and the tightest simple upper bound,
\[
  \udcr(H(d,3))\ \le\ 9,\ 231,\ 3546,\ 44028,\ 490935\qquad(d=2,\dots,6),
\]
with exact coordinates rather than a decimal twist. The order $\udcr(H)=O(n^{2}\log n)$
(Theorem~\ref{thm:udcr}) is independent of this base, holding for every generic realization;
the origami base fixes only the exact small-$n$ values.

\begin{figure}[t]\centering
  \includegraphics[width=.6\linewidth]{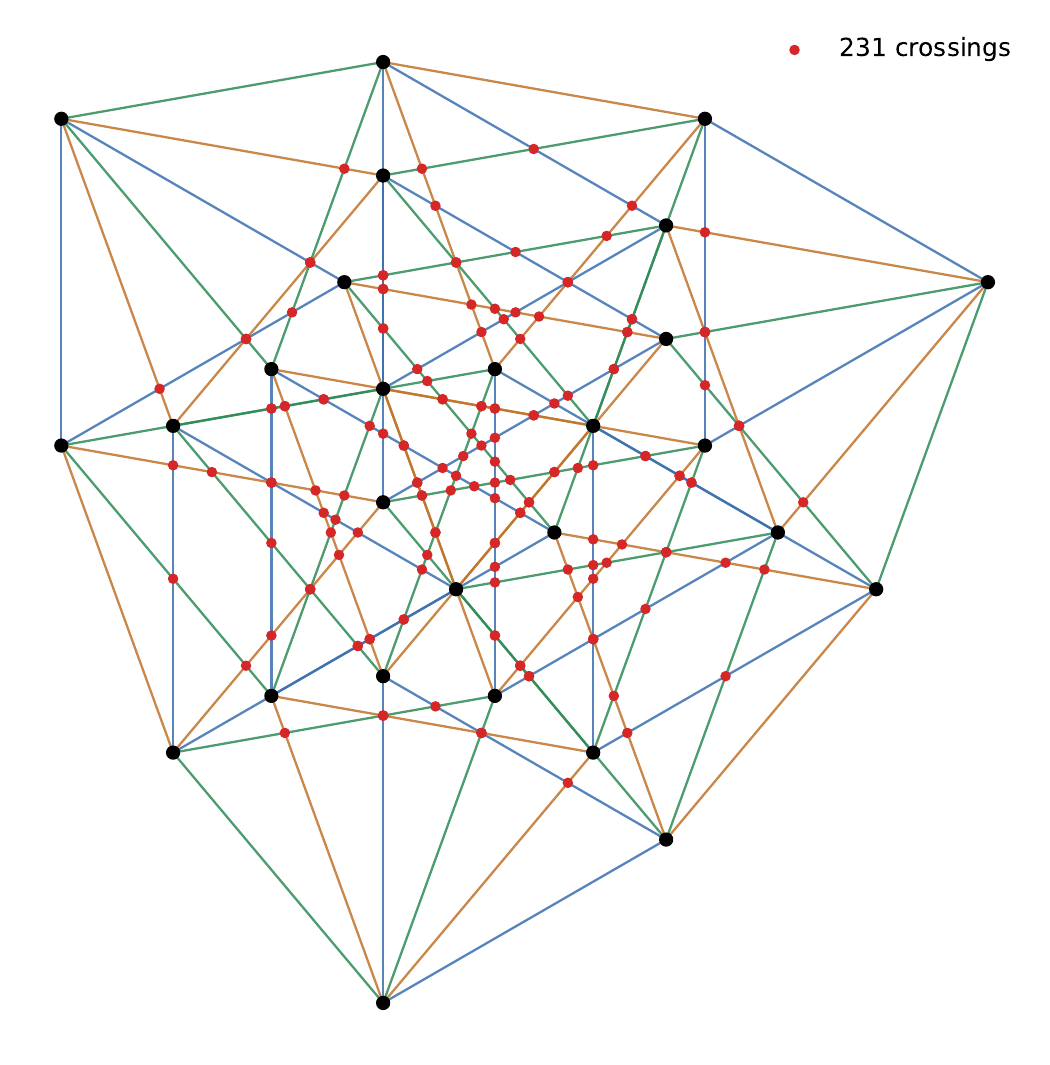}
  \caption{The canonical reference drawing of $H(3,3)=K_3^{\bx3}$: the origami realization, the
  Minkowski sum of three unit triangles (circumradius $1/\sqrt3$) twisted by $0^\circ,20^\circ,
  40^\circ$, with vertices
  $T_j=\tfrac1{\sqrt3}\{(\cos(\varphi_j{+}120^\circ k),\sin(\varphi_j{+}120^\circ k))\}_{k=0,1,2}$,
  $\varphi_j\in\{0^\circ,20^\circ,40^\circ\}$, coordinates in $\Q(\sin20^\circ)$ (degree $3$).
  Edges are colored by the coordinate they change; the $231$ red dots are the crossings; $231$ is the conjectured value of $\udcr(H(3,3))$ (Remark~\ref{rem:231}), against
  $\cro(H(3,3))\le99$ for a free (non-unit) drawing
  (Figure~\ref{fig:h33free}).}\label{fig:h33origami}
\end{figure}

\begin{figure}[t]\centering
  \includegraphics[width=.56\linewidth]{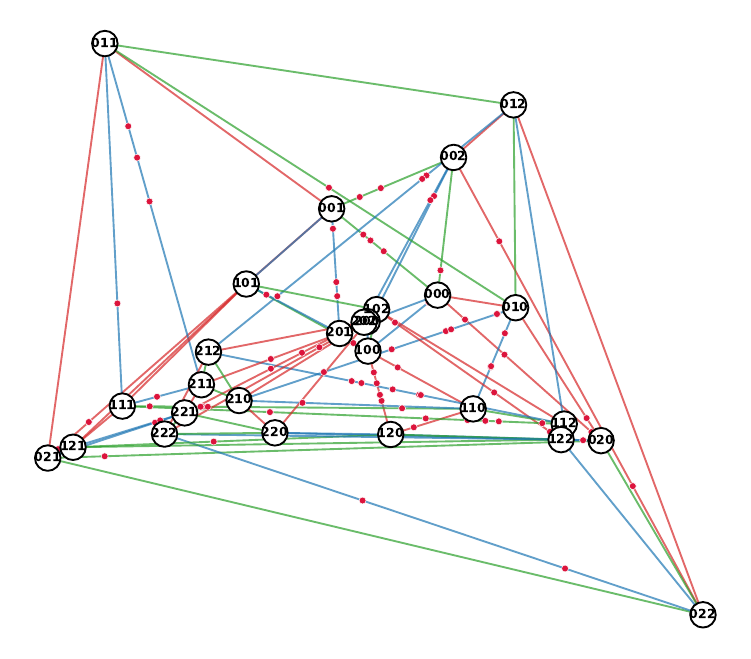}
  \caption{A free (non-unit) rectilinear drawing of $H(3,3)$ with $99$ crossings, found by local
  optimization of the $27$ vertex positions and verified in general position (no vertex on a
  non-incident edge, no three edges concurrent); edges colored by the coordinate they change, red
  dots the crossings, vertices labelled $abc$. This improves on the separated drawing and gives
  $\cro(H(3,3))\le99$ (coordinates in Appendix~\ref{app:free}); the exact value remains open. In
  contrast to the unit realization (Figure~\ref{fig:h33origami}, $231$ crossings), dropping the
  unit-length constraint allows the vertices to be clustered so that far fewer edges cross.}
  \label{fig:h33free}
\end{figure}

\begin{figure}[t]\centering
  \includegraphics[width=.55\linewidth]{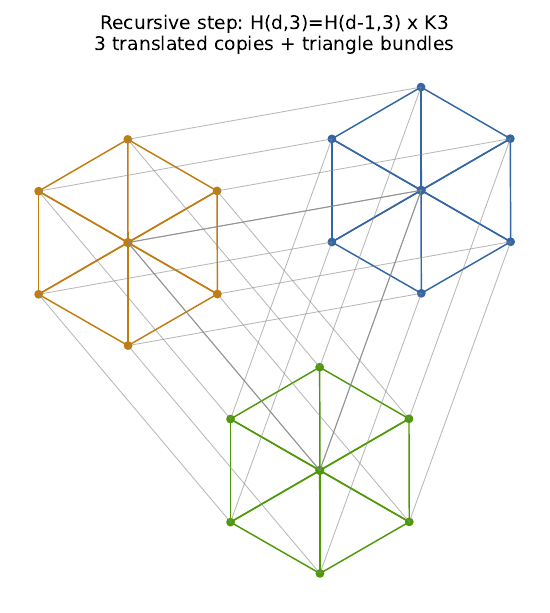}\\[8pt]
  \includegraphics[width=.72\linewidth]{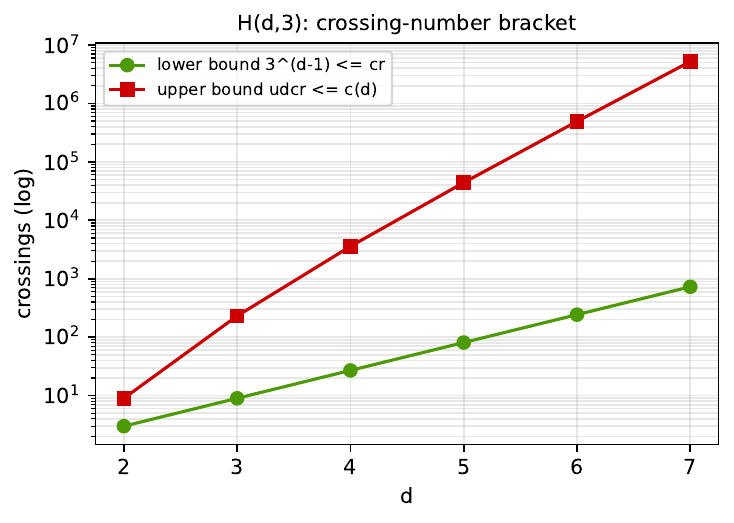}
  \caption{Left: the recursive step $H(d,3)=H(d-1,3)\bx K_3$, three translated copies
  plus triangle bundles. Right: the crossing bracket
  the band $\Omega(n)\le\cro(H)\le\udcr(H)\le c(n)=O(n^2\log n)$ (log scale).}
  \label{fig:recstep}
\end{figure}

\subsection{The lower bound: vertex-disjoint copies}
The same recursive structure yields a lower bound by the same three-copy split.

\begin{theorem}\label{thm:lower}
$\cro(H)=\Omega(n)$: the ordinary crossing number of the instance on $n$ vertices is at
least $n/3$.
\end{theorem}

\begin{proof}
Write $T(n)$ for the crossing number. Fixing one coordinate to each of its three values
exhibits three vertex-disjoint copies of the instance on $n/3$ vertices. In \emph{any}
drawing, the restriction to each copy is a drawing of that instance, so contributes at
least $T(n/3)$ crossings; and since the copies are vertex-disjoint, a crossing internal to
one copy involves no edge of another, so the three tallies are disjoint and add. Hence
$T(n)\ge 3\,T(n/3)$, with $T(9)=\cro(K_3\bx K_3)=3$~\cite{HKS}. Unrolling gives $T(n)\ge n/3$, i.e.\
$\cro(H)=\Omega(n)$.
\end{proof}

At the base the two invariants already separate: $\cro(K_3\bx K_3)=3$, whereas every faithful
unit realization \emph{in the Minkowski family} has exactly nine crossings. In fact a dense
sweep of the entire one-parameter flex ($24000$ faithful values of the relative twist) returns
the count $9$ at \emph{every} point: on the Minkowski locus the crossing number is \emph{constant},
not merely minimized. Hence $\udcr(K_3\bx K_3)\le9$ (Figure~\ref{fig:k3k3udcr}). Moreover
$\cro(K_3\bx K_3)=\mathrm{rcr}(K_3\bx K_3)=3$, so the whole gap from $3$ to $9$ is forced by the
unit-length-plus-faithfulness constraint, not by rectilinearity. In fact equality holds, and the
drawing-dependence disappears: \emph{every} faithful realization is a Minkowski sum.

\begin{proposition}[The base case is rigid: $\udcr(K_3\bx K_3)=9$]\label{prop:base9}
Every faithful unit-distance realization of $K_3\bx K_3$ is congruent to a Minkowski sum
$p(a,b)=u_a+v_b$ of two unit equilateral triangles. Consequently $\udcr(K_3\bx K_3)=9$.
\end{proposition}

\begin{proof}
For each $b$ the points $p(0,b),p(1,b),p(2,b)$ are pairwise at distance $1$, hence form a unit
equilateral triangle (a \emph{row} $R_b$); likewise each \emph{column}
$\{p(a,0),p(a,1),p(a,2)\}$ is a unit equilateral triangle. A nondegenerate equilateral triangle
has an \emph{orientation} $\varepsilon\in\{1,2\}$: the cyclic sense $0\!\to\!1\!\to\!2$ is
counterclockwise ($\varepsilon=1$) or clockwise ($\varepsilon=2$), i.e.\ the sign of its oriented
area. Identifying the plane with $\mathbb C$, a triple $w_0,w_1,w_2$ is equilateral with
orientation $\varepsilon$ iff $w_0+\omega^{\varepsilon}w_1+\omega^{2\varepsilon}w_2=0$, where
$\omega=e^{2\pi i/3}$.

\emph{The orientation is uniform.} Let $\varepsilon_b,\eta_a\in\{1,2\}$ be the orientations of the
rows and columns, and set $z_{a,b}=p(a,b)\in\mathbb C$. Once the orientations are fixed, the six
equilateral conditions $\sum_{a}\omega^{\varepsilon_b a}z_{a,b}=0$ ($b=0,1,2$) and
$\sum_{b}\omega^{\eta_a b}z_{a,b}=0$ ($a=0,1,2$) are \emph{linear} over $\mathbb Q(\omega)$, so the
realizations with a given orientation pattern form a linear subspace $K_{\varepsilon,\eta}\subseteq
\mathbb C^{9}$. An injective realization is a point of $K_{\varepsilon,\eta}$ lying off every
\emph{coincidence hyperplane} $\{z_{a,b}=z_{a',b'}\}$, $(a,b)\neq(a',b')$; and a linear subspace
contained in a finite union of hyperplanes is contained in one of them. It therefore suffices to
show that for every \emph{non-uniform} pattern $K_{\varepsilon,\eta}$ lies in a coincidence
hyperplane: then that pattern admits no injective realization. By the symmetry group (permuting
the three rows and the three columns, transposing the array, and the simultaneous orientation flip
$\varepsilon,\eta\mapsto 3-\varepsilon,3-\eta$), every non-uniform row pattern is equivalent to
$\varepsilon=(1,1,2)$. There rows $0,1$ take the form $z_{a,b}=g_b+c_b\omega^{a}$ and row $2$ the
form $z_{a,2}=g_2+d_2\omega^{2a}$ ($c_0,c_1,d_2\neq0$ by nondegeneracy); substituting into the
column equations shows the three column orientations cannot all coincide, and solving the resulting
linear system forces a vertex coincidence: for equally oriented columns it collapses the third
row, $z_{0,2}=z_{1,2}=z_{2,2}$, impossible for a unit triangle, and each remaining column pattern
forces an analogous coincidence (the explicit elimination for every pattern is
Appendix~\ref{app:uniform}). Hence no faithful realization has a non-uniform pattern: all rows share
one orientation and, by transposition, all columns share one orientation.

\emph{Uniform orientation $\Rightarrow$ Minkowski.} Two equally-oriented unit triangles $R_0,R_b$ are
related by a unique orientation-preserving isometry $g$. If $g$ were a rotation by $\varphi\neq0$
about a center $c$, the displacement of vertex $a$ would have length $2\sin(\varphi/2)\,|p(a,0)-c|$;
the column edges force this to equal $1$ for all three $a$, so the $|p(a,0)-c|$ are equal, $c$ is the
circumcenter, and $\varphi=\pm120^{\circ}$. But a $\pm120^{\circ}$ rotation about the circumcenter fixes the
vertex set of $R_0$, so $R_b$ would reuse its three points, contradicting injectivity. Thus $g$ is a
translation: $R_b=R_0+t_b$ with $|t_b|=1$. Writing $u_a=p(a,0)$ and $v_b=t_b$ gives $p(a,b)=u_a+v_b$,
and the column triangle $u_a+\{0,t_1,t_2\}$ being unit forces $\{0,t_1,t_2\}$ to be a unit triangle.
This is exactly a Minkowski sum of two unit triangles.

\emph{Count.} Up to congruence the faithful Minkowski realizations form the one-parameter twist
family $\varphi\in[0,2\pi)$. The crossing count of a straight-line drawing changes only when a
crossing is created or destroyed, which forces three vertices to become collinear. Among the
$\binom{9}{3}=84$ vertex triples, exactly three are collinear for \emph{all} $\varphi$ (a structural
degeneracy, never a transition); the remaining $81$ collinearity conditions have, in total, exactly
$6$ roots in $[0,2\pi)$. These six \emph{critical angles} cut the circle into six open cells, and on
each cell the drawing is faithful with exactly $9$ crossings (checked at a representative of every
cell in exact sign arithmetic, Figure~\ref{fig:k3k3udcr}). Being constant on each cell, the count is
$9$ at every faithful realization, so $\udcr(K_3\bx K_3)=9$.
\end{proof}

This settles the base case; and it is in fact the base of an induction that resolves the whole
family. The question of whether the Minkowski flex exhausts the faithful realizations of $H(d,3)$,
left open above, has an affirmative answer for every $d$.

\begin{theorem}[The Minkowski flex is the whole \emph{planar} realization variety]\label{thm:allmink}
For every $d\ge1$, every faithful unit-distance realization of $H(d,3)=K_3^{\bx d}$
\emph{in the plane} is a Minkowski sum $p(a_1,\dots,a_d)=\sum_{i=1}^{d}u_i[a_i]$ of $d$ unit
equilateral triangles. Thus the $(d-1)$-parameter Minkowski twist family is the entire
\emph{planar} realization variety: $H(d,3)$ has a $(d-1)$-parameter continuum of faithful planar
realizations and no others. (Planarity is essential: in $\R^3$ the Minkowski family is a proper
subvariety of a strictly larger faithful variety; see the spatial remark in Section~\ref{sec:space}.)
\end{theorem}

\begin{proof}
Induction on $d$. For $d=1$, $H(1,3)=K_3$ is a single unit triangle; $d=2$ is
Proposition~\ref{prop:base9}. Let $d\ge3$ and assume the statement for $d-1$. Fix a faithful
realization $p$. For each value $c$ of the last coordinate, the slice
$\{p(a_1,\dots,a_{d-1},c)\}$ is a faithful realization of $H(d-1,3)$ (its edges change one of the
first $d-1$ coordinates, hence are unit, and faithfulness is inherited), so by induction
$p(\mathbf a,c)=\sum_{i<d}u^{c}_i[a_i]$. Slicing instead on the second coordinate gives, likewise,
$p(\mathbf a)=\sum_{i\neq2}x^{a_2}_i[a_i]$. The first difference of $p$ in $a_1$ is therefore both
$u^{a_d}_1[a_1]-u^{a_d}_1[0]$ (independent of $a_2,\dots,a_{d-1}$) and
$x^{a_2}_1[a_1]-x^{a_2}_1[0]$ (independent of $a_3,\dots,a_d$); being independent of $a_2$ and of
$a_d$ at once, it depends on $a_1$ alone. By the symmetry of the $d$ coordinates, the first
difference of $p$ in each coordinate depends only on that coordinate, so $p$ is additive:
$p(\mathbf a)=\sum_i U_i[a_i]+\mathrm{const}$, with each $U_i$ a unit triangle (the $a_i$-line is
unit). This is a Minkowski sum.
\end{proof}

\begin{corollary}\label{cor:udcrmink}
$\udcr(H(d,3))$ is the minimum, over the $(d-1)$-parameter Minkowski family, of the crossing number
of a general-position unit drawing. For $d=2$ this minimum is the constant $9$. For $d=3$,
$\udcr(H(3,3))\le231$, by good drawings near the origami configuration.
\end{corollary}

\begin{remark}[The value $231$]\label{rem:231}
Every general-position Minkowski drawing of $H(3,3)$ that we sampled has $231$, $237$ or $243$
crossings (near the origami and compass configurations and the equal-spacing one); for instance
$400$ random twist pairs gave $231$ ($227$ times), $237$ ($142$) and $243$ ($31$). This supports
the conjecture $\udcr(H(3,3))=231$, but sampling a two-parameter continuum is not a proof. A proof
would follow, given Theorem~\ref{thm:allmink}, from a certified cell decomposition of the twist
plane: the count is constant off the finitely many curves where a vertex meets the line of an edge,
so it suffices to evaluate it exactly at one point of each cell. (Smaller counts occur only on the
degeneracy locus, where the drawing is not a good drawing.)
\end{remark}

\begin{figure}[htbp]\centering
  \includegraphics[width=.51\linewidth]{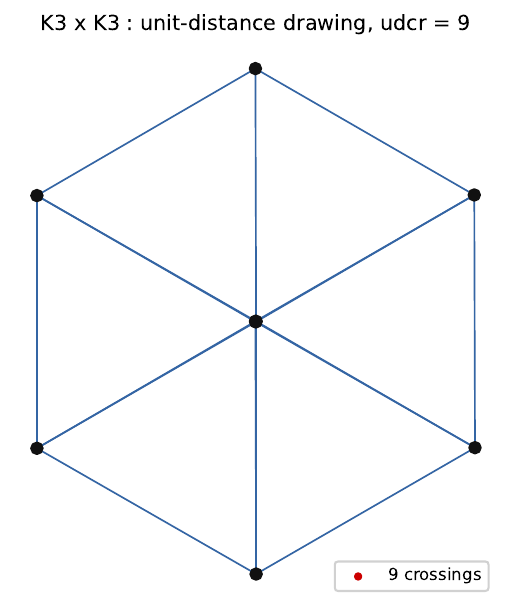}
  \caption{The unit-distance drawing of $K_3\bx K_3$ with its nine forced crossings:
  $\udcr(K_3\bx K_3)=9$ against $\cro(K_3\bx K_3)=3$.}\label{fig:k3k3udcr}
\end{figure}

\subsection{An unconditional $O(n^{2})$ bound}
For the \emph{ordinary} crossing number the unit constraint is dropped, and a one-page
(book) layout settles the order outright. Order the $n=3^{d}$ vertices on a line by their
base-$3$ value (equivalently, lexicographically), and draw every edge as an arc on one side;
call this layout $\Lambda$. Two edges cross iff their vertex intervals strictly interleave.
The \emph{span} of an edge is the distance between its endpoints.

\begin{lemma}\label{lem:span}
$\displaystyle\sum_{e}\operatorname{span}(e)=\tfrac23\,(n^{2}-n)$, where $\operatorname{span}(e)$
is the difference of the endpoint positions.
\end{lemma}
\begin{proof}
An edge changing coordinate $i$ (of weight $3^{\,d-i}$ in the base-$3$ value) joins two
vertices whose values differ by that weight, so within each of the $3^{\,d-1}$ triangles for
coordinate $i$ the three spans are $3^{\,d-i},3^{\,d-i},2\cdot3^{\,d-i}$, summing to
$4\cdot3^{\,d-i}$. Hence
$\sum_e\operatorname{span}(e)=\sum_{i=1}^{d}3^{\,d-1}\cdot4\cdot3^{\,d-i}
=4\cdot3^{\,2d-1}\sum_{i=1}^{d}3^{-i}=2\cdot3^{\,2d-1}(1-3^{-d})=\tfrac23(9^{d}-3^{d})
=\tfrac23(n^{2}-n)$.
\end{proof}

\begin{theorem}\label{thm:upper}
The one-page layout $\Lambda$ satisfies the recursion $X(n)=3\,X(n/3)+g(n)$ with the
\emph{exact} per-level count
\[
  g(n)=\tfrac79\,n^{2}-2\,n\log_3 n-\tfrac13\,n\ <\ \tfrac79\,n^{2},
\]
and therefore $\cro(H)\le X(n)<\tfrac76\,n^{2}=O(n^{2})$.
\end{theorem}
\begin{proof}
The $n$ vertices split into three consecutive blocks of $n/3$, each read on its own the layout
$\Lambda$ of the instance on $n/3$ vertices; the only inter-block edges are the $n$
coordinate-$1$ (``long'') edges (three per internal label $x\in\{0,\dots,n/3-1\}$, of types
$\{0,1\},\{1,2\},\{0,2\}$), so $X(n)=3\,X(n/3)+g(n)$ with $g$ the crossings meeting a long
edge; write $g=g_{\ell\ell}+g_{\ell s}$.

\emph{Long--long.} Two long edges of the same label are nested or share an endpoint, so they do
not cross. For labels $x<y$ the six endpoints fall in the order $0x,0y,1x,1y,2x,2y$, and of the
$3\times3$ type pairs exactly six interleave; hence
$g_{\ell\ell}=6\binom{n/3}{2}=\tfrac13 n^{2}-n$.

\emph{Long--short.} A long edge of type $\{p,q\}$ crosses a within-block edge only in its two
endpoint blocks $p,q$, and there exactly those whose interval strictly contains the label $x$;
an interior (intermediate) block is spanned entirely, giving nesting, not crossing. Thus each
label contributes $6$ times the number of sub-layout edges strictly containing a fixed vertex
position, i.e.\ $g_{\ell s}=6\bigl(\sum_{e}\operatorname{span}(e)-m'\bigr)$ over the $H(n/3)$
sub-layout, where $m'=(\log_3 n-1)\tfrac n3$ is its edge count. By Lemma~\ref{lem:span},
$\sum_e\operatorname{span}(e)=\tfrac23\bigl((n/3)^2-n/3\bigr)$, so
$g_{\ell s}=\tfrac49 n^{2}-2\,n\log_3 n+\tfrac23 n$.

Adding, $g(n)=\tfrac79 n^{2}-2\,n\log_3 n-\tfrac13 n<\tfrac79 n^{2}$. Since $g(m)<\tfrac79 m^{2}$
for every $m$, unrolling with $X(1)=0$ gives
$X(n)=\sum_{k\ge0}3^{k}g(n/3^{k})<\tfrac79 n^{2}\sum_{k\ge0}3^{-k}=\tfrac76 n^{2}$, and
$\cro(H)\le X(n)$.
\end{proof}

With Theorem~\ref{thm:lower} this brackets the ordinary crossing number as
$\Omega(n)\le\cro(H)\le O(n^{2})$, both ends unconditional; the one-page constant $\tfrac76$ is
sharp for $\Lambda$ (its per-level count is exact), and only a \emph{different} drawing, not a
better count of this one, could lower it further. The one-page layout $\Lambda$ is a
\emph{convex} rectilinear drawing (vertices on a line), so it bounds the convex crossing
number, and the whole rectilinear chain stays quadratic,
$\cro(H)\le\mathrm{rcr}(H)\le\mathrm{cr}^{*}(H)=O(n^{2})$~\cite{DujmovicLaRose,BienstockDean}
(convex and rectilinear drawings of cube-like graphs are studied for the hypercube by Anti\'c,
Fuladi, Limbach and Valtr~\cite{AFLV}). The extra logarithm in $\udcr$ is therefore not a cost
of rectilinearity: it is a cost of the \emph{unit-length} constraint alone. A separate \emph{separated} drawing, three shrunken disjoint copies of the $H(n/3)$ drawing joined by radial bundles
(Figures~\ref{fig:recsep} and~\ref{fig:sepd4}), gives the
sharper explicit form
\begin{equation}\label{eq:closed}
  \cro(H)\ \le\ \tfrac12\,n(n-3)\ =\ \binom{n}{2}-n,
\end{equation}
which we verified \emph{computationally} for all $n\le729$. It would follow from the same
recursion whenever the radial bundle cost obeys the per-level bound $B(n)\le n^{2}/3$: the
bundle--bundle part, $3\binom{n/3}{2}<n^{2}/6$, always does, but the bundle--copy part already
\emph{saturates} this per-level bound at $n=729$, so the recursion does \emph{not} by itself
extend $\tfrac12 n(n-3)$ past the checked range, and we state it as a computational bound rather
than a theorem. What is proved is only the order, $\cro(H)=\Theta(n^{2})$.

\begin{proposition}[Two-page book bound: $\cro(H)\le\tfrac23 n^{2}$]\label{prop:twopage}
Keep the base-$3$ spine order of $\Lambda$ but place each edge on one of \emph{two} pages of the
book~\cite{BernhartKainen}, the page being the parity of the coordinate it changes (so only edges
changing coordinates of equal parity can meet). This two-page layout $\Lambda_2$ satisfies the
recursion $c_2(n)=3\,c_2(n/3)+g_2(n)$ with the \emph{exact} per-level count
\[
  g_2(n)=\tfrac49 n^{2}-\tfrac19\,9^{-E}n^{2}-n(1+2E),\qquad E=\Big\lfloor\tfrac{d-1}{2}\Big\rfloor,
\]
and therefore $\cro(H)\le c_2(n)<\tfrac23 n^{2}$, improving the one-page $\tfrac76 n^{2}$ of
Theorem~\ref{thm:upper} by a factor $\tfrac{7}{4}$.
\end{proposition}
\begin{proof}
Peeling the top coordinate splits the spine into three consecutive blocks of $n/3$; the only
inter-block edges are the $n$ coordinate-$0$ (``long'') edges, so $c_2(n)=3\,c_2(n/3)+g_2(n)$ with
$g_2$ the same-page crossings meeting a long edge. All long edges change the same coordinate, hence
share page $0$; so $g_2$ is long--long plus long--short with the short edge on page $0$, that is, of \emph{even} coordinate.

\emph{Long--long.} As in Theorem~\ref{thm:upper}, two long edges cross iff their labels interleave,
giving $6\binom{n/3}{2}=\tfrac13 n^{2}-n$.

\emph{Long--short, coordinate $i$.} A long edge of rest $r$ crosses a coordinate-$i$ edge in one of
its two endpoint blocks iff that edge's rest-interval strictly contains $r$; each such short edge is
met by exactly the two long edges whose endpoint lies in its block, so the coordinate-$i$ total is
$2(\Sigma_i-n)$, where $\Sigma_i=\sum_{\text{coord }i}\operatorname{span}$. Each coordinate-$i$
triangle contributes spans $3^{\,d-1-i},3^{\,d-1-i},2\cdot3^{\,d-1-i}$, and there are $n/3$ of them,
so $\Sigma_i=\tfrac{4n^{2}}{9}3^{-i}$ and the coordinate-$i$ count is $\tfrac89 n^{2}3^{-i}-2n$.

Summing over the even coordinates $i=2,4,\dots\le d-1$ (there are $E=\lfloor(d-1)/2\rfloor$, with
$\sum 3^{-i}=\tfrac18(1-9^{-E})$) and adding the long--long term gives the stated $g_2(n)$. Dropping
its negative terms, $g_2(n)<\tfrac49 n^{2}$, so
$c_2(n)=\sum_{k\ge0}3^{k}g_2(n/3^{k})<\tfrac49 n^{2}\sum_{k\ge0}3^{-k}=\tfrac23 n^{2}$. The
per-level formula is exact; the counts $c_2(n)=18,288,3456,35316,337770$ for $d=2,\dots,6$ confirm
it, staying below $\tfrac23 n^{2}$ (ratios $0.22,0.40,0.53,0.60,0.64$).
\end{proof}

\begin{remark}[Two pages: an improvement in the constant only]
The two-page counts $18$, $288$, $3456$, $35316$, $337770$ for $d=2,\dots,6$ sit below the one-page
$24,\,468,\,5832,\,60912,\,587088$, so the book crossing number~\cite{BernhartKainen} is also
$\nu_2(H)=\Theta(n^{2})$: the second page improves the leading constant from $\tfrac76$ to $\tfrac23$
but not the order. The two bounds live in \emph{different} models: a two-page book draws its edges as
\emph{arcs} in two opposite half-planes, so Proposition~\ref{prop:twopage} bounds the ordinary and
book crossing numbers $\cro(H),\nu_2(H)$ but \emph{not} the rectilinear one; the proven
\emph{rectilinear} record stays the one-page convex $\operatorname{rcr}(H)\le\tfrac76 n^{2}$ of
Theorem~\ref{thm:upper}, and $\tfrac23$ does not beat $\tfrac76$ inside the rectilinear model.
Even so, the book \emph{is still not a good visual representation of the graph}:
its structure is \emph{several copies} (the recursion~\eqref{eq:rec}) rather than two half-planes.
The representations that do fit $H(d,3)$ are the recursive copies and, for the unit count, the
\emph{rectilinear} unit drawing ($\udcr$ is the rectilinear crossing number under the unit-length
constraint), so the one-page and rectilinear models, not the book, are the ones aligned with the
graph.
\end{remark}

\begin{figure}[htbp]\centering
  \includegraphics[width=.52\linewidth]{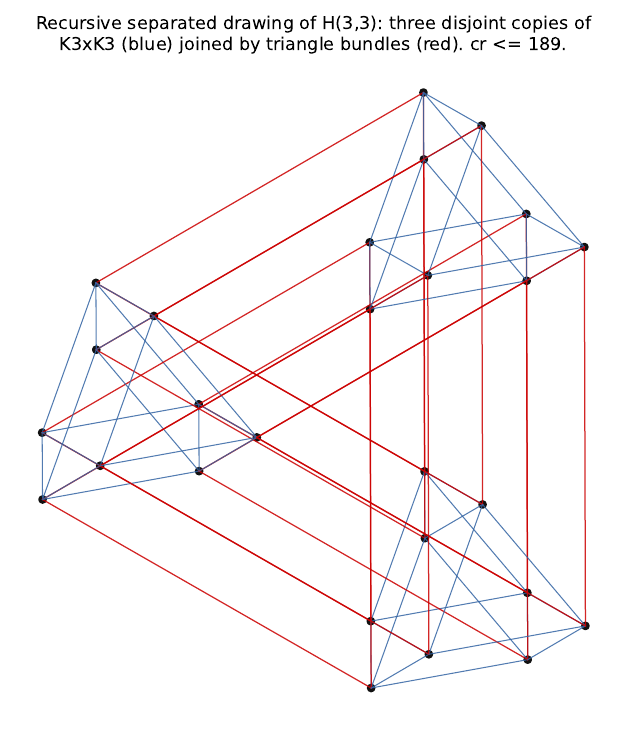}
  \caption{The separated recursive drawing of $H(3,3)$: three disjoint copies of
  $K_3\bx K_3$ (blue, no copy--copy crossings) joined by triangle bundles (red), which
  carry all the crossings. This non-unit drawing has fewer crossings than the unit
  drawing on the same vertices.}\label{fig:recsep}
\end{figure}

\begin{figure}[tp]\centering
  \includegraphics[width=.62\linewidth]{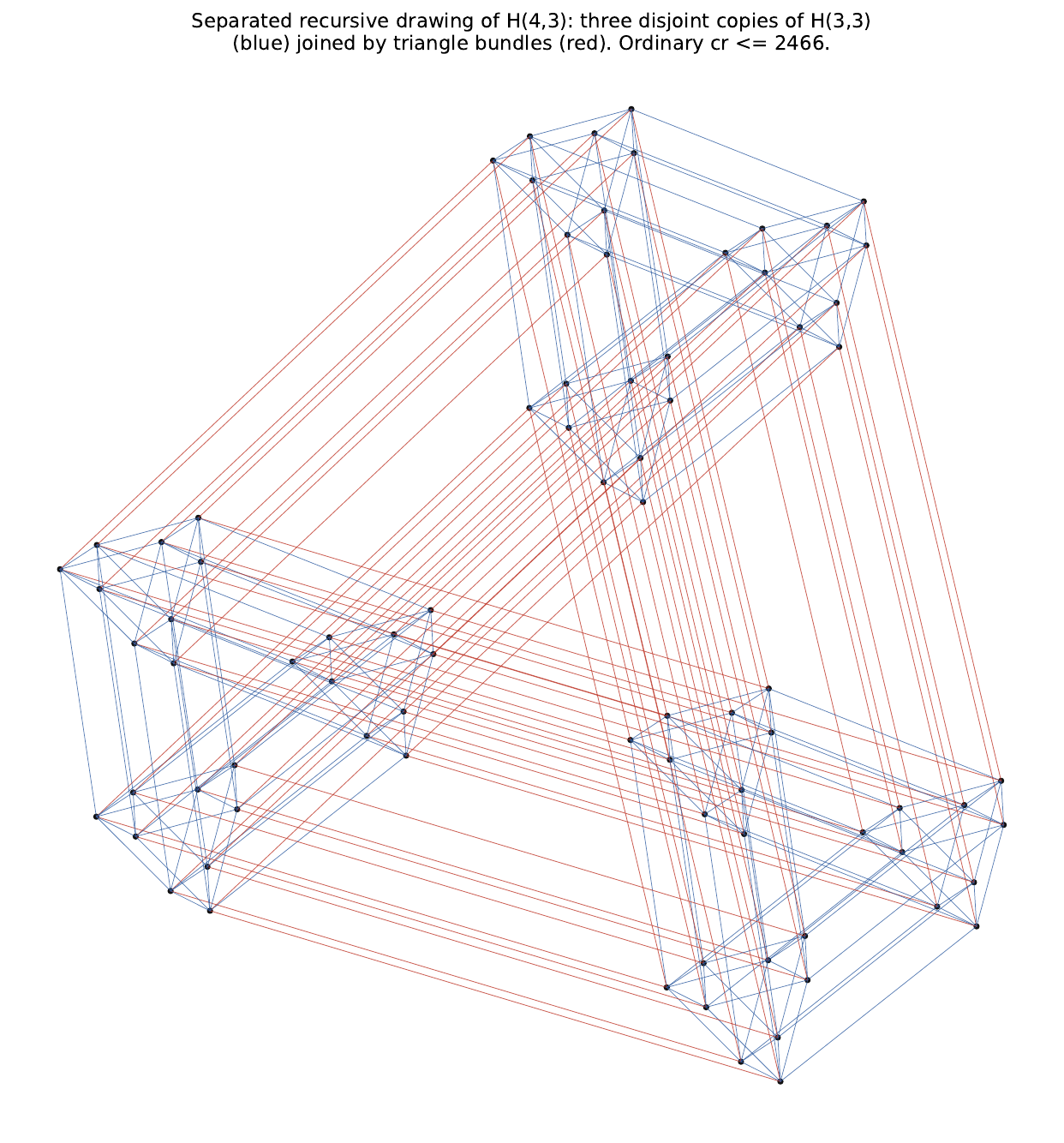}
  \caption{The separated recursive drawing one level up: $H(4,3)$ as three disjoint
  copies of the $H(3,3)$ drawing of Figure~\ref{fig:recsep} (blue) joined by triangle
  bundles (red). This free (non-unit) drawing stays well below the unit immersion and
  inside the $O(n^{2})$ band $\tfrac12 n(n-3)$.}
  \label{fig:sepd4}
\end{figure} The unit-distance drawing carries more
crossings than the separated one at every level: their ratio is a factor $\Theta(\log n)$,
the conjectured separation between the two invariants.

\subsection{The bracket in the number of vertices}
All the bounds are bands in $n$, and constant factors or fractions do not move them; only the
exponent of $n$ and the logarithmic factor matter. The instance has $n$ vertices and
$m=\Theta(n\log n)$ edges (a sparse graph), so a quadratic crossing number sits far below the
maximum $\Theta(n^{4})$ of $K_n$. The bounds by status are
\[
  \underbrace{\tfrac{7}{6}\,n^{2}=O(n^{2})}_{\text{Thm.~\ref{thm:upper}}}\ \ge\ \cro(H)\ \ge\
  \underbrace{\Omega(n^{2})}_{\text{Thm.~\ref{thm:tight}}},
  \qquad \Omega(n^{2})\le\ \udcr(H)\ \le O(n^{2}\log n)\ \ (\text{Thm.~\ref{thm:udcr}}).
\]
Both ends of the ordinary bracket match: $\cro(H)=\Theta(n^{2})$ unconditionally (upper bound
from a one-page layout, Theorem~\ref{thm:upper}; lower bound from bisection width,
Theorem~\ref{thm:tight}). The sharper explicit constant $\tfrac12 n(n-3)$ is verified for
$n\le729$. The unit count is sandwiched between $\Omega(n^{2})$ and $O(n^{2}\log n)$
(Theorem~\ref{thm:udcr}): the upper bound now matches the Minkowski count's order, so only the
matching lower bound $\Omega(n^{2}\log n)$ is left for the full $\Theta(n^{2}\log n)$.

\subsection{The exact order: $\cro(H)=\Theta(n^{2})$}
A balanced \emph{bisection} of a graph on $n$ vertices splits them into two halves of size
$n/2$; the \emph{bisection width} $\mathrm{bw}(G)$ is the least number of edges across such a
split. Two classical facts combine to force many crossings.

\begin{lemma}\label{lem:bw}
$\mathrm{bw}(H)\ge \tfrac34\,n=\Omega(n)$.
\end{lemma}
\begin{proof}
\emph{The spectrum, and why we need it.} The Hamming graph $H(d,q)=K_q^{\,\square d}$ is
distance-regular (it is the graph of the Hamming association scheme~\cite{BrouwerHaemers}) and
the Laplacian spectrum of a Cartesian product is the set of $d$-fold sums of the factor
spectra~\cite[\S1.4.6]{BrouwerHaemers}. The factor $K_q$ has Laplacian eigenvalues $0$ (once) and $q$
(with multiplicity $q-1$), so $H(d,q)$ has Laplacian eigenvalues
\[
  q\,i\quad\text{with multiplicity}\quad \binom{d}{i}(q-1)^{i},\qquad i=0,1,\dots,d .
\]
The \emph{algebraic connectivity} (the smallest nonzero eigenvalue~\cite{Fiedler}) is therefore
$\lambda_2=q$, which for $q=3$ gives $\lambda_2=3$ (for $K_3^{\,\square d}$ the factor spectrum
is $\{0,3,3\}$ and the $d$-fold sums reproduce the list above). We use $\lambda_2$ in the one
standard way it bounds a cut~\cite{Fiedler,Mohar}: for a $\pm1$ indicator $x\perp\mathbf 1$ of a
balanced split, the cut size is
$\tfrac14 x^{\top}Lx\ge\tfrac14\lambda_2\|x\|^{2}=\tfrac14\cdot3\cdot n$, so
$\mathrm{bw}(H)\ge\tfrac34 n$. This is the \emph{only} place the spectrum enters the crossing
theory: the recursion of \S\ref{ss:recdraw} is purely combinatorial, but the \emph{lower} bound
$\cro(H)=\Omega(n^{2})$ (Theorem~\ref{thm:tight}) rests entirely on this $\lambda_2=3$, through
the bisection--crossing inequality.
\end{proof}

\begin{theorem}\label{thm:tight}
$\cro(H)=\Theta(n^{2})$.
\end{theorem}
\begin{proof}
The upper bound is Theorem~\ref{thm:upper}. For the lower bound, the bisection--crossing
inequality of Leighton~\cite{Leighton} and of Sýkora--Vrťo~\cite{SykoraVrto,PSS} gives an absolute constant
$\beta>0$ with $\mathrm{bw}(G)^{2}\le \beta\bigl(\cro(G)+\sum_v \deg(v)^{2}\bigr)$ for every
graph $G$. Here every vertex has degree $2d=2\log_3 n$, so
$\sum_v\deg(v)^{2}=n\,(2\log_3 n)^{2}=O(n\log^{2}n)$. With $\mathrm{bw}(H)\ge\tfrac34 n$
(Lemma~\ref{lem:bw}),
\[
  \cro(H)\ \ge\ \tfrac1\beta\,\mathrm{bw}(H)^{2}-\sum_v\deg(v)^{2}
  \ \ge\ \tfrac{9}{16\beta}\,n^{2}-O(n\log^{2}n)\ =\ \Omega(n^{2}).
\]
\end{proof}

So the ordinary crossing number is settled up to a constant: $\cro(H)=\Theta(n^{2})$.

\subsection{Bounding the unit count}
For the unit-distance crossing number a concentration argument gives the tight order from
above, matching the count of the Minkowski drawing itself.

Regard $P(a)=\sum_{i=1}^{d}T_i[a_i]$ as a sum of $d$ independent bounded vectors, each uniform
on the three vertices of a unit triangle, with mean $0$ and, for a generic twist, non-degenerate
covariance a fixed positive multiple of the identity; the total covariance is $\Theta(d)\,I$.
Two concentration estimates bound how many of the $n$ points a fixed unit disk can hold; we use
the first for the sharp order and give the second as a self-contained fallback.

\begin{lemma}[Sharp density, by a planar concentration bound]\label{lem:conc}
At a generic Minkowski realization every open unit disk contains $O(n/d)$ of the $n$ vertices.
\end{lemma}
\begin{proof}
For a generic twist the increments $T_i[a_i]$ span $\R^2$ but need \emph{not} lie on any common
lattice, so a lattice local limit theorem does not apply; we use instead the two-dimensional
\emph{concentration-function} inequality of Esseen (the planar Kolmogorov--Rogozin
bound)~\cite{Esseen}. For a sum of $d$ independent $\R^2$-valued vectors with total covariance
$\Theta(d)\,I$, it gives
\[
  Q(P,1)\ :=\ \sup_{x\in\R^2}\Pr\!\big[P(a)\in D(x,1)\big]\ =\ O\!\big(1/\det{}^{1/2}(\Theta(d)I)\big)\ =\ O(1/d),
\]
with \emph{no} lattice hypothesis, only the non-degeneracy of the planar covariance. Hence any
unit disk contains at most $n\,Q(P,1)=O(n/d)$ of the $n$ points. (Numerically the maximum is
$\approx 2n/d$; when the twist happens to be lattice-supported the same $O(1/d)$ also follows from
the classical lattice local limit theorem~\cite{BhattacharyaRao}, as a check.)
\end{proof}

\begin{lemma}[Weaker density, elementary]\label{lem:concelem}
At a generic Minkowski realization every open unit disk contains $O(n/\sqrt d)$ of the $n$
vertices, with no appeal to the central limit theorem.
\end{lemma}
\begin{proof}
Fix a generic direction $u\in S^1$ and project: $u\!\cdot\!P(a)=\sum_i u\!\cdot\!T_i[a_i]$ is a
sum of $d$ independent variables, each taking three values (with probability $\tfrac13$ each)
that are pairwise $\delta$-separated for some $\delta=\delta(u)>0$. Any interval of length
$<\delta$ misses at least one of the three, so the L\'evy concentration function satisfies
$Q(u\!\cdot\!T_i,\delta)\le\tfrac23$, i.e.\ $1-Q\ge\tfrac13$. The Kolmogorov--Rogozin
inequality~\cite{Rogozin} then gives, for any length $\lambda$,
$Q\bigl(u\!\cdot\!P,\lambda\bigr)\le C\,(\lambda/\delta+1)/\sqrt{\textstyle\sum_i(1-Q)}
=O(1/\sqrt d)$. A unit disk lies in a strip $\{u\!\cdot\!y\in[x,x+2]\}$, so it contains at most
$n\,Q(u\!\cdot\!P,2)=O(n/\sqrt d)$ of the points.
\end{proof}

\begin{theorem}\label{thm:udcr}
$\ \Omega(n^{2})\le\udcr(H)\le O(n^{2}\log n)$.
\end{theorem}
\begin{proof}
The lower bound is immediate: $\udcr(H)\ge\cro(H)=\Omega(n^{2})$ by Theorem~\ref{thm:tight}.
For the upper bound we bound the crossings of the Minkowski drawing. Every edge is a unit
segment, so an edge can be crossed only by edges having an endpoint within distance $2$ of it.
By Lemma~\ref{lem:conc} the $2$-neighborhood of a unit segment contains $O(n/d)$ vertices,
each of degree $2d$; hence each edge meets at most $O(n/d)\cdot 2d=O(n)$ other edges, and
crosses at most that many. Summing over the $m=d\,n$ edges,
\[
  \udcr(H)\ \le\ c(n)\ \le\ \tfrac12\,m\cdot O(n)\ =\ \tfrac12\,(d\,n)\,O(n)\ =\ O(n^{2}\log n).
\]
\end{proof}

Thus the Minkowski drawing is order-optimal from above, $\udcr(H)=O(n^{2}\log n)$, and the
count itself is $\Theta(n^{2}\log n)$ in every computation to $n=729$, the single logarithm coming from the sum over the $\log_3 n$ coordinate directions, each contributing $O(n^{2})$
crossings. The constant behind this $O$ is explicit.

\begin{proposition}[Explicit single-logarithm bound]\label{prop:udcrexplicit}
For radii in the relevant range $1\le r\le2$ the planar concentration bound of
Lemma~\ref{lem:conc} takes the form $Q(P,r)\le\kappa\,r^{2}/d$, so a disk of radius $r$ holds at
most $\kappa r^{2}n/d$ of the $n$ vertices ($\kappa$ is finite by the Esseen inequality; the
value $\kappa\le2$ is the one we \emph{compute} over the range $d\le6$, which makes the bound
explicit there, and is not a proven universal constant, a closed-form universal $\kappa$ from
Esseen is left open). (For $r\to0$ this fails, since an atom carries positive mass, but only $r=\tfrac32$ is used below.) Then the Minkowski drawing satisfies
\[
  \udcr(H)\ \le\ c(n)\ \le\ \tfrac94\,\kappa\,n^{2}\log_3 n\ \bigl(\le\ \tfrac92\,n^{2}\log_3 n\bigr).
\]
\end{proposition}
\begin{proof}
Two unit segments cross only if each has an endpoint within distance $1$ of the other; the
distance-$1$ neighborhood of a unit edge lies in a disk of radius $\tfrac32$, which holds at
most $\kappa(\tfrac32)^{2}n/d=\tfrac94\kappa\,n/d$ vertices. Each has degree $2d$, so a unit
edge is met by at most $\tfrac94\kappa\,\tfrac nd\cdot2d=\tfrac92\kappa\,n$ edges. Summing over
the $m=dn$ edges and halving, $c(n)\le\tfrac12(dn)\cdot\tfrac92\kappa n=\tfrac94\kappa\,n^{2}\log_3 n$.
\end{proof}

The origami drawing itself realizes the far smaller $c(n)\approx0.11\,n^{2}\log_3 n$ (verified
$d\le6$); the gap between $0.11$ and the explicit ceiling is the slack of the union count, not of
the order. The only remaining gap (the central open problem of this half of the paper) is
the matching \emph{lower} bound $\udcr(H)=\Omega(n^{2}\log n)$: it cannot come from the recursion
(which produces only drawings, hence upper bounds) nor from bisection ($\udcr$ is geometric, not a
graph invariant), and would require forcing \emph{every} faithful unit realization to
$\Omega(n^{2}\log n)$ crossings. Theorem~\ref{thm:allmink} sharpens this to a concrete,
one-direction statement.

We separate here, carefully, what is proved from what is not. The following identity-level bound is
unconditional.

\begin{proposition}[Double counting by coordinate class]\label{prop:lbreduce}
For a faithful Minkowski drawing $D$ let $X_i(D)$ be the number of crossings of $D$ that involve at
least one edge of coordinate class $i$ (the edges changing coordinate $i$). Counting each crossing
once for every class present among its two edges,
\[
  \tfrac12\sum_{i=1}^{d}X_i(D)\ \le\ \mathrm{cr}_u(D),
\]
with equality unless two edges of the \emph{same} class cross.
\end{proposition}

By Theorem~\ref{thm:allmink} every faithful realization is such a $D$, so a uniform lower bound on
the $X_i$ would force the logarithm. We isolate exactly that missing ingredient as a conjecture:
a single, directional density statement about one explicit point configuration, no longer a
quantifier over all abstract immersions.

\begin{conjecture}[Uniform directional density]\label{conj:directional}
There is a constant $c>0$ such that $X_i(D)\ge c\,n^{2}$ for every faithful Minkowski drawing $D$ of
$H(d,3)$ and every coordinate class $i$.
\end{conjecture}

\begin{corollary}\label{cor:cond}
Conjecture~\ref{conj:directional} implies $\udcr(H(d,3))=\Omega(n^{2}\log n)$, and hence, with the
upper bound of Theorem~\ref{thm:udcr}, settles Conjecture~\ref{conj:cr}: $\udcr(H)=\Theta(n^{2}\log n)$.
\end{corollary}
\begin{proof}
For any faithful $D$, Proposition~\ref{prop:lbreduce} and Conjecture~\ref{conj:directional} give
$\mathrm{cr}_u(D)\ge\tfrac12\sum_i X_i(D)\ge\tfrac{c}{2}\,n^{2}d=\tfrac{c}{2}\,n^{2}\log_3 n$;
minimizing over $D$ (which by Theorem~\ref{thm:allmink} is all faithful realizations) gives
$\udcr(H)\ge\tfrac{c}{2}\,n^{2}\log_3 n$.
\end{proof}

We emphasize that Proposition~\ref{prop:lbreduce} and Corollary~\ref{cor:cond} are theorems, while
Conjecture~\ref{conj:directional} is \emph{unproved}: it carries the whole remaining difficulty,
now localized to one coordinate direction. The evidence for it is strong but computational: over the
minimizing orientations the count grows linearly in $d$, with
$\udcr(H(d,3))/n^{2}\approx 0.11,0.32,0.53$ for $d=2,3,4$ (increments $\approx0.21$, the signature of
$n^{2}d$), and the per-class share is uniformly $X_i/n^{2}\in[0.18,0.28]$ across directions in every
sampled drawing. The class-$i$ edges are the $n/3$ translates of the triangle $T_i$, so
Conjecture~\ref{conj:directional} is an anti-concentration statement of crossing-lemma type, the
natural next target.

The same argument with the elementary Lemma~\ref{lem:concelem} in place of
Lemma~\ref{lem:conc} needs no limit theorem at all:

\begin{corollary}[Elementary bound]\label{cor:udcrelem}
Without invoking the central limit theorem, $\udcr(H)=O\!\bigl(n^{2}(\log n)^{3/2}\bigr)$.
\end{corollary}
\begin{proof}
By Lemma~\ref{lem:concelem} the $2$-neighborhood of a unit edge holds $O(n/\sqrt d)$ vertices,
each of degree $2d$, so each edge crosses $O(n/\sqrt d)\cdot2d=O(n\sqrt d)$ others; summing over
the $m=dn$ edges, $\udcr(H)\le\tfrac12 m\,O(n\sqrt d)=O(n^{2}d^{3/2})=O(n^{2}(\log_3 n)^{3/2})$.
\end{proof}

Theorem~\ref{thm:udcr} is an order bound with an implicit constant. From the proved recursion
we can instead give an \emph{explicit closed-form majorant} (at the cost of one extra
logarithm) with every step verified.

\begin{theorem}[Closed-form majorant]\label{thm:udcrclosed}
Writing $L=\log_3 n$ (so $L=d$, the number of coordinates, since $n=3^{d}$; a different base
rescales $L$ by a constant and changes only the implied constant, not the order), the Minkowski
unit drawing satisfies the recursion $c(n)=3\,c(n/3)+g(n)$ with the per-level bound
$g(n)\le n^{2}L^{2}$, and therefore
\[
  \udcr(H)\ \le\ c(n)\ \le\ \tfrac32\,n^{2}\bigl(L^{2}-L+1\bigr)
  \ =\ \tfrac32\,n^{2}\bigl((\log_3 n)^{2}-\log_3 n+1\bigr)\ =\ O(n^{2}\log^{2}n).
\]
\end{theorem}

\begin{proof}
The recursion is \eqref{eq:rec} (Proposition~\ref{prop:rec}): peeling one coordinate splits
the drawing into three translated copies (giving $3\,c(n/3)$ exactly) plus the bundle sums
$g(n)=X+Y+Z$. We bound the three classes by counting edge pairs, using that there are $n$
bundle edges and $m-n=n(L-1)$ copy edges:
\[
  Y\le 9\binom{n/3}{2}<\tfrac12 n^{2},\quad
  Z\le n\cdot n(L-1)=n^{2}(L-1),\quad
  X\le 3\Bigl(\tfrac{n(L-1)}{3}\Bigr)^{2}=\tfrac13 n^{2}(L-1)^{2}.
\]
Hence $g(n)\le n^{2}\bigl(\tfrac13(L-1)^{2}+(L-1)+\tfrac12\bigr)\le n^{2}L^{2}$ for $n\ge9$.
Now argue by induction on $d$ (with $n=3^{d}$, $L=d$). Put $f(n)=\tfrac32 n^{2}(L^{2}-L+1)$.
The base case $n=9$ holds, $c(9)=9\le f(9)=\tfrac{729}{2}$. The function $f$ absorbs the extremal
per-level term \emph{exactly}:
\[
  f(n)-3\,f(n/3)=\tfrac32 n^{2}(L^{2}-L+1)-3\cdot\tfrac32\Bigl(\tfrac n3\Bigr)^{2}\!\bigl((L{-}1)^{2}-(L{-}1)+1\bigr)=n^{2}L^{2}.
\]
Assuming $c(n/3)\le f(n/3)$, the recursion and $g(n)\le n^{2}L^{2}$ give
$c(n)=3\,c(n/3)+g(n)\le 3\,f(n/3)+n^{2}L^{2}=f(n)$. Since $\udcr(H)\le c(n)$, the bound follows.
\end{proof}

\begin{remark}\label{rem:trivial}
The majorant of Theorem~\ref{thm:udcrclosed} is weaker than the trivial bound by the number of pairs
of edges: with $m=dn$ edges, $\udcr(H)\le\binom{m}{2}<\tfrac12 n^{2}L^{2}$, about a third of
$\tfrac32 n^{2}(L^{2}-L+1)$ for large $L$ (for $d=2$: $153$ against $364.5$). Its interest is only
that it comes out of the recursion with an explicit constant; the substantive upper bound is the
order $O(n^{2}\log n)$ of Theorem~\ref{thm:udcr}, by concentration.
\end{remark}

\begin{conjecture}\label{conj:cr}
For the Hamming triangle family on $n$ vertices, $\udcr(H)=\Theta(n^{2}\log n)$; equivalently
$\udcr(H)=\Omega(n^{2}\log n)$, matching Theorem~\ref{thm:udcr}. With Theorem~\ref{thm:tight}
the unit constraint then inflates the crossing number by a factor $\Theta(\log n)$:
$\ \udcr(H)/\cro(H)\to\infty$.
\end{conjecture}

\subsection{Small cases and the values for $d\le6$}\label{ss:smallcases}
The self-similar structure is best seen in the small cases. Figure~\ref{fig:gallery} draws the
faithful unit realizations of $H(3,3)$, $H(4,3)$ and $H(5,3)$ at the origami base: each is three
shrunken copies of the previous one plus a coordinate bundle, exactly the recursion
\eqref{eq:rec}. Table~\ref{tab:values} collects the bounds; Figure~\ref{fig:values} plots them.

\begin{figure}[t]\centering
  \includegraphics[width=\linewidth]{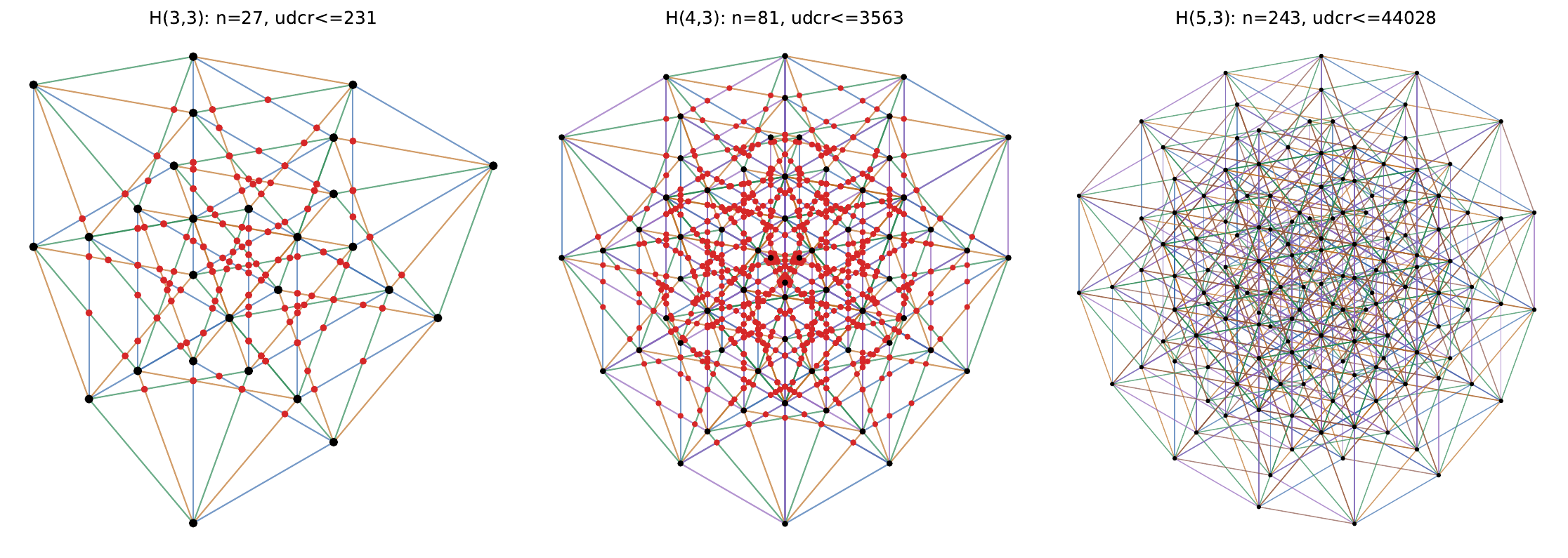}
  \caption{The origami unit realizations of $H(3,3),H(4,3),H(5,3)$ ($n=27,81,243$); edges are
  colored by the coordinate they change, and each drawing is three translated copies of the
  one before it plus a bundle. Crossings marked for $d=3,4$.}\label{fig:gallery}
\end{figure}

\begin{table}[t]\centering
\begin{tabular}{c|c|c|c|c|c}
$d$ & $n=3^d$ & edges $dn$ & $\cro\ge n/3$ & $\cro\le\tfrac12 n(n-3)$ & $\udcr\le c(n)$\\\hline
$2$ & $9$   & $18$   & $3^{\ast}$ & $27$    & $9^{\ast}$\\
$3$ & $27$  & $81$   & $9$        & $324$   & $231$\\
$4$ & $81$  & $324$  & $27$       & $3159$  & $3546$\\
$5$ & $243$ & $1215$  & $81$       & $29160$ & $44028$\\
$6$ & $729$ & $4374$ & $243$      & $264627$& $490935$
\end{tabular}
\caption{Bounds for $H(d,3)$. Starred: $\cro(K_3\bx K_3)=3$ and $\udcr(K_3\bx K_3)=9$, both
\emph{exact} (the latter by Proposition~\ref{prop:base9}). For $d\ge3$ the last column is a faithful good-drawing count, an upper bound for $\udcr$, \emph{not} an exact value; for $d=3$
the value $231$ is conjecturally exact (Remark~\ref{rem:231}). The entries come from drawings near
constructible points; for $d\ge4$ a generic drawing can be slightly smaller, e.g.\ $3486$ for $d=4$
(\S\ref{ss:mintier}). The $\cro$ column is
the separated closed form $\tfrac12 n(n-3)$.}
\label{tab:values}
\end{table}

\begin{figure}[t]\centering
  \includegraphics[width=.89\linewidth]{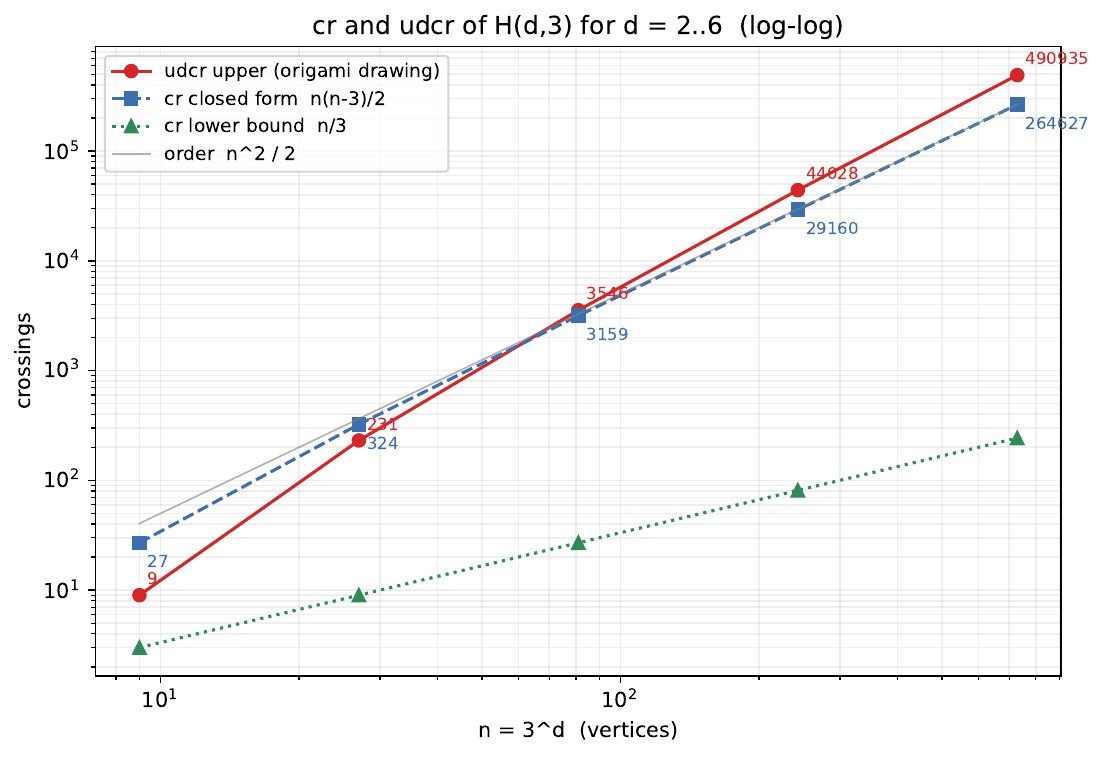}
  \caption{The bounds of Table~\ref{tab:values} on a log--log scale: the $\cro$ closed form
  $\tfrac12 n(n-3)$, its lower bound $n/3$, and the origami $\udcr$ drawing, all near the
  quadratic guide $\tfrac12 n^2$.}\label{fig:values}
\end{figure}

\emph{On exact values.} For $d=2$ both invariants are settled exactly, $\cro(K_3\bx K_3)=3$~\cite{HKS} and
$\udcr(K_3\bx K_3)=9$ (Proposition~\ref{prop:base9}: every faithful realization is a Minkowski sum,
and the count is constantly $9$ on the flex). For $d\ge3$ the table gives \emph{bounds}: the lower bound $n/3$
(Theorem~\ref{thm:lower}), the ordinary upper bound $\tfrac12 n(n-3)$ (separated drawing,
verified $n\le729$), and the unit upper bound from the origami drawing. For $\udcr$ and $d=3$,
by Theorem~\ref{thm:allmink} every faithful realization is a Minkowski sum, and over that family
every general-position count we sampled is $231$, $237$ or $243$, so we conjecture
$\udcr(H(3,3))=231$ (Remark~\ref{rem:231}); the upper bound $\udcr(H(3,3))\le231$ is proved. Also
open is the exact ordinary $\cro(H(d,3))$: for $d=3$ the best drawing we have is a computer-optimized free rectilinear
drawing with $99$ crossings (Figure~\ref{fig:h33free}, verified in general position; coordinates in
Appendix~\ref{app:free}), so $\cro(H(3,3))\le99$, and the same local optimization gives a
general-position free rectilinear drawing of $H(4,3)$ with $1817$ crossings, so
$\cro(H(4,3))\le1817$, about $0.28\,n^{2}$, far under the one-page $\tfrac76 n^{2}$. These
free drawings are evidence that the true $\cro(H(d,3))$ sits well below the proven rectilinear
$\tfrac76 n^{2}$: a separated recursive \emph{rectilinear} drawing (translated copies, whose
parallel connector bundles drop the long--long term from $\approx6\binom{n/3}{2}$ to
$\approx2\binom{n/3}{2}$) empirically reaches $\approx\tfrac12 n^{2}$, but the gain is geometric
(average-case), not a worst-case bound, so a \emph{proven} rectilinear bound below $\tfrac76 n^{2}$
remains open. For
the growth rate, the matching lower bound $\udcr(H)=\Omega(n^2\log n)$: the same barrier met for the
Robertson cage and the hypercube, a structural lower bound beating $\Omega(n)$ (for $\cro$) being
what is missing.

\emph{Complexity, made explicit.} The recurrence $c(n)=3\,c(n/3)+g(n)$ is solved by the master
theorem with $a=3$ branches, $b=3$ size reduction, and critical exponent $\log_b a=1$, i.e.\
$n^{\log_b a}=n$. The regime is set by $g$: for the separated drawing $g(n)=O(n^2)$ dominates
$n$, so the top level rules and $\cro=\Theta(n^2)$; for the unit drawing $g(n)=O(n^2\log n)$,
giving $\udcr=O(n^2\log n)$. Explicitly, unrolling the sum,
\[
  c(n)=\sum_{k\ge0}3^{k}\,g(n/3^{k})
      =\begin{cases}\Theta(n^2), & g(m)=\Theta(m^2),\\[2pt]
                    \Theta(n^2\log n), & g(m)=\Theta(m^2\log m),\end{cases}
\]
since $\sum_k 3^k(n/3^k)^2=n^2\sum_k 3^{-k}=\tfrac32 n^2$ in the first case, and the extra
$\log(n/3^k)$ inside the sum contributes the single $\log n$ in the second. In vertices, both
are polynomial: with $m=dn=\Theta(n\log n)$ edges, $\cro=\Theta(n^2)$ is far below the maximum
$\Theta(n^4)$ of $K_n$, and the recognition problem behind these graphs is $\exists\R$-complete
even though the counts here are computed in closed polynomial form.

\paragraph{Checking the recursion on the data.} The computed good-drawing sequence
$c(n)=9,231,3546,44028,490935$ ($d=2,\dots,6$) confirms the recursion $c(n)=3\,c(n/3)+g(n)$
directly: the intra-copy term $3\,c(n/3)$ is exact by construction, and the bundle remainder
$g(n)=c(n)-3\,c(n/3)$ follows the predicted $\Theta(n^{2}\log n)$ law,
\[
\renewcommand{\arraystretch}{1.15}
\begin{array}{r|cccc}
d & 3 & 4 & 5 & 6\\\hline
3\,c(n/3) & 27 & 693 & 10638 & 132084\\
g(n) & 204 & 2853 & 33390 & 358851\\
g(n)/(n^{2}\log_3 n) & 0.093 & 0.109 & 0.113 & 0.113
\end{array}
\]
so $g(n)\approx0.11\,n^{2}\log_3 n$ settles to a constant times $n^{2}\log n$: exactly the
per-level growth that Theorem~\ref{thm:udcrclosed} needs, with the $H(5,3)$ point ($c=44028$,
$g=33390$) squarely on the line.

\paragraph{$H(5,3)$ in detail.} As a worked instance beyond the base case, take $H(5,3)$
($n=243$, $1215$ edges). A short search over the flex and over the separated-drawing parameters
sharpens the tabulated bounds to
\[
  81\ =\ \tfrac n3\ \le\ \cro(H(5,3))\ \le\ 27222,\qquad
  \Omega(n^{2})\ \le\ \udcr(H(5,3))\ \le\ 42876 ,
\]
the ordinary upper bound from an optimized separated drawing (below the closed form
$\tfrac12 n(n-3)=29160$) and the unit one from the best twist found (below the origami value
$44028$). Figure~\ref{fig:h53best} shows that best unit realization. The exact values remain
open, as for every $d\ge3$.

\begin{figure}[t]\centering
  \includegraphics[width=.62\linewidth]{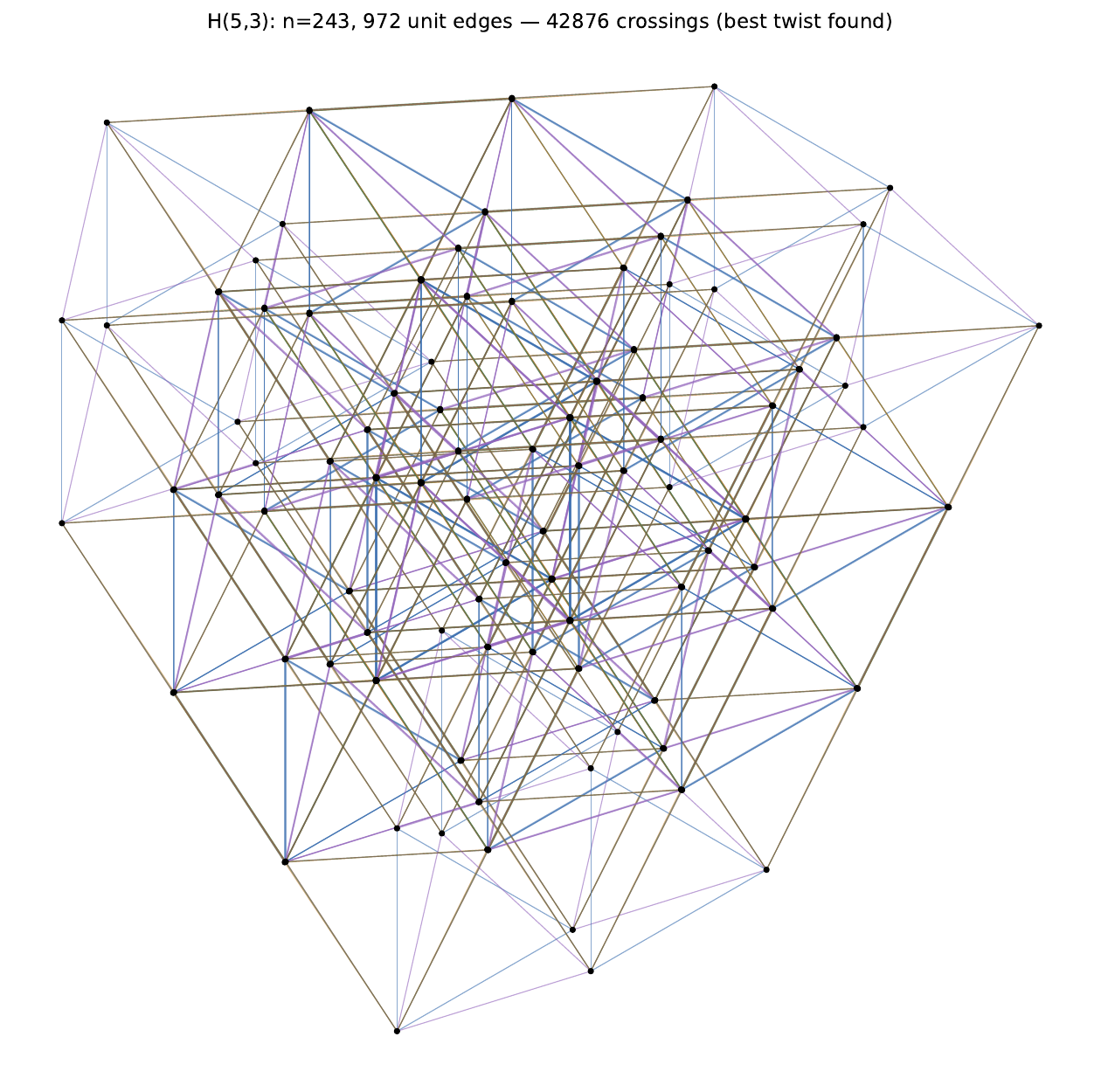}
  \caption{The best faithful unit realization of $H(5,3)$ found here ($243$ vertices, $1215$
  unit edges), with $42876$ crossings, the current upper bound on $\udcr(H(5,3))$, edges
  colored by the coordinate they change.}\label{fig:h53best}
\end{figure}

\subsection{Constructible configurations are near-optimal}\label{ss:mintier}
The value $\udcr(H)$ is a minimum of the crossing count over the $(d-1)$-parameter flex, and
it is natural to ask where in that flex, at a generic point, or at one of
the constructible points of Sections~\ref{sec:fields}--\ref{sec:faithful}, the minimum sits. The
evidence is $d$-dependent.

For $H(3,3)$ (two-parameter flex) the lowest crossing count we find is at the \emph{origami}
configuration (more precisely, in good drawings near it), the enneagon twist by $20^\circ=60^\circ/3$ in the trisection field
$\Q(\sin20^\circ)$; robust counts (median over general-position perturbations) give
\[
  \underbrace{231}_{\text{origami }20^\circ}\ <\
  \underbrace{237}_{\text{compass; generic}}\ <\
  \underbrace{243}_{\text{equal spacing }40^\circ},
\]
so among the drawings tried the lowest count is the constructible (enneagon) one, degree $3$
over $\Q$, while equal spacing (the ``obvious'' symmetry) is near the maximum. (By
Theorem~\ref{thm:allmink} every faithful realization is a Minkowski sum, and every general-position
count we sampled over the family is at least $231$; this supports $\udcr(H(3,3))=231$, see
Remark~\ref{rem:231}.) For $H(4,3)$
(three-parameter flex) we do \emph{not} know the exact minimum; what follows is only a
comparison of \emph{distinct} faithful drawings of the same graph, each an upper bound on
$\udcr(H(4,3))$, not an exact value, at four representative points of the flex plus the best
drawing a random search found:
\[
  \underbrace{3486}_{\text{best found (generic)}}\ <\
  \underbrace{3546}_{\text{compass }15^\circ\!,30^\circ\!,45^\circ}\ <\
  \underbrace{3558}_{\text{origami}}\ <\
  \underbrace{3582}_{\text{generic}}\ <\
  \underbrace{3594}_{\text{equal spacing}}.
\]
On these drawings compass and origami still beat both a generic point and the
symmetric one (by about $1\%$), but a random search finds a slightly lower generic drawing, so
here the constructible points are \emph{near}-optimal rather than exactly optimal.

\begin{conjecture}\label{conj:mintier}
Constructible (compass or origami) realizations of $H(d,3)$ are near-optimal for $\udcr$: they
have fewer crossings than the equal-spacing symmetric realization, and their count exceeds the
minimum over the flex by a factor $1+o(1)$. For $H(3,3)$ the lowest count found, $231$, occurs in
good drawings near the origami enneagon twist of $\Q(\sin20^\circ)$; for $H(4,3)$ the best
drawing found is generic, with $3486$ crossings against $3546$ at the compass point, a difference
below $2\%$. Whether the minimizer is constructible for $d\ge4$ is open.
\end{conjecture}

The two halves of the paper describe one object. The first half studies \emph{where} the
faithful realizations of $H(d,3)$ live: a $(d-1)$-parameter flex on which the compass, origami and exotic points lie (Sections~\ref{sec:fields}--\ref{sec:faithful}). The second half asks \emph{which}
realization to draw, and the answer is found near one of those points: the crossing count is a
function on the flex, and the lowest value we find (for $H(3,3)$) is at the origami
enneagon twist, the $\Q(\sin20^\circ)$ realization that Section~\ref{ss:mintier} adopts as the
reference drawing (its global optimality left open). The constructibility tier is thus the location of the optimum. The
flexibility that lets $H(3,3)$ range over the tiers is the same $(d-1)$-parameter freedom over
which $\udcr$ is minimized, and the spatial realizations
(Figures~\ref{fig:space3d},~\ref{fig:h43space}) are the crossing-free lifts that the plane
forbids; rigidity/arithmetic and crossing combinatorics are read off the same flex. The
total-crossing invariant $\udcr$ complements the recent \emph{per-edge} theory of unit-distance
graphs that bounds crossings or edges locally: the matchstick ($0$-crossing) extremal bound of
Lavoll\'ee and Swanepoel~\cite{LavolleeSwanepoel}, the $1$-planar unit-distance graphs of Geh\'er
and T\'oth~\cite{GeherToth}, and the $k$-planar theory of Geh\'er, P\'alv\"olgyi, Simon and
T\'oth~\cite{GPST}. Those works ask how \emph{dense} a unit-distance graph can be while keeping few
crossings \emph{per edge}; $\udcr$ asks the dual question (the least \emph{total} crossings of a
\emph{fixed} graph over all its faithful unit realizations), which is not addressed there. Our
lower bound $\cro(H)=\Omega(n^{2})$ rests on the bisection--crossing inequality
(Lemma~\ref{lem:bw}, Theorem~\ref{thm:tight}), not on the crossing lemma.

\section{Discussion}

\subsection{Conclusions: what is proved, what is conjectured, and which drawing is faithful}
The organizing contribution of this paper is a \emph{definition}: the unit-distance crossing number
$\udcr(G)$, the least number of crossings over all \emph{faithful unit} drawings of $G$, drawings in which every edge is a segment of length exactly $1$. It is the geometric crossing number native to
the world where only unit-length edges exist; equivalently, the rectilinear crossing number taken
subject to the unit-length constraint, so that $\cro(G)\le\operatorname{rcr}(G)\le\udcr(G)$, the drawing
model shrinking at each step. Where $\cro$ and $\operatorname{rcr}$ let edges be any length, $\udcr$ is
the invariant intrinsic to unit-distance graphs themselves. With it in hand we can separate cleanly
the established facts from the open ones, and the crossing-\emph{optimal} drawings from the
structurally \emph{faithful} ones.

\paragraph{Theorems (unconditional).}
\begin{enumerate}[label=\textup{(T\arabic*)}]
\item \emph{Structure of unit realizations.} Every faithful planar unit realization of $H(d,3)$ is a
Minkowski sum of $d$ unit triangles (Theorem~\ref{thm:allmink}); the $(d-1)$-parameter Minkowski
flex is therefore the \emph{entire} faithful realization variety.
\item \emph{Unit counts.} $\udcr(K_3\bx K_3)=9$ exactly (Proposition~\ref{prop:base9});
$\udcr(H(3,3))\le231$, by good drawings near the origami configuration (Corollary~\ref{cor:udcrmink});
and $\udcr(H(d,3))=O(n^{2}\log n)$ from above.
\item \emph{Ordinary crossing number, exact order.} $\cro(H(d,3))=\Theta(n^{2})$: upper bound
$\tfrac76 n^{2}$ from the one-page convex (rectilinear) layout (Theorem~\ref{thm:upper}), lower bound
$\Omega(n^{2})$ from the spectral bisection width $\mathrm{bw}\ge\tfrac34 n$ through the
bisection--crossing inequality (Theorem~\ref{thm:tight}). This is the one place the spectrum enters.
\item \emph{A better constant, topologically.} $\cro(H(d,3))\le\tfrac23 n^{2}$ from the two-page book
layout (Proposition~\ref{prop:twopage}); and the convex rectilinear record is
$\operatorname{rcr}(H)\le\tfrac76 n^{2}$, with $\operatorname{rcr}=\Theta(n^{2})$.
\item \emph{Small-case drawings.} $\cro(H(3,3))\le99$ and $\cro(H(4,3))\le1817$, verified in general
position (Appendix~\ref{app:free}).
\end{enumerate}

\paragraph{Conjectures and open problems.}
\begin{enumerate}[label=\textup{(C\arabic*)}]
\item \emph{The exact constant of $\cro$.} Only the \emph{order} $\Theta(n^{2})$ is settled; the
asymptotic constant $\cro/n^{2}$ is open. It lies below $\tfrac23$, and the free drawings
($\cro/n^{2}\le0.14$ at $d=3$, $\le0.28$ at $d=4$) and the separated rectilinear drawing
($\approx\tfrac12 n^{2}$) suggest it is markedly smaller than every proven upper constant. In
particular $\cro=\tfrac23 n^{2}$ is \emph{not} claimed: $\tfrac23$ is a ceiling we can prove, not a
value.
\item \emph{The unit lower bound.} $\udcr(H(d,3))=\Omega(n^{2}\log n)$, which would give
$\udcr=\Theta(n^{2}\log n)$, is \emph{conditional} on the directional
Conjecture~\ref{conj:directional}; unconditionally only $\udcr\ge\cro=\Omega(n^{2})$ is known. The
missing logarithm on the lower side is the Erd\H{o}s--Guy-type barrier.
\item \emph{A proven rectilinear improvement.} Whether a \emph{rectilinear} drawing beats
$\tfrac76 n^{2}$ with a provable worst-case bound is open: the separated construction reaches
$\approx\tfrac12 n^{2}$ only on average, not provably.
\item \emph{The exact value for $d=3$.} $\udcr(H(3,3))=231$: the upper bound is proved, and the
lower bound is supported by sampling the Minkowski flex (Remark~\ref{rem:231}).
\end{enumerate}

\paragraph{Which drawing represents the graph, and a contrast between the two invariants.}
For the \emph{ordinary} crossing number the crossing-optimal drawings are structurally opaque: the
free drawing of Appendix~\ref{app:free} and the two-page book both scatter the vertices to save
crossings and show nothing of $H(d,3)=H(d-1,3)\bx K_3$. The drawing that \emph{is} a good
representation, the recursive Minkowski-sum drawing, three copies joined by triangle bundles
(\S\ref{ss:recdraw}), pays for that fidelity with more crossings. So for $\cro$ there is a genuine
gap between the \emph{faithful} drawing and the \emph{optimal} one. The unit invariant behaves
oppositely. By the structure theorem~(T1) every faithful unit realization is a Minkowski sum, and the
$\udcr$-optimum is attained \emph{inside} that same family: for $d=3$ every general-position count we
sampled over the flex is $231$, $237$ or $243$, so the best drawing found ($231$) differs from the other faithful
representations only by a choice of twist angles, not in kind. Imposing unit length collapses the
gap: unlike the general crossing number, the $\udcr$-optimal drawing \emph{is} a faithful
representation of the graph; the best drawing and a good drawing are essentially the same object.

\medskip
The contrast is structural: a \emph{rigid} unit-distance graph pins its
coordinate field (an isolated realization), whereas a \emph{flexible} one such
as $H(3,3)$ has a positive-dimensional realization variety and is realizable
across the whole constructibility hierarchy (the same $81$ edges over
different coordinate fields) and the same is true in space: the
spatial realizations are compass ($\Q(\sqrt3)$ from the coordinate planes,
$\Q(\sqrt2,\sqrt3)$ from the faithful tetrahedron), but a $20^\circ$ or
$\tfrac{2\pi}{11}$ turn of one triangle reaches the origami and exotic tiers in
$\R^3$ too. By contrast, among the
uniform polyhedra the coordinate fields climb only as far as origami: the icosahedron and dodecahedron give $\Q(\sqrt5)$ (degree $2$), the snub cube the
tribonacci cubic $t^3-t^2-t-1$ (degree $3$), and the snub dodecahedron a
degree-$6$ field, so no single Archimedean solid reaches the exotic tier; it
is the flexibility of $H(3,3)$, not a rigid solid, that gets there.

Two points stand out. First, $H(3,3)$ is generically rigid, yet its
unit-distance realizations all sit on a two-dimensional flex; that flex is what lets
one graph occupy compass, origami, and exotic fields at once. Second, the two
questions one asks of a realization are of different kinds:
\emph{faithfulness} is an open geometric condition, certified over $\R$ with no
algebraic extension, while the \emph{constructibility tier} is an
arithmetic invariant of the coordinate field, certified only inside that extension.
The tier is defined only at the countably many algebraic angles; elsewhere a
faithful realization has no tier. The unit-distance
geometry and the Galois arithmetic are two independent certificates of the same
graph.

\paragraph{Relation to complexity, and to the coloring literature.}
Deciding whether an abstract graph is a unit-distance graph in $\R^2$ is
$\exists\R$-complete~\cite{Schaefer}, with $\mathrm{NP}\subseteq\exists\R\subseteq
\mathrm{PSPACE}$ (the upper bound is Canny's~\cite{Canny}); with integer coordinates the
problem is NP-complete~\cite{IntegerUDG}. The family $H(d,3)$ is not a hard instance: we
realize it explicitly. Nor do our results bear on the coloring theory of unit-distance graphs
(the Hadwiger--Nelson problem~\cite{Soifer}), since the chromatic number does not depend on the
realizing field.

\paragraph{Reproducibility.} All realizations, edge counts, ghost counts, degree
and irreducibility statements, and every figure in this paper are produced and
verified in exact arithmetic by the accompanying program \texttt{generate\_h33.py}
(SymPy for exact algebra, NumPy/Matplotlib for the figures).

\paragraph{Formal verification.} The structural and algebraic results of this paper are
formally verified in the Lean~4 proof assistant with the Mathlib library, in about $3{,}400$
lines with no \texttt{sorry} and no axioms beyond Lean's standard ones:
\url{https://github.com/haroldosfilho/HammingUDG}. The verified statements are: $H(d,3)$ is
$2d$-regular and every edge of the Minkowski sum is unit (Proposition~\ref{prop:mink}); the
Minkowski realization is faithful in the plane for every $d$; the three faithful realizations
of $H(3,3)$ with coordinate fields of degrees $4$, $6$ and $5$; $\operatorname{edim}(H(d,q))=q-1$
(Theorem~\ref{thm:edim}); the flex Theorem~\ref{prop:flex} in the strengthened form above, valid for
every orientation of the triangles; and, for the number $X(d)$ of crossings of the
lexicographic one-page drawing, the recurrence, the closed form
$X(d)=\bigl[7\cdot9^{d}-(6d^{2}+8d+7)\,3^{d}\bigr]/6$ (new) and the bound $X(d)<\tfrac76n^{2}$.
For the unit-distance crossing number, the formalization defines $\udcr$ for good straight-line
unit drawings, proves that a valid faithful unit drawing exists for every $d$, and checks in exact
arithmetic over $\mathbb{Z}[\sqrt3]$ explicit good drawings with $9$ and $231$ crossings, so
$\udcr(K_3\bx K_3)\le9$ and $\udcr(H(3,3))\le231$; these drawings use rotations with rational cosine
and sine ($\cos=\tfrac45$; and $\cos=\tfrac{24}{25},\tfrac45$), near the configurations of the text. The asymptotic crossing-number bounds and the exact lower bounds are not formalized, since
Mathlib has no theory of graph drawings yet.

\section{Future work}
The family $H(d,3)$ leaves the following concrete questions for future work.
\begin{enumerate}[label=\textup{(P\arabic*)}]
\item \emph{The unit lower bound.} The ordinary crossing number is settled,
$\cro(H)=\Theta(n^{2})$ (Theorems~\ref{thm:upper},~\ref{thm:tight}), and the unit count is
now bracketed $\Omega(n^{2})\le\udcr(H)\le O(n^{2}\log n)$ (Theorem~\ref{thm:udcr}), the upper
bound matching the Minkowski count's order. Only one step remains for
$\udcr(H)=\Theta(n^{2}\log n)$ and the tight penalty $\udcr/\cro=\Theta(\log n)$: a matching
lower bound $\udcr(H)=\Omega(n^{2}\log n)$, i.e.\ that \emph{every} faithful unit realization
is forced into $\Omega(n^{2}\log n)$ crossings. This cannot come from bisection, since $\udcr$
is geometric rather than a graph invariant; it is the one open point.
\item \emph{The explicit constant.} Theorem~\ref{thm:upper} gives the sharp one-page value
$\cro(H)\le\tfrac76 n^{2}$. A \emph{separated} (non-book) drawing does better numerically,
approaching $\tfrac12 n(n-3)=\binom{n}{2}-n$ (verified for $n\le729$), but the margin is narrow:
proving it by the recursion needs a uniform per-level bundle bound $B(n)\le n^{2}/3$, and the
computed bundle term sits right at that threshold (its ratio to $n^2$ climbs through $1/3$ near
$d=6$), so a generic separated drawing already exceeds $\tfrac12 n(n-3)$. Whether some drawing
attains $\tfrac12 n(n-3)$ for all $n$ (equivalently, whether $\cro(H)$ has leading constant
$\tfrac12$ rather than the $\tfrac76$ proved here) is open; it will not come from this
recursion alone. The \emph{exact} crossing number $\cro(H(d,3))$, even its leading constant, is a
hard target of independent interest; what the present recursion settles is only the order
$\Theta(n^{2})$.
\item \emph{The order of $g(n)$.} The recursion $c(n)=3\,c(n/3)+g(n)$ holds for each
drawing; a proof that the bundle sum $g(n)=\Theta(n^{2}\log n)$ (equivalently that the
copy--copy and copy--bundle sums $X,Z$ each carry one logarithmic factor) would settle the
$\udcr$ order. Constant factors are irrelevant here; only the exponent and the $\log$ matter.
\item \emph{Minimizing over the flex.} $\udcr(H)$ is the minimum of the crossing count over
the flex and all faithful realizations; is the minimum attained at a symmetric twist, and
does its order equal that of the Minkowski count $c(n)$?
\item \emph{Higher $q$.} For $q\ge4$, $\operatorname{edim}(H(d,q))=q-1\ge3$: study the
\emph{spatial} unit-distance crossing number and the analogous recursion in $\R^{q-1}$.
\end{enumerate}

\paragraph{Acknowledgments.} This study was financed in part by the Coordena\c c\~ao de
Aperfei\c coamento de Pessoal de N\'ivel Superior -- Brasil (CAPES) -- Finance Code 001. The
first author thanks the Programa de P\'os-Gradua\c c\~ao em Ci\^encias Computacionais e
Modelagem Matem\'atica (PGCComp) of the Universidade do Estado do Rio de Janeiro (UERJ) for
its support.

\appendix
\section{Reproducibility: algorithms}
Every numerical claim in this paper is produced by short, self-contained procedures. We
record them in pseudocode; the running implementations (Python/NumPy) and all vector
figures are available from the authors.

\paragraph{A.1 The Minkowski realization and faithfulness.}
\begin{verbatim}
input  d, twist angles theta[0..d-1]
T[i][k] = R*( cos(theta[i]+2*pi*k/3), sin(theta[i]+2*pi*k/3) )
          for k=0,1,2   # unit triangle i, R=1/sqrt(3)
for each vertex v in {0,1,2}^d:
    P[v] = sum_i T[i][ v[i] ]                      # Minkowski sum
assert |P[u]-P[w]| = 1 for every edge (one coord differs)   # Prop 2
faithful = ( all vertices distinct  and
             no non-adjacent pair has |P[u]-P[w]|=1 )   # ghost
\end{verbatim}

\paragraph{A.2 The flex dimension.}
\begin{verbatim}
build rigidity matrix M (rows = edges, columns = 2*|V|):
    row(u,w): +[P[u]-P[w]] at u, -[P[u]-P[w]] at w
rank = numeric rank(M)
flex = (2*|V| - rank) - 3          # subtract the 3 rigid motions
# returns d-1 for H(d,3): Proposition (flex)
\end{verbatim}

\paragraph{A.3 The four-type crossing decomposition.}
\begin{verbatim}
for each unordered pair of independent edges e, f:
    if segments P(e), P(f) cross:
        classify e,f as copy/bundle (bundle = changes last coordinate)
        if both copy, same copy k :  intra += 1
        if both copy, different k  :  X += 1      # copy-copy
        if both bundle             :  Y += 1      # bundle-bundle
        else                       :  Z += 1      # copy-bundle
assert intra == 3*c(d-1)  and  c(d) == intra + X + Y + Z    # eq. (decomp)
\end{verbatim}

\paragraph{A.4 The recursive upper bounds.}
The unit count $c(d)$ (Table in \S\ref{ss:decomp}) is obtained by A.3 on the Minkowski
drawing; the ordinary bound $c_R(d)$ by the same crossing count on the separated
drawing (three shrunken disjoint copies at the corners of a triangle, radial bundles),
minimized over the layout parameters $(\rho,\sigma,\varphi)$. The lower bound
$3^{d-1}=n/3$ needs no computation (vertex-disjoint copies).

\section{A computer-free proof that the orientations are uniform}\label{app:uniform}
This appendix replaces, by an explicit finite argument, the exhaustive check used in the
``orientation is uniform'' step of Proposition~\ref{prop:base9}; no enumeration over all $2^{6}$
patterns is needed.

\paragraph{Setup.} Let $z_{a,b}=p(a,b)\in\mathbb C$ ($a,b\in\{0,1,2\}$), $\omega=e^{2\pi i/3}$.
Each row $(z_{0,b},z_{1,b},z_{2,b})$ and each column $(z_{a,0},z_{a,1},z_{a,2})$ is a nondegenerate
unit equilateral triangle, with orientation $\varepsilon_b,\eta_a\in\{1,2\}$ meaning
$\sum_a\omega^{\varepsilon_b a}z_{a,b}=0$ and $\sum_b\omega^{\eta_a b}z_{a,b}=0$. Writing the DFT in
the first index, $z_{a,b}=g_b+c_b\omega^{a}+d_b\omega^{2a}$; then
$\sum_a\omega^{a}z_{a,b}=3d_b$ and $\sum_a\omega^{2a}z_{a,b}=3c_b$, so orientation $1$ means $d_b=0$
($z_{a,b}=g_b+c_b\omega^{a}$) and orientation $2$ means $c_b=0$ ($z_{a,b}=g_b+d_b\omega^{2a}$);
nondegeneracy means the surviving coefficient is nonzero.

\paragraph{The principle.} For a \emph{fixed} orientation pattern $(\varepsilon,\eta)$ the six
conditions are linear, so their solutions form a linear subspace
$K_{\varepsilon,\eta}\subseteq\mathbb C^{9}$. An injective realization is a point of
$K_{\varepsilon,\eta}$ off every coincidence hyperplane $\{z_{a,b}=z_{a',b'}\}$; and a linear
subspace contained in a finite union of hyperplanes lies in one of them. So it suffices to show that
each \emph{non-uniform} pattern has $K_{\varepsilon,\eta}$ inside a coincidence hyperplane.

\paragraph{Reduction by symmetry.} The group generated by permuting the three rows, permuting the
three columns, transposing the array, and the simultaneous flip
$\varepsilon,\eta\mapsto 3-\varepsilon,3-\eta$ acts on the patterns. Up to it, a non-uniform row
pattern is $\varepsilon=(1,1,2)$: rows $0,1$ are $z_{a,b}=g_b+c_b\omega^{a}$ ($c_0,c_1\neq0$) and
row $2$ is $z_{a,2}=g_2+d_2\omega^{2a}$ ($d_2\neq0$). By transposition the same argument then settles
the columns, so it is enough to treat $\varepsilon=(1,1,2)$ and every $\eta$.

\paragraph{Equally oriented columns are impossible.} Suppose $\eta\equiv1$. For each $a$,
\[
  z_{a,0}+\omega z_{a,1}+\omega^{2}z_{a,2}
  =\underbrace{(g_0+\omega g_1+\omega^{2}g_2)}_{\text{const.}}
   +\underbrace{(c_0+\omega c_1)}_{}\,\omega^{a}
   +\underbrace{\omega^{2}d_2}_{}\,\omega^{2a}=0 .
\]
As $a$ runs over $\{0,1,2\}$ the functions $1,\omega^{a},\omega^{2a}$ are linearly independent, so
the three equations force each coefficient to vanish; in particular $\omega^{2}d_2=0$, i.e.\
$d_2=0$. Then $z_{a,2}=g_2$ for all $a$: the third row collapses to a point, contradicting that it is
a unit triangle. The case $\eta\equiv2$ is identical with $\omega\leftrightarrow\omega^{2}$. Hence in
a faithful realization the column orientations are \emph{not} all equal.

\paragraph{Mixed columns force a coincidence.} A mixed $\eta$ (two equal, one different) is, up to a
cyclic permutation of the columns, either $(1,1,2)$ or $(1,2,2)$. In each case the three column
equations are three linear relations in $g_0,g_1,g_2,c_0,c_1,d_2$; eliminating the $g$'s and $c$'s
leaves a single linear dependence among the $z_{a,b}$. For $\eta=(1,1,2)$ it is
$z_{0,1}=z_{1,0}$, and for $\eta=(1,2,2)$ it is $z_{0,0}=z_{2,2}$: a vertex coincidence either way.
(For completeness, the forced coincidence for each $\eta$ with $\varepsilon=(1,1,2)$ is
$111:z_{0,2}{=}z_{1,2}$, $112:z_{0,1}{=}z_{1,0}$, $121:z_{0,0}{=}z_{2,1}$, $122:z_{0,0}{=}z_{2,2}$,
$211:z_{0,0}{=}z_{1,2}$, $212:z_{0,1}{=}z_{2,0}$, $221:z_{0,0}{=}z_{1,1}$, $222:z_{0,2}{=}z_{1,2}$;
the two constant patterns $111,222$ are the collapse above.)

\paragraph{Conclusion.} Every non-uniform pattern forces $z_{a,b}=z_{a',b'}$ for some
$(a,b)\neq(a',b')$, so no faithful realization has one. Thus the rows share a single orientation
and, by transposition, so do the columns: the input to the ``uniform $\Rightarrow$ Minkowski''
step of Proposition~\ref{prop:base9}, which is purely geometric. The whole base case is therefore
established without a computer. $\qed$

\section{A free drawing of $H(3,3)$ with $99$ crossings}\label{app:free}
The following $27$ points give the drawing of Figure~\ref{fig:h33free}, a free (non-unit)
straight-line drawing of $H(3,3)$ with $\mathrm{cr}_{\mathrm{rect}}=99$ crossings, so
$\cro(H(3,3))\le99$. It was found by local optimization of the vertex coordinates and is in general
position: no vertex lies on a non-incident edge and no two crossings coincide, both checked in exact
sign arithmetic, so the count is honest. Vertices are labelled $abc\in\{0,1,2\}^3$; coordinates to
three decimals.
\[
\renewcommand{\arraystretch}{1.1}
\begin{array}{llll}
000{:}(0.265,0.160) & 001{:}(-0.028,0.399) & 002{:}(0.309,0.541) & 010{:}(0.481,0.125)\\
011{:}(-0.656,0.857) & 012{:}(0.476,0.687) & 020{:}(0.718,-0.243) & 021{:}(-0.814,-0.291)\\
022{:}(1.000,-0.725) & 100{:}(0.073,0.005) & 101{:}(-0.264,0.191) & 102{:}(0.097,0.120)\\
110{:}(0.364,-0.155) & 111{:}(-0.608,-0.147) & 112{:}(0.616,-0.197) & 120{:}(0.135,-0.226)\\
121{:}(-0.744,-0.261) & 122{:}(0.607,-0.240) & 200{:}(0.070,0.087) & 201{:}(-0.006,0.054)\\
202{:}(0.062,0.085) & 210{:}(-0.284,-0.132) & 211{:}(-0.389,-0.088) & 212{:}(-0.369,0.002)\\
220{:}(-0.185,-0.221) & 221{:}(-0.435,-0.166) & 222{:}(-0.491,-0.224) &
\end{array}
\]
This $99$ improves the separated drawing's count for $d=3$; the exact $\cro(H(3,3))$ remains open.

\end{document}